\documentclass{daj}

\usepackage{amsmath, amsthm, amsfonts, amssymb} 
\usepackage{mathrsfs} 
\usepackage{enumerate} 
\usepackage{xcolor} 
\usepackage{bbm} 
\usepackage{hyperref} 
\usepackage{thmtools}
\usepackage{thm-restate}
\usepackage{cleveref}

\theoremstyle{theorem}
\newtheorem{thm}{Theorem}[section]
\newtheorem{lem}[thm]{Lemma}
\newtheorem{prop}[thm]{Proposition}
\newtheorem{cor}[thm]{Corollary}
\newtheorem{conj}[thm]{Conjecture}
\theoremstyle{definition}
\newtheorem{defn}[thm]{Definition}
\newtheorem{quest}[thm]{Question}
\newtheorem{eg}[thm]{Example}
\newtheorem{rem}[thm]{Remark}

\newtheorem{definition}[thm]{Definition}
\newtheorem{lemma}[thm]{Lemma}
\newtheorem{cla}[thm]{Claim}
\newcommand\ben{\begin{enumerate}}
	\newcommand\een{\end{enumerate}}
\newcommand\cale{\mathscr{E}}
\newcommand\be{\begin{equation}}
	\newcommand\ee{\end{equation}}
\newcommand{\prodo}[2]{\prod_{{#1}=1}^{{#2}}}

\newcommand{\N}{\mathbb{N}}
\newcommand{\Z}{\mathbb{Z}}
\newcommand{\Q}{\mathbb{Q}}
\newcommand{\R}{\mathbb{R}}
\newcommand{\C}{\mathbb{C}}
\newcommand{\T}{\mathbb{T}}

\newcommand{\F}{\mathbb{F}}

\renewcommand{\O}{\mathcal{O}}

\newcommand{\Prime}{\mathbb{P}}
\newcommand{\Adele}{\mathbb{A}}
\newcommand{\K}{\mathbb{K}}

\newcommand{\I}{\mathcal{I}}

\newcommand{\B}{\mathcal{B}}
\newcommand{\D}{\mathcal{D}}

\newcommand{\M}{\mathcal{M}}

\newcommand{\Hil}{\mathcal{H}}

\newcommand{\es}{\emptyset}

\newcommand{\eps}{\varepsilon}

\newcommand{\ind}{\mathbbm{1}}

\newcommand{\innprod}[2]{\left\langle #1, #2 \right\rangle}
\newcommand{\norm}[2]{\left\| #2 \right\|_{#1}}

\newcommand{\floor}[1]{\left\lfloor #1 \right\rfloor}

\renewcommand{\hat}{\widehat}
\renewcommand{\tilde}{\widetilde}

\newcommand{\E}[2]{\mathbb{E}\left[ {#1} \mid {#2} \right]}

\newcommand{\X}{\mathbf{X}}
\newcommand{\Y}{\mathbf{Y}}

\newcommand{\UClim}{\text{UC-}\lim}

\newcommand{\Hom}{\textrm{Hom}}

\renewcommand{\Re}[1]{\text{Re}\left( #1 \right)}

\dajAUTHORdetails{%
  title = {Multiple Recurrence and Large Intersections for Abelian Group Actions}, 
  author = {Ethan Ackelsberg, Vitaly Bergelson, and Andrew Best},
  plaintextauthor = {Ethan Ackelsberg, Vitaly Bergelson, Andrew Best},
  keywords = {multiple recurrence, characteristic factors, cocycles, group actions},
}   

\dajEDITORdetails{%
   year={2021},
   number={18},
   received={15 January 2021},   
   published={1 October 2021},  
   doi={10.19086/da.28877},       
}   

\begin{document}

\begin{frontmatter}[classification=text]


\author[ack]{Ethan Ackelsberg}
\author[ber]{Vitaly Bergelson}
\author[bes]{Andrew Best}

\begin{abstract}
The purpose of this paper is to study the phenomenon of large intersections in the framework of multiple recurrence for measure-preserving actions of countable abelian groups.
Among other things, we show:
\begin{enumerate}
	\item	If $G$ is a countable abelian group and $\varphi, \psi : G \to G$ are homomorphisms such that
	$\varphi(G)$, $\psi(G)$, and $(\psi - \varphi)(G)$ have finite index in $G$, then for every ergodic measure-preserving system $(X, \B, \mu, (T_g)_{g \in G})$, every set $A \in \B$, and every $\eps > 0$, the set
	$\{g \in G : \mu(A \cap T_{\varphi(g)}^{-1}A \cap T_{\psi(g)}^{-1}A) > \mu(A)^3 - \eps\}$
	is syndetic.
	\item	If $G$ is a countable abelian group
	and $r, s \in \Z$ are integers such that
	$rG$, $sG$, and $(r \pm s)G$ have finite index in $G$,
	then for every ergodic measure-preserving system $(X, \B, \mu, (T_g)_{g \in G})$,
	every set $A \in \B$, and every $\eps > 0$, the set
	$\{g \in G : \mu(A \cap T_{rg}^{-1}A \cap T_{sg}^{-1}A \cap T_{(r+s)g}^{-1}A) > \mu(A)^4 - \eps\}$
	is syndetic.
\end{enumerate}
In particular, these extend and generalize results from \cite{bhk} concerning $\Z$-actions and \cite{btz2} on $\F_p^{\infty}$-actions.
Using an ergodic version of the Furstenberg correspondence principle, we obtain new combinatorial applications.
We also discuss numerous examples shedding light on the necessity of the various hypotheses above.
Our results lead to a number of interesting questions and conjectures, formulated in the introduction and at the end of the paper.
\end{abstract}
\end{frontmatter}

\section{Introduction}

The purpose of this paper is to study the phenomenon of \emph{large intersections} in the framework of multiple recurrence for measure-preserving systems.
For context, let us juxtapose some classical and some more recent results for $\Z$-actions. First, a definition:
\begin{definition} Let $(G,+)$ be a countable abelian group. A subset $S \subset G$ is called \emph{syndetic} if the union of finitely many translates of $S$ covers $G$, i.e., if there exist $g_1, \ldots, g_k \in G$ such that $G = \bigcup_{i=1}^k (g_i + S)$, where by $g_i + S$ we understand the set $\{g_i + s : s \in S\}$.
\end{definition}
\begin{thm}[Khintchine's recurrence theorem \cite{khintchine}]\label{history 1} For any invertible probability measure-preserving system $(X,\B,\mu,T)$, any $\eps > 0$, and any $A \in \B$, the set
	\be \{ n \in \Z : \mu(A \cap T^n A) > \mu(A)^2 - \eps \}
	\ee is syndetic.
\end{thm}

This improves on the classical Poincar\'{e} recurrence theorem in two ways: it provides a lower bound on the size of the intersection and shows that the set of return times is large.
Notice that the bound $\mu(A)^2$ is optimal, since for mixing systems, $\mu(A \cap T^nA) \to \mu(A)^2$.
Colloquially, we say that Theorem~\ref{history 1} shows \emph{syndeticity of large intersections for single recurrence}.
A natural question to ask is whether similar improvements can be made for results about \emph{multiple recurrence}.
Recall the Furstenberg multiple recurrence theorem (also known as the ergodic Szemer\'{e}di theorem):
\begin{thm}[Furstenberg \cite{diag}] \label{thm: Szemeredi}
	For any invertible probability measure-preserving system $(X,\B,\mu,T)$, any $A \in \B$ with $\mu(A)>0$,
	and any positive integer $k \geq 1$,
	\begin{align} \label{eq: Furstenberg erg avg}
		\liminf_{N-M \to \infty}{\frac{1}{N-M} \sum_{n=M}^{N-1}{
				\mu\left( A \cap T^nA \cap \cdots \cap T^{kn}A \right)}} > 0.
	\end{align}
\end{thm}

An immediate consequence of the positivity of this limit is that there is a constant $c > 0$ such that the set
\begin{align}
	\left\{ n \in \Z : \mu \left( A \cap T^nA \cap \cdots \cap T^{kn}A \right) > c \right\}
\end{align}
is syndetic.
In trying to find the optimal such $c > 0$ for $k \ge 2$, a curious picture emerges:
\begin{thm}[{\cite[Theorems 1.2 and 1.3]{bhk}}]\label{history 2} ~
	\begin{enumerate}[(i)]
		\item For any ergodic invertible probability measure-preserving system $(X, \B, \mu,T)$, any $\eps > 0$, and any $A \in \B$, the set
		\be \{ n \in \Z : \mu(A \cap T^nA \cap T^{2n}A) > \mu(A)^3 - \eps \} \ee is syndetic.
		\item For any ergodic invertible probability measure-preserving system $(X, \B, \mu,T)$, any $\eps > 0$, and any $A \in \B$, the set \be \{n \in \Z : \mu(A \cap T^n A \cap T^{2n}A \cap T^{3n}A) > \mu(A)^4 - \eps\} \ee is syndetic.
		\item There exists an ergodic system $(X, \B, \mu, T)$ with the following property: for any integer $\ell \geq 1$, there is a set $A =A(\ell) \in \B$ of positive measure such that
		\be
		\mu(A \cap T^nA \cap T^{2n} A \cap T^{3n}A\cap T^{4n}A ) \leq \frac 12\mu(A)^\ell
		\ee
		for every integer $n \neq 0$.
	\end{enumerate}
\end{thm}
\noindent In other words, parts (i) and (ii) of Theorem \ref{history 2} show syndeticity of large returns for double (resp. triple) recurrence, and part (iii) shows that a natural generalization of parts (i) and (ii) for longer expressions cannot proceed.

The results of parts (i) and (ii) were subsequently generalized by Frantzikinakis,\footnote{We do not give a full statement of \cite[Theorem C]{frathree}, as it deals with various polynomial configurations that are beyond the scope of this paper. The linear patterns that appear here come as a special case.} and part (iii) was generalized by Donoso, Le, Moreira, and Sun:\footnote{Part (iii) of Theorem \ref{thm: fra} also follows from our more general considerations in Section \ref{sec: quadruple} (see Corollary \ref{cor: Z quadruple}).}
\begin{thm}[\cite{frathree}, special case of Theorem C; \cite{dlms}, Theorem 1.5] \label{thm: fra} ~
	\begin{enumerate}[(i)]
		\item	Let $a, b \in \Z$ be nonzero and distinct.
		For any ergodic invertible probability measure-preserving system
		$(X, \B, \mu,T)$, any $\eps > 0$, and any $A \in \B$, the set
		\be \{ n \in \Z : \mu( A \cap T^{an}A \cap T^{bn}A) > \mu(A)^3 - \eps \} \ee is syndetic.
		\item Let $a, b, c \in \Z$ be nonzero and distinct such that $(0,a,b,c)$ forms a parallelogram
		in the sense that $0+c = a+b$.
		For any ergodic invertible probability measure-preserving system
		$(X, \B, \mu,T)$, any $\eps > 0$, and any $A \in \B$, the set
		\be \{ n \in \Z : \mu( A \cap T^{an}A \cap T^{bn}A \cap T^{cn}A) > \mu(A)^4 - \eps \} \ee
		is syndetic.
		\item	Let $a, b, c, d \in \Z$ be nonzero and distinct.
		There exists an ergodic invertible probability measure-preserving system $(X, \B, \mu, T)$ such that,
		for any $\ell \ge 1$, there is a set $A = A(\ell) \in \B$ of positive measure such that
		\be
		\mu(A \cap T^{an}A \cap T^{bn} A \cap T^{cn}A\cap T^{dn}A ) \leq \frac 12\mu(A)^\ell
		\ee
		for every integer $n \neq 0$.
	\end{enumerate}
\end{thm}

An astute reader may observe at this point that Theorem \ref{history 2}, though stated only for ergodic systems, must hold for certain non-ergodic systems.
Indeed, by Theorem \ref{thm: fra}, the set
$\{n \in \Z : \mu(A \cap T^{2n}A \cap T^{4n}A) > \mu(A)^3 - \eps\}$ is syndetic for ergodic $T$.
But this is the same as the set $\{n \in \Z : \mu(A \cap S^nA \cap S^{2n}A) > \mu(A)^3 - \eps\}$ for the transformation $S = T^2$, and $S$ need not be ergodic.
A similar argument shows that any power of an ergodic system still satisfies the conclusion of Theorem \ref{history 2}.
The essential fact about these systems is that their ergodic decompositions have only finitely many ergodic components.
For non-ergodic systems with more complicated ergodic decomposition, the conclusion of Theorem \ref{history 2}(i) may fail:

\begin{thm}[\cite{bhk}, Theorem 2.1]\label{history 2 needs ergodicity} There exists a non-ergodic system $(X, \B, \mu, T)$ with the following property: for any integer $\ell \geq 1$, there is a set $A = A(\ell) \in \B$ of positive measure such that
	\be
	\mu(A \cap T^nA \cap T^{2n}A) \leq \frac 12 \mu(A)^\ell
	\ee
	for every integer $n \neq 0$.
\end{thm}

\bigskip

It is reasonable to inquire whether the above results have a version for actions of other groups. As a reminder, if $G = (G,+)$ is a countable abelian group which acts on a probability space $(X,\B,\mu)$ by measure-preserving automorphisms $(T_g)_{g \in G}$, then we refer to the quadruple $\X = (X,\B,\mu,(T_g)_{g\in G})$ as a measure-preserving $G$-system, or $G$-system for short. Along these lines, some results were obtained in \cite{btz2} for measure-preserving actions of the additive group $\mathbb{F}_p^\infty = \bigoplus_{n=1}^{\infty}{\F_p}$, the direct sum of countably many copies of a finite field
of prime order $p$, which has a natural vector space structure over $\mathbb{F}_p$.

\begin{thm}[{\cite[Theorems 1.12 and 1.13]{btz2}}]\label{history 3} ~
	\begin{enumerate}[(i)]
		\item Fix a prime $p > 2$ and distinct elements $c_0, c_1, c_2 \in \mathbb{F}_p$. For any ergodic measure-preserving $\mathbb{F}_p^\infty$-system $(X, \B, \mu, (T_g)_{g\in \mathbb{F}_p^\infty} )$, any $\eps > 0$, and any $A \in \B$, the set \be \{ g \in \mathbb{F}_p^\infty : \mu(T_{c_0g} A \cap T_{c_1g}A \cap T_{c_2g}A) > \mu(A)^3 - \eps \} \ee is syndetic.
		\item Fix a prime $p > 3$ and distinct elements $c_0, c_1, c_2, c_3 \in \mathbb{F}_p$ which form a parallelogram in the sense that $c_i+c_j=c_k+c_\ell$ for some permutation $\{i,j,k,\ell\}$ of $\{0,1,2,3\}$. For any ergodic measure-preserving $\mathbb{F}_p^\infty$-system $(X, \B, \mu, (T_g)_{g\in \mathbb{F}_p^\infty} )$, any $\eps > 0$, and any $A \in \B$, the set \be \{ g \in \mathbb{F}_p^\infty : \mu(T_{c_0g} A \cap T_{c_1g}A \cap T_{c_2g}A \cap T_{c_3g}A) > \mu(A)^4 - \eps \} \ee is syndetic.
	\end{enumerate}
\end{thm}
\begin{rem} By adapting from \cite{bhk} the idea of the proof of Theorem~\ref{history 2 needs ergodicity}, one can find that ergodicity is again a necessary assumption for Theorem~\ref{history 3}. This is done in Proposition~\ref{prop: ergodicity large p}.
\end{rem}

In this article, we show that Theorems~\ref{history 2},~\ref{history 2 needs ergodicity},~and~\ref{history 3} are instances of a more comprehensive phenomenon which pertains to actions of general countable abelian groups. The most general framework for our results is the following.
We let $G$ be a countable discrete abelian group and consider homomorphisms $\varphi_1, \dots, \varphi_k : G \to G$.
Our goal is to understand under what conditions the set
\begin{align} \label{eq: large intersection hom}
	\left\{ g \in G :
	\mu \left( A \cap T_{\varphi_1(g)}A \cap \cdots \cap T_{\varphi_k(g)}A \right) > \mu(A)^{k+1} - \eps \right\}
\end{align}
is syndetic for all ergodic $G$-systems $\left( X, \B, \mu, (T_g)_{g \in G} \right)$, sets $A \in \B$ with $\mu(A) > 0$, and $\eps > 0$.
In this case, we say the family $\{\varphi_1, \dots, \varphi_k\}$ has the \emph{large intersections property}.

A natural approach to take when addressing this question is to study \emph{uniform Ces\`{a}ro averages}, a special case of which are the averages \eqref{eq: Furstenberg erg avg} appearing in Furstenberg's Multiple Recurrence Theorem.
To develop appropriate averaging schemes in general abelian groups, we need the notion of a F{\o}lner sequence.
A \emph{F{\o}lner sequence} in $G$ is a sequence of finite subsets $F_N \subseteq G$ such that, for every $g \in G$,
\begin{align}
	\frac{|(F_N + g) \triangle F_N|}{|F_N|} \to 0.
\end{align}
A sequence $(u_g)_{g \in G}$ has \emph{uniform Ces\`{a}ro limit} equal to $u$, denoted by $\UClim_{g \in G}{u_g} = u$, if for every F{\o}lner sequence $(F_N)_{N \in \N}$,
\begin{align}
	\lim_{N \to \infty}{\frac{1}{|F_N|} \sum_{g \in F_N}{u_g}} = u.
\end{align}
For $G = \Z$, the uniform Ces\`{a}ro limit agrees with the average appearing in Theorem \ref{thm: Szemeredi}
(see Lemma \ref{lem: UC uniform shifts}).
The utility of this averaging approach comes from the following fact,
which we prove in Section \ref{sec: Cesaro}:
\begin{restatable}{lem}{SyndeticThick} \label{lem: syndetic thick}
	A set $S \subseteq G$ is syndetic if and only if for every F{\o}lner sequence $(F_N)_{N \in \N}$ in $G$, one has
	$\bigcup_{N \in \N} ({F_N} \cap S) \ne \es$.
\end{restatable}

A standard technique for handling (multiple) ergodic averages is to work with \emph{characteristic factors}. A system $\Y = \left( Y, \D, \nu, (S_g)_{g \in G} \right)$ is a \emph{factor} of $\X = \left( X, \B, \mu, (T_g)_{g \in G} \right)$ if there are full measure subsets $X_0 \subseteq X$ and $Y_0 \subseteq Y$ and a measure-preserving map $\pi : X_0 \to Y_0$ such that $S_g\pi(x) = \pi(T_gx)$ for every $x \in X_0$, $g \in G$. There is a natural correspondence between factors and invariant sub-$\sigma$-algebras: $\Y$ corresponds to the $\sigma$-algebra $\pi^{-1}(\D) \subseteq \B$. In a standard abuse of notation, we will write $\E{f}{Y}$ to denote $\E{f}{\pi^{-1}(\D)}$ for $f \in L^2(\mu)$.
The factor $\Y$ is called a \emph{characteristic factor} for the family $\{\varphi_1, \dots, \varphi_k\}$ if for all $f_1, \dots, f_k \in L^{\infty}(\mu)$,
\begin{align}
	\UClim_{g \in G}{\left( \prod_{i=1}^k{T_{\varphi_i(g)}f_i}
		- \prod_{i=1}^k{T_{\varphi_i(g)}\E{f_i}{Y}} \right)} = 0
\end{align}
in $L^2(\mu)$.

While the proof of Theorem~\ref{history 2} in \cite{bhk} and the proof of Theorem~\ref{history 3} in \cite{btz2} both use explicit topological descriptions of characteristic factors for $\Z$-systems and $\mathbb{F}_p^\infty$-systems, respectively, the essential ingredients can be recovered with a softer approach that applies in greater generality. Namely, we adapt tools pioneered by Conze and Lesigne \cite{cl} and advanced in \cite{fw} and \cite{hk} to our setting. This approach provides just enough information about $G$-systems to produce large intersection results.

\bigskip

We begin by looking at the case $k = 2$ of \eqref{eq: large intersection hom}.
It is known that for any pair of homomorphisms $\{\varphi, \psi\}$, there is a constant $c > 0$ such that $\{g \in G : \mu(A \cap T_{\varphi(g)}^{-1}A \cap T_{\psi(g)}^{-1}A) > c\}$ is syndetic.
Indeed, this follows from the IP Szemer\'{e}di theorem of Furstenberg and Katznelson (see \cite[Theorem A]{fk-IP}). In fact, Roth's theorem holds for any countable group whatsoever (see \cite[Theorem 1.3]{bm-roth}).

Despite the wide generality in which the ergodic Roth theorem holds, we encounter an obstacle to its large intersections variant even in our significantly restricted setting: results of \cite{chu} can be utilized to show that not all pairs of homomorphisms have the large intersections property (see Example \ref{eg: Chu}).
However, imposing a simple condition on the homomorphisms (which we conjecture to be necessary) allows us to adapt the machinery of \cite{cl, fw} to recover large intersections.
Namely, we will assume that $\varphi, \psi : G \to G$ have the property that $\varphi(G)$, $\psi(G)$, and $(\varphi - \psi)(G)$ all have finite index in $G$. We call such pairs $\{\varphi, \psi\}$ \emph{admissible}.\footnote{Note that this encompasses the cases handled in \cite{bhk} and \cite{btz2}. Multiplication by an integer is a group homomorphism which has finite index image in $\Z$ (resp., $\F_p^\infty$) so long as the integer is nonzero (resp., nonzero modulo $p$).}
In this case, we will show that the \emph{Kronecker factor}\footnote{For ergodic systems, the \emph{Kronecker factor} is the maximal factor that can be described (up to isomorphism) as an action by rotations on a compact abelian group; see Section \ref{sec: factors} for further discussion.} $\mathbf{Z}$ is characteristic for certain ergodic averages: for any $f_1,f_2 \in L^\infty(\mu)$,
\be
\UClim_{g \in G}{\left( T_{\varphi(g)}f_1 \cdot T_{\psi(g)}f_2
	- T_{\varphi(g)}\E{f_1}{Z} \cdot T_{\psi(g)}\E{f_2}{Z} \right)} = 0
\ee
in $L^2(\mu)$.
Moreover, we actually derive a limit formula (still in norm) for expressions $\UClim_{g \in G} T_{\varphi(g)}f_1T_{\psi(g)}f_2$; see Theorem~\ref{thm: Kronecker} for a precise formulation. With the help of this limit formula, we will prove the following theorem.
\begin{restatable}{thm}{DoubleKhintchine} \label{thm: double Khintchine}
	Let $G$ be a countable discrete abelian group, and let $\{\varphi,\psi\}$ be an admissible pair of homomorphisms.
	For any ergodic system $\X = \left( X, \B, \mu, (T_g)_{g \in G} \right)$, any $\eps > 0$, and any $A \in \B$, the set
	\be
	\left\{ g \in G : \mu \left( A \cap T_{\varphi(g)}^{-1}A \cap T_{\psi(g)}^{-1}A \right)
	> \mu(A)^3 - \eps \right\}
	\ee
	is syndetic in $G$.
\end{restatable}
This theorem is a common generalization of Theorem \ref{history 2}(i), Theorem \ref{thm: fra}(i), and Theorem \ref{history 3}(i).
Once again, the assumption of ergodicity is crucial for producing large intersections, and we give a detailed discussion with many examples in Section \ref{sec: ergodicity}.

\bigskip

Now, let us turn to analyzing \eqref{eq: large intersection hom} for $k \ge 3$ with the intention of generalizing part (ii) of Theorem~\ref{history 2} and part (ii) of Theorem~\ref{history 3}.
We again impose a condition on the family of homomorphisms to avoid counterexamples along the lines of Example \ref{eg: Chu}.
Namely, we say a family of homomorphisms $\{\varphi_1, \varphi_2, \dots, \varphi_k\}$
is \emph{admissible} if $\varphi_i(G)$ has finite index in $G$ for every $1 \le i \le k$
and $(\varphi_j - \varphi_i)(G)$ has finite index in $G$ for every $1 \le i < j \le k$.
This condition appears in \cite[Chapter 4]{griesmer} in the study of multiple ergodic averages for $\Z^d$-systems, where it is also shown that this condition is necessary for obtaining certain results about characteristic factors. The issue of admissibility is addressed in more detail in Section \ref{sec: admissible}.

For these longer expressions, we turn to studying the $L^2$-limit of averages
\begin{equation} \label{intro eqn 2}
	\UClim_{g \in G}{T_{\varphi_1(g)}f_1 \cdot T_{\varphi_2(g)}f_2 \cdots T_{\varphi_k(g)}f_k}
\end{equation}
for $f_1,f_2, \dots, f_k \in L^\infty(\mu)$.
We will show that the characteristic factors for these averages are built inductively as towers of compact extensions, starting from the Kronecker factor (Theorem \ref{thm: compact ext}).

When $k = 3$, for admissible families of the form $\{\varphi, \psi, \theta\}$ with $\theta = \varphi + \psi$,
we can show that the characteristic factor is a skew-product of the Kronecker factor with a compact abelian group over a cocycle satisfying a \emph{Conze--Lesigne equation}.
This can be seen as a generalized form of the \emph{Conze--Lesigne factor}.
The relevant terminology will be made precise in Section \ref{sec: CL}.

With the help of a limit formula valid in the case that $\varphi$ and $\psi$ are multiplication by (necessarily distinct) integers (see Theorem~\ref{thm: limit formula}), we will also prove that admissible\footnote{In a slight abuse of notation, we say that a family of integers $\{r_1, \dots, r_k\}$ is admissible if the family of homomorphisms $\{\varphi_1, \dots, \varphi_k\}$ given by $\varphi_i(g) = r_i g$ is admissible.} triples $\{r,s,r+s\}$ have the large intersections property:\footnote{Theorem \ref{thm: triple Khintchine} was proved independently by a different method in a recent paper by Shalom (see \cite[Theorem 1.3]{shalom}).}

\begin{restatable}{thm}{TripleKhintchine}\label{thm: triple Khintchine}
	Let $G$ be a countable discrete abelian group.
	Let $r,s \in \Z$ be distinct nonzero integers such that
	$rG$, $sG$, and $(r \pm s)G$ have finite index in $G$.
	Then for any ergodic system $\X = \left( X, \B, \mu, (T_g)_{g \in G} \right)$,
	any $\eps > 0$, and any $A \in \B$, the set
	\be
	\left\{ g \in G : \mu \left( A \cap T_{rg}^{-1}A \cap T_{sg}^{-1}A \cap T_{(r+s)g}^{-1}A \right) > \mu(A)^4 - \eps \right\}
	\ee
	is syndetic in $G$.
\end{restatable}

Theorem \ref{thm: triple Khintchine} contains Theorem \ref{history 2}(ii), Theorem \ref{history 3}(ii), and Theorem \ref{thm: fra}(ii) as special cases.
However, it is still not fully satisfactory: it leaves unaddressed the question of what happens for intersections of the form $A \cap T_{\varphi(g)}^{-1}A \cap T_{\psi(g)}^{-1}A \cap T_{(\varphi+\psi)(g)}^{-1}A$ for homomorphisms not arising as multiplication by integers (and hence does not fully extend Theorem \ref{thm: double Khintchine}).
Recent results about related combinatorial patterns in finite abelian groups suggest that additional conditions on the family $\{\varphi, \psi, \varphi + \psi\}$, beyond admissibility, are needed to guarantee large intersections (see \cite{bsst}).

\begin{quest}
	What are necessary and sufficient conditions for a family of homomorphisms $\{\varphi,\psi,\theta\}$ to have the large intersections property, i.e., to satisfy the following:
	for any ergodic system $\X = \left( X, \B, \mu, (T_g)_{g \in G} \right)$, any $\eps > 0$, and any $A \in \B$, the set
	\be
	\left\{ g \in G : \mu \left( A \cap T_{\varphi(g)}^{-1}A \cap T_{\psi(g)}^{-1}A \cap T_{\theta(g)}^{-1}A \right) > \mu(A)^4 - \eps \right\}
	\ee
	is syndetic in $G$?
\end{quest}

\bigskip

Results about multiple recurrence in measure-preserving systems have combinatorial implications.
To formulate combinatorial consequences of the above results, we need an ergodic version of the Furstenberg correspondence principle.
First, a definition: the \emph{upper Banach density} of a set $E \subseteq G$ is the quantity
\begin{align}
	d^*(E) := \sup\left\{ \limsup_{N \to \infty}{\frac{|E \cap F_N|}{|F_N|}} :
	(F_N)_{N \in \N}~\text{is a F{\o}lner sequence in}~G \right\}.
\end{align}
\begin{thm}[Ergodic Furstenberg Correspondence Principle \cite{bfm}, Theorem 2.8] \label{thm: erg Furstenberg}
	Let $G$ be a countable discrete abelian group.
	Let $E \subseteq G$ with $d^*(E) > 0$.
	Then there exists an ergodic measure-preserving system $\X = \left( X, \B, \mu, (T_g)_{g \in G} \right)$
	and a set $A \in \B$ with $\mu(A) = d^*(E)$ such that,
	for every $k \in \N$ and every $g_1, \dots, g_k \in G$,
	\begin{align}
		d^* \left( \bigcap_{i=1}^k{(E - g_i)} \right) \ge \mu \left( \bigcap_{i=1}^k{T_{g_i}^{-1}A} \right).
	\end{align}
\end{thm}

A routine application of Theorem \ref{thm: erg Furstenberg} translates Theorems \ref{thm: double Khintchine} and \ref{thm: triple Khintchine} into the following two combinatorial results:\footnote{Rewriting the intersections, these theorems show an abundance of combinatorial configurations of the form $\{x, x + \varphi(g), x + \psi(g)\}$ and $\{x, x + rg, x + sg, x + (r+s)g\}$ respectively.}
\begin{thm} \label{thm: combinatorial dK}
	Let $G$ be a countable discrete abelian group.
	Let $\{\varphi, \psi\}$ be an admissible pair of homomorphisms.
	For any $E \subseteq G$ with $d^*(E) > 0$ and any $\eps > 0$, the set
	\begin{align}
		\left\{ g \in G : d^* \left( E \cap (E - \varphi(g)) \cap (E - \psi(g)) \right) > d^*(E)^3 - \eps \right\}
	\end{align}
	is syndetic.
\end{thm}
\begin{thm} \label{thm: combinatorial tK}
	Let $G$ be a countable discrete abelian group.
	Let $r, s \in \Z$ be distinct nonzero integers such that
	$rG$, $sG$, and $(r \pm s)G$ have finite index in $G$.
	For any $E \subseteq G$ with $d^*(E) > 0$ and any $\eps > 0$, the set
	\begin{align}
		\left\{ g \in G : d^* \left( E \cap (E - rg) \cap (E - sg) \cap (E - (r+s)g) \right) > d^*(E)^4 - \eps \right\}
	\end{align}
	is syndetic.
\end{thm}

Families of homomorphisms with the large intersections property produce an endless variety of combinatorial configurations.

A natural question to ask is whether combinatorial results about large intersections, such as Theorem \ref{thm: combinatorial dK} and Theorem \ref{thm: combinatorial tK}, have finitary versions:
\begin{quest} \label{quest: finitary}
	Fix a F{\o}lner sequence $(F_N)_{N \in \N}$ in $G$.
	Suppose that $\mathcal{F} = \{\varphi_1, \dots, \varphi_k\}$ has the large intersections property.
	Given $\delta > 0$, $\eps > 0$, does there exist $N_0 = N_0(\delta, \eps)$ such that
	if $N \ge N_0$ and $A \subseteq F_N$ has cardinality $|A| \ge \delta |F_N|$,
	then there exists $g \ne 0$ such that
	\begin{align}
		\left| A \cap (A - \varphi_1(g)) \cap \cdots \cap (A - \varphi_k(g)) \right|
		> \left( \delta^{k+1} - \eps \right) |F_N|?
	\end{align}
\end{quest}
We will now briefly discuss a few examples and reformulate Question \ref{quest: finitary} for these concrete situations.

When $G = \Z$ and $F_N = \{1, \dots, N\}$, Question \ref{quest: finitary} has a positive answer:\footnote{In \cite{green, gt}, these results are only stated for arithmetic progressions (corresponding to the case $r=1, s=2$).
	However, the same method extends to general triples or parallelogram configurations as stated here.}
\begin{thm}[\cite{green}, Theorem 1.10; \cite{gt}, Theorem 1.12] \label{thm: GT APs}
	Let $r, s \in \Z$ be distinct and nonzero.
	Suppose $\delta, \eps > 0$.
	There exists $N_0 = N_0(r, s, \delta, \eps) \in \N$ such that if $N \ge N_0(\delta, \eps)$
	and $A \subseteq \{1, \dots, N\}$ has size $|A| \ge \delta N$, then there exist $n, m \ne 0$ such that
	\begin{align}
		\left| A \cap (A - rn) \cap (A - sn) \right| & > \left( \delta^3 - \eps \right) N \\
		\intertext{and}
		\left| A \cap (A-rm) \cap (A-sm) \cap (A-(r+s)m) \right| & > \left( \delta^4 - \eps \right) N.
	\end{align}
\end{thm}

Now let us turn to $G = \Z^2$.
We will show in Example \ref{eg: Chu} that the pair of homomorphisms $\varphi(n,m) = (n,0)$, $\psi(n,m) = (0,n)$ fails to have the large intersections property.
The corresponding combinatorial configurations, $\{(a, b), (a+c, b), (a, b+c)\}$, are known as \emph{corners}.
Geometrically, we can view corners as isosceles right triangles in $\Z^2$ with legs parallel to the axes.
Sah, Sawhney, and Zhao have shown a finitary analogue of the fact that this pattern is not good for large intersections:
\begin{thm}[\cite{ssz}, Theorem 1.4]
	For any $l < 4$, there exists $\delta > 0$ such that for arbitrarily large $N$,
	there exists a set $A \subseteq \{1, \dots, N\}$ with $|A| > \delta N$ such that
	\begin{align*}
		\left| \left\{(a,b) \in \Z : \{(a,b), (a+c,b), (a,b+c)\} \subseteq A \right\} \right| \le \delta^l N
	\end{align*}
	for every $c \ne 0$.
\end{thm}

We now show that a slight variant of corner configurations is good for large intersections.
Consider the homomorphisms $\varphi(n,m) = (n,m)$ and $\psi(n,m) = (-m,n)$.
A configuration produced by the pair $\{\varphi,\psi\}$ has the form $\{(a,b), (a+n,b+m), (a-m,b+n)\}$.
This is an isosceles right triangle (with legs of length $\sqrt{n^2+m^2}$), but we now have an additional degree of freedom: varying the ratio $\frac{m}{n}$ changes the angle between the legs of the triangle and the coordinate axes.
When $m=0$, this reduces to a corner with sides of length $n$.
The pair $\{\varphi, \psi\}$ is admissible, so by Theorem \ref{thm: combinatorial dK},
there are syndetically many pairs $(n,m) \in \Z^2$ such that
\begin{align}
	d^*\left( \left\{(a,b) \in \Z^2 : \{(a,b), (a+n,b+m), (a-m,b+n)\} \subseteq E \right\} \right)
	> d^*(E)^3 - \eps.
\end{align}
Stated another way, out of all isosceles right triangle configurations in the two-dimensional integer lattice, syndetically many of them appear up to a shift with high frequency in $E$.
A finitary version of this result also holds, answering a question posed by us in an earlier version of this paper:
\begin{thm}[\cite{kovac}, Theorem 1; \cite{bsst}, Theorem 1.1]
	Let $\delta, \eps > 0$.
	There exists $N_0 = N_0(\delta, \eps) \in \N$ such that if $N \ge N_0(\delta, \eps)$
	and $A \subseteq \{1, \dots, N\}^2$ has size $|A| \ge \delta N^2$,
	then there exists $(n,m) \in \Z^2 \setminus \{(0,0)\}$ such that
	\begin{align}
		\left| A \cap \left( A - (n,m) \right) \cap \left( A - (-m,n) \right) \right| > (\delta^3 - \eps)N^2.
	\end{align}
\end{thm}

\bigskip

The structure of the paper is as follows.
After a brief discussion of key definitions and lemmas in Section \ref{sec: prelim}, the paper is broken into three main parts.
First, in Sections \ref{sec: Kronecker}--\ref{sec: limit formula 3}, we establish characteristic factors for families of admissible homomorphisms and prove limit formulae for families $\{\varphi,\psi\}$ and $\{r, s, r+s\}$.
This includes the bulk of the difficult technical arguments, combining ideas from \cite{cl, fw, ziegler, btz1} on characteristic factors and \cite{fw, hk} on computing multiple ergodic averages as integrals.
Then, in Sections \ref{sec: dK} and \ref{sec: tK}, we apply the limit formulae to prove Theorems \ref{thm: double Khintchine} and \ref{thm: triple Khintchine} using a strategy from \cite{frathree, btz2}.
The final four sections discuss each of the hypotheses that appear in Theorems \ref{thm: double Khintchine} and \ref{thm: triple Khintchine}: Section \ref{sec: admissible} deals with the issue of admissibility, Section \ref{sec: ergodicity} with ergodicity, Section \ref{sec: quadruple} with larger families of homomorphisms, and Section \ref{sec: parallelogram} with the parallelogram condition.



\section{Preliminaries} \label{sec: prelim}


\subsection{Standing assumptions}

Let $G = (G,+)$ be a countable abelian group which acts on a probability space $(X,\B,\mu)$ by measure-preserving automorphisms $(T_g)_{g \in G}$. As usual, we assume the space $(X,\B,\mu)$ is separable; that is, the $\sigma$-algebra $\B$ is countably generated modulo null sets. We refer to the quadruple $\X = (X,\B,\mu,(T_g)_{g\in G})$ as a measure-preserving $G$-system, or $G$-system.
Recall that a $G$-system $\X = (X,\B,\mu,(T_g)_{g\in G})$ is ergodic if every set $A \in \B$ satisfying $T_g^{-1}A = A$
for all $g \in G$ has $\mu(A) \in \{0,1\}$.

Suppose $\Y = \left( Y, \D, \nu, (S_g)_{g \in G} \right)$ is a factor of $\X$.
Recall that we denote by $\E{f}{Y}$ the conditional expectation $\E{f}{\pi^{-1}(\D)}$, where $\pi : \X \to \Y$ is the factor map.
We can also define the \emph{pullback map} $\pi^* : L^2(Y) \to L^2(X)$ by $\pi^*f = f \circ \pi$
and the \emph{pushforward map} $\pi_* : L^2(X) \to L^2(Y)$ as the adjoint of $\pi^*$.
In the case that the factor $\Y$ arises as a $(T_g)_{g \in G}$-invariant sub-$\sigma$-algebra of $\B$,
we have the equality $\pi_*f = \E{f}{Y}$.
In a standard abuse of notation, we will therefore also use $\E{f}{Y}$ to denote the pushforward $\pi_*f$.

It is well known that $\mu$ can be disintegrated with respect to $\Y$ into a family of nonnegative Borel probability measures $(\mu_y)_{y\in Y}$ on $X$ so that $\mu = \int_Y \mu_y d\nu(y)$. Note that $\E{f}{Y}(y) = \int f \ d\mu_y$ for a.e. $y \in Y$ and for every $f \in L^1(\mu)$ such that $f \in L^1(\mu_y)$ for a.e. $y$. See \cite[Theorem 5.8]{fursbook}. When $\Y$ is the factor associated to the sub-$\sigma$-algebra of $(T_g)_{g\in G}$-invariant sets, we obtain the ergodic decomposition, which, should we need to write it, will be written $\mu = \int \mu_x \ d\mu(x)$. See \cite[Theorem 3.22]{glasner} or \cite{vara}.

For $1\leq p \leq \infty$, we write $L^p(\mu)$ or $L^p(X)$ for the Lebesgue space $L^p(X,\B,\mu)$ of complex-valued functions with finite $p$-norm, where as usual, two functions in $L^p(\mu)$ are identified if they agree $\mu$-a.e. This identification makes $L^2(\mu)$ into a separable Hilbert space. We also denote by $L^0(\mu)$ the space of all measurable complex-valued functions up to equivalence $\mu$-a.e.


\subsection{Uniform Ces\`{a}ro limits} \label{sec: Cesaro}

We use uniform Ces\`{a}ro limits extensively in this paper and present several useful lemmas here.
All of these results are standard, but we present them here with proofs for the convenience of the reader.
First, we show that the uniform Ces\`{a}ro limit is the same as the uniform limit of shifted Ces\`{a}ro-type averages:
\begin{lem} \label{lem: UC uniform shifts}
	Let $(F_N)_{N \in \N}$ be a F{\o}lner sequence in $G$.
	For a sequence $(u_g)_{g \in G}$ in a Hilbert space $\Hil$,
	\begin{align*}
		\UClim_{g \in G}{u_g} = u
	\end{align*}
	if and only if
	\begin{align*}
		\frac{1}{|F_N|} \sum_{g \in F_N + h}{u_g} \to u
	\end{align*}
	uniformly in $h \in G$ as $N \to \infty$.
\end{lem}
\begin{proof}
	Suppose $\UClim_{g \in G}{u_g} = u$.
	For each $N \in \N$, choose $K_N \ge N$ and $h_N \in G$ such that
	\begin{align*}
		\left\| \frac{1}{|F_{K_N}|} \sum_{g \in F_{K_N} + h_N}{u_g} - u \right\|
		\ge \sup_{K \ge N}{\sup_{h \in G}{
				\left\| \frac{1}{|F_K|} \sum_{g \in F_K + h}{u_g} - u \right\|}} - \frac{1}{N}.
	\end{align*}
	Set $\Phi_N := F_{K_N} + h_N$ for $N \in \N$.
	Since $K_N \to \infty$ and $(F_N)_{N \in \N}$ is a F{\o}lner sequence,
	$(\Phi_N)_{N \in \N}$ is also a F{\o}lner sequence in $G$.
	Thus,
	\begin{align*}
		\limsup_{N \to \infty}{\sup_{h \in G}{\left\| \frac{1}{|F_N|} \sum_{g \in F_N + h}{u_g} - u \right\|}}
		& \le \limsup_{N \to \infty}{\left\| \frac{1}{|F_{K_N}|} \sum_{g \in F_{K_N} + h_N}{u_g} - u \right\|} \\
		& = \lim_{N \to \infty}{\left\| \frac{1}{|\Phi_N|} \sum_{g \in \Phi_N}{u_g} - u \right\|} = 0.
	\end{align*}
	
	Conversely, suppose $\frac{1}{|F_N|} \sum_{g \in F_N + h}{u_g}$
	converges to $u$ uniformly in $h \in G$ as $N \to \infty$.
	Define a sequence of functions $v_N : G \to \Hil$ by
	\begin{align*}
		v_N(h) := \frac{1}{|F_N|} \sum_{g \in F_N + h}{u_g}.
	\end{align*}
	By assumption, $\sup_{h \in G}{\|v_N(h) - u\|} \to 0$ as $N \to \infty$.
	On the other hand, since averages along F{\o}lner sequences are shift-invariant, we have
	$\UClim_{g \in G}{\left( v_N(g) - u_g \right)} = 0$ for each $N \in \N$.
	Hence, $\UClim_{g \in G}{u_g} = u$.
\end{proof}

Using Lemma \ref{lem: UC uniform shifts}, we can prove a version of the van der Corput trick for uniform Ces\`{a}ro limits.
\begin{lem}[van der Corput Trick] \label{lem: vdC}
	Let $G$ be a (countable discrete) abelian group, and
	let $(u_g)_{g \in G}$ be a bounded sequence in a Hilbert space $\Hil$.
	Suppose that for every $h \in G$,
	\begin{align*}
		\gamma_h := \UClim_{g \in G}{\innprod{u_{g+h}}{u_g}}
	\end{align*}
	exists and
	\begin{align*}
		\UClim_{h \in G}{\gamma_h} = 0.
	\end{align*}
	Then
	\begin{align*}
		\UClim_{g \in G}{u_g} = 0.
	\end{align*}
\end{lem}
\begin{proof}
	Without loss of generality, assume $\|u_g\| \le 1$ for every $g \in G$.
	Fix a F{\o}lner sequence $(F_N)_{N \in \N}$ in $G$, and let $\eps > 0$.
	Since $\UClim_{h \in G}{\gamma_h} = 0$, by Lemma \ref{lem: UC uniform shifts}, the sequence of functions
	\begin{align*}
		f_M(k) := \frac{1}{|F_M|} \sum_{h \in F_M-k}{\gamma_h}
	\end{align*}
	converges to $0$ uniformly in $k \in G$ as $M \to \infty$.
	Thus, we can find $M \in \N$ such that $|f_M(k)| < \eps$ for every $k \in G$.
	
	Now, since $(F_N)_{N \in \N}$ is a F{\o}lner sequence, there is an $N_1 \in \N$
	such that if $N \ge N_1$, then $F_N$ is $(F_M, \eps)$-invariant.\footnote{A set $F$ is called
		\emph{$(K,\eps)$-invariant} if $\frac{|(F+K) \triangle F|}{|F|} < \eps$.
		It is easy to check that $(F_N)_{N \in \N}$ is a F{\o}lner sequence if and only if
		for every finite set $K \subseteq G$ and every $\eps > 0$, there is an $N_0 \in \N$
		so that $F_N$ is $(K,\eps)$-invariant for every $N \ge N_0$.}
	In particular,
	\begin{align*}
		\left\| \frac{1}{|F_N|} \sum_{g \in F_N}{u_g}
		- \frac{1}{|F_M|}\sum_{h \in F_M}{\frac{1}{|F_N|} \sum_{g \in F_N}{u_{g+h}}} \right\| < \eps.
	\end{align*}
	
	Moreover, we can find $N_2 \in \N$ such that if $N \ge N_2$, then
	\begin{align*}
		\left| \frac{1}{|F_N|} \sum_{g \in F_N+k}{\innprod{u_{g+h}}{u_g}} - \gamma_h \right| < \eps
	\end{align*}
	for $h \in F_M - F_M$, $k \in G$.
	
	Set $N_0 := \max\{N_1, N_2\}$.
	Then for $N \ge N_0$, we have
	\begin{align*}
		\left\| \frac{1}{|F_N|} \sum_{g \in F_N}{u_g} \right\|^2
		& < \left\| \frac{1}{|F_M|}\sum_{h \in F_M}{
			\frac{1}{|F_N|} \sum_{g \in F_N}{u_{g+h}}} \right\|^2 + \eps(2+\eps) \\
		& \le \frac{1}{|F_N|} \sum_{g \in F_N}{
			\left\| \frac{1}{|F_M|} \sum_{h \in F_M}{u_{g+h}} \right\|^2} + \eps(2+\eps) \\
		& = \frac{1}{|F_N|} \sum_{g \in F_N}{
			\frac{1}{|F_M|^2} \sum_{h_1, h_2 \in F_M}{\innprod{u_{g+h_1}}{u_{g+h_2}}}} + \eps(2+\eps) \\
		& = \frac{1}{|F_M|^2} \sum_{h_1,h_2 \in F_M}{\frac{1}{|F_N|} \sum_{g \in F_N+h_2}{
				\innprod{u_{g+h_1-h_2}}{u_g}}} + \eps(2+\eps) \\
		& < \left| \frac{1}{|F_M|^2} \sum_{h_1, h_2 \in F_M}{\gamma_{h_1-h_2}} \right| + \eps(3 + \eps) \\
		& \le \frac{1}{|F_M|} \sum_{k \in F_M}{\left| \frac{1}{|F_M|} \sum_{h \in F_M-k}{\gamma_h} \right|}
		+ \eps(3 + \eps) \\
		& < \eps(4 + \eps).
	\end{align*}
\end{proof}

The following lemma shows that positivity of uniform Ces\`{a}ro limits implies syndeticity of return times.

\SyndeticThick*

\begin{proof}
	Suppose $S$ is syndetic, and let $(F_N)_{N \in \N}$ be a F{\o}lner sequence.
	We will prove a stronger statement:
	$\underline{d}_{(F_N)}(S) := \liminf_{N \to \infty}{\frac{|S \cap F_N|}{|F_N|}} > 0$.
	Let $K \subseteq G$ be a finite set such that $S + K = G$.
	By a standard averaging argument, for each $N \in \N$, there is a $k_N \in K$ such that
	\begin{align*}
		\frac{|S \cap (F_N - k_N)|}{|F_N|} = \frac{|(S+k_N) \cap F_N|}{|F_N|} \ge \frac{1}{|K|}.
	\end{align*}
	Now by the F{\o}lner property, we have
	\begin{align*}
		\limsup_{N \to \infty}{\max_{k \in K}{\frac{|(F_N - k) \triangle F_N|}{|F_N|}}} = 0.
	\end{align*}
	Thus,
	\begin{align*}
		\liminf_{N \to \infty}{\frac{|S \cap F_N|}{|F_N|}}
		& \ge \liminf_{N \to \infty}{\frac{|S \cap (F_N-k_N) \cap F_N|}{|F_N|}} \\
		& \ge \liminf_{N \to \infty}{\frac{|S \cap (F_N - k_N)|}{|F_N|}}
		- \limsup_{N \to \infty}{\frac{|(F_N-k_N) \setminus F_N|}{|F_N|}} \\
		& \ge \frac{1}{|K|}.
	\end{align*}
	
	Conversely, suppose $S$ is not syndetic.
	We want to construct a F{\o}lner sequence entirely in $G \setminus S$.
	Let $(F_N)_{N \in \N}$ be any F{\o}lner sequence in $G$.
	Since $S$ is not syndetic, we have $S - F_N \ne G$ for every $N \in \N$.
	Thus, we can find elements $g_N \in G$ such that $g_N \notin S - F_N$.
	Equivalently, $S \cap (F_N + g_N) = \es$.
	Now $\Phi_N := F_N + g_N$ is a F{\o}lner sequence, and $\bigcup_{N \in \N}{\Phi_N} \cap S = \es$.
\end{proof}


\subsection{Ergodic theorem}

In our setting, we will need a general form of von Neumann's mean ergodic theorem.
This result is a standard exercise using the invariant splitting of a Hilbert space (see, e.g., \cite[Theorem 4.15]{et/dp} for a short proof). We record it here for reference:

\begin{thm}[Ergodic theorem] \label{thm: ergodic theorem}
	Let $(G, +)$ be a countable discrete abelian group.
	Let $\X = \left( X, \B, \mu, (T_g)_{g \in G} \right)$,
	and let $\mathcal{I} \subseteq \B$ be the $\sigma$-algebra of $T$-invariant sets.
	Then for $f \in L^2(\mu)$,
	\begin{align}
		\UClim_{g \in G}{T_gf} = \E{f}{\mathcal{I}}.
	\end{align}
	In particular, if $\X$ is ergodic, then
	\begin{align}
		\UClim_{g \in G}{T_gf} = \int_X{f~d\mu}.
	\end{align}
\end{thm}

In the language of characteristic factors, the ergodic theorem says that the trivial factor is characteristic for the averages
\begin{align}
	\UClim_{g \in G}{T_g f}
\end{align}
in ergodic systems.
If we apply an admissible homomorphism $\varphi : G \to G$,
then the action $\left( T_{\varphi(g)} \right)_{g \in G}$ will have finitely many ergodic components,
so $\UClim_{g \in G}{T_{\varphi(g)}f}$ is a function with finitely many values
(it is the conditional expectation with respect to the finitely many ergodic components).
Another way of describing this behavior is to say that the characteristic factor for the average
\begin{align}
	\UClim_{g \in G}{T_{\varphi(g)} f}
\end{align}
is built from rotations on the finite group $G/\varphi(G)$.


\subsection{Kronecker factor} \label{sec: factors}

Fix a countable discrete abelian group $(G, +)$.
A system $\X = \left( X, \B, \mu, (T_g)_{g \in G} \right)$ is \emph{compact} if there is a compact abelian group $Y$ and an action by group rotations $S_gy = y + a_g$ such that $\X$ is isomorphic to $\Y = \left( Y, \D, \nu, (S_g)_{g \in G} \right)$, where $\D$ is the Borel $\sigma$-algebra on $Y$ and $\nu$ is the Haar measure on $Y$.
The \emph{Kronecker factor} of an ergodic system $\X$ is the maximal compact factor of $\X$.

In our setting, the Kronecker factor can be described quite concretely.
Consider the Jacobs--de Leeuw--Glicksberg decomposition into compact and weakly mixing functions:
\begin{align*}
	L^2(\mu) = L^2(\mu)_c \oplus L^2(\mu)_{wm},
\end{align*}
where
\begin{align*}
	L^2(\mu)_c & = \left\{ f \in L^2(\mu) : \overline{\{T_gf : g \in G\}}^{\|\cdot\|_2}~\text{is compact} \right\}, \\
	L^2(\mu)_{wm} & = \left\{ f \in L^2(\mu) : \UClim_{g \in G}{|\innprod{T_gf}{f}|} = 0 \right\}.
\end{align*}
The compact functions, $L^2(\mu)_c$, are spanned by an orthonormal basis of eigenfunctions,
say $(f_\lambda)_{\lambda \in \Lambda}$, where $\Lambda$ is a countable subset of $\hat{G}$,
and $f_{\lambda}$ has the corresponding ``eigenvalue'' $\lambda$,
i.e. $T_g f_\lambda = \lambda(g) f_\lambda$ for all $g \in G$ and $\lambda \in \Lambda$.
We can assume, by rescaling if necessary, that $|f_\lambda| =1$ and $f_{\lambda}f_{\mu} = f_{\lambda+\mu}$
for all $\lambda, \mu \in \Lambda$
(the case $G = \Z$ is shown in \cite[Chapter 3]{walters}, and this easily generalizes to our setting).

For $\lambda \in \Lambda$ and $\varphi : G \to G$,
we can define $\varphi \lambda := \lambda \circ \varphi$.
While this defines a homomorphism on $G$ (so that $\varphi \lambda \in \hat{G}$),
it is not necessarily the case that $\varphi \lambda \in \Lambda$.

Now let $Z = \hat{\Lambda}$, where $\Lambda$ is treated as a discrete group so that $Z$ is compact,
and define an action
\begin{align*}
	\left( S_g(z) \right)(\lambda) := \lambda(g) z(\lambda).
\end{align*}
Equivalently, letting $\hat{g}$ be the evaluation map $\hat{g}(\lambda) = \lambda(g)$
and writing $Z$ additively (since it is an abelian group), we have $S_g(z) = z + \hat{g}$.
It is easy to check that $L^2(Z)$ is spanned by the evaluation functions $\{e_\lambda : \lambda \in \Lambda\}$,
where $e_\lambda(z) := z(\lambda)$.
These functions are also the eigenfunctions for the action $(S_g)_{g \in G}$.
The map $f_\lambda \mapsto e_{\lambda}$ then induces the factor map (see \cite{walters} for details).

We review a few facts about group rotations that will be helpful for analyzing the behavior of the Kronecker factor.
These facts are essentially the same as in the special case of $\Z$-systems (c.f. \cite[Theorem 6.20]{walters}),
but we present them here with proofs in the general case for completeness.
Recall that a $G$-action by homeomorphisms $(T_g)_{g \in G}$ on a compact Hausdorff space $X$ is \emph{minimal}
if the only closed $T$-invariant subsets of $X$ are $\es$ and $X$.
Equivalently, every orbit $\{T_gx : g \in G\}$, $x \in X$, is dense in $X$.

\begin{lem} \label{lem: compact group rotations}
	Let $G$ be an abelian group and $X$ a compact abelian group.
	Suppose $G$ acts on $X$ by group rotations $S_gx = x + a_g$,
	where $g \mapsto a_g$ is a (continuous) homomorphism $G \to X$.
	Then the following are equivalent:
	\begin{enumerate}[(i)]
		\item	The orbit $\{a_g : g \in G\}$ is dense in $X$;
		\item	The action $(S_g)_{g \in G}$ is minimal;
		\item	The action $(S_g)_{g \in G}$ is uniquely ergodic,
		and the unique invariant measure is the Haar measure on $X$.
	\end{enumerate}
\end{lem}
\begin{proof}
	The implication (ii)$\implies$(i) is trivial.
	For (iii)$\implies$(ii), we note that any closed invariant set supports an invariant measure,
	and Haar measure always has full support.
	
	Now we will show (i)$\implies$(iii).
	Assume that $\{a_g : g \in G\}$ is dense in $X$.
	Let $\mu$ be a $G$-invariant measure on $X$, and let $f \in C(X)$.
	For every $g \in G$, invariance of $\mu$ gives the identity
	\begin{align*}
		\int_X{f(x + a_g)~d\mu(x)} = \int_X{f(x)~d\mu(x)}.
	\end{align*}
	Now, given $y \in X$, condition (i) implies that $y = \lim_{n \to \infty}{a_{g_n}}$
	for some sequence $(g_n)_{n \in \N}$ in $G$.
	Thus, by the dominated convergence theorem,
	\begin{align*}
		\int_X{f(x+y)~d\mu(x)} = \lim_{n \to \infty}{\int_X{f(x + a_{g_n})~d\mu(x)}} = \int_X{f(x)~d\mu(x)}.
	\end{align*}
	This invariance property uniquely defines the Haar measure, so $\mu$ is the Haar measure on $X$.
\end{proof}


\section{Kronecker factor is characteristic for double recurrence} \label{sec: Kronecker}

In this section, we show that the Kronecker factor is characteristic for the averages
\begin{align} \label{eq: double avg}
	\UClim_{g \in G}{T_{\varphi(g)}f_1 \cdot T_{\psi(g)}f_2}.
\end{align}
For ease of notation, we will write $\tilde{f}_i$ for the image of $f_i$ under the factor map $\pi : X \to Z$.
That is, $\E{f_i}{Z} = \tilde{f}_i \circ \pi$.

\begin{thm} \label{thm: Kronecker}
	Let $\X = (X, \B, \mu, (T_g)_{g \in G})$ be an ergodic measure-preserving system
	with Kronecker factor $\mathbf{Z}$ and factor map $\pi : \X \to \mathbf{Z}$.
	Let $\varphi, \psi : G \to G$ be homomorphisms
	such that $\varphi$, $\psi$, and $\psi - \varphi$ have finite index image in $G$.
	Then for $f_1, f_2 \in L^{\infty}(\mu)$, the limit
	\begin{align*}
		\UClim_{g \in G}{f_1(T_{\varphi(g)} x) f_2(T_{\psi(g)} x)}
	\end{align*}
	exists in $L^2(\mu)$ and is equal to
	\begin{align} \label{eq: ET}
		\UClim_{g \in G}{\tilde{f}_1(z + \hat{\varphi(g)}) \tilde{f}_2(z + \hat{\psi(g)})}
		= \int_{Z^2}{\tilde{f}_1(z+w_1) \tilde{f}_2(z+w_2)~d\nu_{\varphi,\psi}(w_1, w_2)},
	\end{align}
	where $z = \pi(x)$ and $\nu_{\varphi,\psi}$ is the Haar (probability) measure on the subgroup $Z_{\varphi,\psi} := \overline{\{(\hat{\varphi(g)}, \hat{\psi(g)}) : g \in G\}}$ of $Z^2$.
\end{thm}

\begin{rem}
	If the spectrum of $\X$ has the additional property that
	$\varphi\Lambda + \psi\Lambda \subseteq \Lambda$,
	then we can define $(\varphi w)(\lambda) := w(\varphi \lambda)$ for $w \in Z$ and similarly for $\psi$,
	in which case $Z_{\varphi,\psi} = \{(\varphi w, \psi w) : w \in Z\}$, so the limit in \eqref{eq: ET} simplifies to
	\begin{align*}
		\int_Z{\tilde{f}(z + \varphi w) \tilde{g}(z + \psi w)~dw}.
	\end{align*}
	A formula of this form appears in \cite{fw} for $\Z$-systems,
	where $\varphi$ and $\psi$ are multiplication by $a$ and $b$ respectively.
\end{rem}

\begin{proof}[Proof of Theorem \ref{thm: Kronecker}]
	First we will show that the Kronecker factor is characteristic.
	By linearity, this is the same as showing
	\begin{align} \label{eq: lim=0}
		\UClim_{g \in G}{T_{\varphi(g)}f_1 \cdot T_{\psi(g)}f_2} = 0
	\end{align}
	when either $\tilde{f}_1 = 0$ or $\tilde{f}_2 = 0$.
	Since the expressions are symmetric in $f_1$ and $f_2$ (by interchanging $\varphi$ and $\psi$),
	we will assume $\tilde{f}_1 = 0$.
	
	Let $u_g = T_{\varphi(g)}f_1 \cdot T_{\psi(g)}f_2$.
	We will use the van der Corput trick (Lemma \ref{lem: vdC}) to show \eqref{eq: lim=0}.
	Since $T_{\varphi(g)}$ is $\mu$-preserving, we have
	\begin{align*}
		\innprod{u_{g+h}}{u_g}
		= \int{\left(\overline{f}_1 T_{\varphi(h)}f_1\right) T_{(\psi-\varphi)(g)}
			\left( \overline{f}_2 T_{\psi(h)}f_2 \right)~d\mu}.
	\end{align*}
	Since $(\psi - \varphi)(G) \subseteq G$ has finite index,
	the action of this subgroup $(T_{(\psi - \varphi)(g)})_{g \in G}$
	has only finitely many components in its ergodic decomposition.
	In fact, the ergodic decomposition is of the form $\mu = \frac{1}{d} \sum_{j=1}^d{\mu_j}$,
	where $\mu_j$ is the (normalized) restriction of $\mu$ to an invariant set of measure $\frac{1}{d}$,
	and $d \le [G : (\psi - \varphi)(G)]$.
	Now for each $h \in G$, the ergodic theorem gives
	\begin{align*}
		\gamma_h := \UClim_{g \in G}{\innprod{u_{g+h}}{u_g}}
		= \frac{1}{d} \sum_{j=1}^d{\left( \int{\overline{f}_1 T_{\varphi(h)}f_1~d\mu_j}
			\int{\overline{f}_2 T_{\psi(h)}f_2~d\mu_j} \right)}.
	\end{align*}
	Applying the triangle inequality and the Cauchy--Schwarz inequality,
	we have for any finite set $F \subseteq G$,
	\begin{align*}
		\left| \frac{1}{|F|} \sum_{h \in F}{\gamma_h} \right|
		& \le \frac{1}{d} \sum_{j=1}^d{\frac{1}{|F|} \sum_{h \in F}{ \left|
				\int{\overline{f}_1 T_{\varphi(h)}f_1~d\mu_j \int{\overline{f}_2 T_{\psi(h)}f_2~d\mu_j}} \right|}} \\
		& \le \frac{1}{d} \sum_{j=1}^d{\left(
			\frac{1}{|F|} \sum_{h \in F}{\left| \int{\overline{f}_1 T_{\varphi(h)}f_1~d\mu_j} \right|^2}
			\frac{1}{|F|} \sum_{h \in F}{\left| \int{\overline{f}_2 T_{\psi(h)}f_2~d\mu_j} \right|^2}
			\right)^{1/2}}.
	\end{align*}
	We assumed $\tilde{f}_1 = 0$, so $f_1 \in L^2(\mu)_{wm}$.
	It follows that $f_1$ is a weakly mixing function for the subaction
	along the finite index subgroup\footnote{
		To see this, consider $(T_g)_{g \in G}$ restricted to $L^2(\mu)_{wm}$.
		We want to show that $(T_{\varphi(h)})_{h \in G}$ is a weakly mixing action on this space.
		If not, then there is a nonzero function $f \in L^2(\mu)_{wm}$
		such that $\{T_{\varphi(h)}f : h \in G\}$ is pre-compact.
		But $\varphi(G)$ has finite index, so every element of $G$ can be expressed as
		$g = \varphi(h) + k$ for some $h \in G$ and some $k$ belonging to a finite set $K$.
		Hence,
		$\{T_g f : g \in G\} = \bigcup_{k \in K}{T_k \left( \{T_{\varphi(h)} f : h \in G\} \right)}$
		is a finite union of pre-compact sets and therefore pre-compact.
		But $f \in L^2(\mu)_{wm}$, so this is a contradiction.
	}
	$\varphi(G)$, so for each $1 \le j \le d$,
	\begin{align*}
		\UClim_{h \in G}{\left| \int{\overline{f}_1 T_{\varphi(h)}f_1~d\mu_j} \right|^2} = 0.
	\end{align*}
	Since $\left| \int{\overline{f}_2 T_{\psi(h)}f_2~d\mu_j} \right|^2 \le d^2 \|f_2\|_2^4 < \infty$, we have
	\begin{align*}
		\UClim_{h \in G}{\UClim_{g \in G}{\innprod{u_{g+h}}{u_g}}} = 0,
	\end{align*}
	so that \eqref{eq: lim=0} holds by Lemma \ref{lem: vdC}.
	Thus, the Kronecker factor is characteristic:
	\begin{align*}
		\UClim_{g \in G}{f_1(T_{\varphi(g)} x) f_2(T_{\psi(g)} x)}
		= \UClim_{g \in G}{\tilde{f}_1(z + \hat{\varphi(g)}) \tilde{f}_2(z + \hat{\psi(g)})}.
	\end{align*}
	
	It remains to compute the limit as an integral.
	Consider the diagonal action $(S_g)_{g \in G}$ on $Z_{\varphi,\psi}$ given by
	$S_g(w_1, w_2) = (w_1 + \hat{\varphi(g)}, w_2 + \hat{\psi(g)})$.
	It is easy to check that $S_g(Z_{\varphi,\psi}) \subseteq Z_{\varphi,\psi}$,
	so this action is well-defined.
	By Lemma \ref{lem: compact group rotations}, $(S_g)_{g \in G}$ is uniquely ergodic
	with unique invariant measure $\nu_{\varphi,\psi}$.
	Define $f : Z_{\varphi,\psi} \to \C$ by
	\begin{align*}
		f(w_1, w_2) := \tilde{f}_1(z + w_1) \tilde{f}_2(z + w_2).
	\end{align*}
	If $\tilde{f}_1$ and $\tilde{f}_2$ are continuous, then $f$ is continuous and unique ergodicity gives
	\begin{align*}
		\UClim_{g \in G}{\tilde{f}_1(z + \hat{\varphi(g)}) \tilde{f}_2(z + \hat{\psi(g)})}
		& = \UClim_{g \in G}{f(S_g0)} \\
		& = \int_{Z_{\varphi,\psi}}{f(w)~dw} \\
		& = \int_{Z^2}{\tilde{f}_1(z+w_1) \tilde{f}_2(z+w_2)~d\nu_{\varphi,\psi}(w_1, w_2)}.
	\end{align*}
	In general, we can approximate $\tilde{f}_1$ and $\tilde{f}_2$ by continuous functions
	(for example, with finite linear combinations of characters) to get the desired convergence in $L^2(\mu)$.
\end{proof}


\section{Compact extensions} \label{sec: compact ext}

The first step in studying characteristic factors for triple (and longer) recurrence is to observe
that they are formed as towers of \emph{compact extensions}.

For $k \in \N$, let $\mathbf{Z}_k$ be the minimal factor of $\X$ that is characteristic for all $k$-element admissible families. 
That is, for every admissible family $\{\varphi_1, \dots, \varphi_k\}$
and every $f_1, \dots, f_k \in L^{\infty}(\mu)$,
\begin{equation} \label{eq: multiple average}
	\UClim_{g \in G}{\left( T_{\varphi_1(g)}f_1 \cdots T_{\varphi_k(g)}f_k -T_{\varphi_1(g)} \E{f_1}{Z_k} \cdots T_{\varphi_k(g)} \E{f_k}{Z_k}\right)} = 0.
\end{equation}
Our goal is to prove the following:
\begin{thm} \label{thm: compact ext}
	For all $k \in \N$, $\mathbf{Z}_{k+1}$ is a compact extension of $\mathbf{Z}_k$.
\end{thm}
For $G = \Z$, this appears in \cite{ziegler}, and our methodology here is similar.
There are two key ingredients in the proof.
First, in the process of using the van der Corput trick (Lemma \ref{lem: vdC}),
we will encounter averages of the form
\begin{equation} \label{eq: int avg}
	\UClim_{g \in G}{\int_X{T_{\varphi_1(g)}f_1
			\cdots T_{\varphi_{k+1}(g)}f_{k+1}~d\mu}}.
\end{equation}
Since $T_{\varphi_1(g)}$ is measure-preserving, this is equal to the average
\begin{equation*}
	\UClim_{g \in G}{\int_X{f_1 T_{(\varphi_2-\varphi_1)(g)}f_2 \cdots T_{(\varphi_{k+1}-\varphi_1)(g)}f_{k+1}~d\mu}}.
\end{equation*}
An important observation at this stage is that the family
$\{\varphi_2-\varphi_1, \dots, \varphi_{k+1}-\varphi_1\}$ is admissible.
This follows easily from the definition of an admissible family, and it means that the average \eqref{eq: int avg}
is controlled by the factor $\mathbf{Z}_k$.

Analyzing these averages involves a diagonal measure on $X^{k+1}$.
We show that this diagonal measure is sufficiently well behaved to conclude that
the projection onto the invariant $\sigma$-algebra for the diagonal action is built from functions
on the maximal compact extension of $\mathbf{Z}_k$.
It will then follow that the maximal compact extension of $\mathbf{Z}_k$
is a characteristic factor for the averages \eqref{eq: multiple average}, completing the proof.

First, we define our terms.

\begin{definition} Suppose $\X = (X, \B,\mu,(T_g)_{g\in G})$ is a measure-preserving system with an ergodic factor $\Y = (Y,\D,\nu,(S_g)_{g\in G})$ and associated factor map $\alpha : X \to Y$.
	\ben
	\item A closed subspace $M \subset L^2(X)$ is called a \emph{$\Y$-module} if for every $f \in M$ and $h \in L^0(Y)$, if $hf \in L^2(X)$, then $hf \in M$.
	\item If $M$ is a $\Y$-module, we say that a subset $L \subset M$ \emph{spans} $M$ if for every $f \in M$, there exist sequences of functions $(f_n)_{n=1}^\infty \subset L$ and $(c_n)_{n=1}^\infty \subset L^0(Y)$ such that $f(x) = \sum_n c_n(\alpha(x)) f_n(x)$.
	\item A function $f \in L^2(X)$ is called a \emph{generalized eigenfunction} or \emph{$\Y$-eigenfunction} if the $\Y$-module spanned by $\{ T_gf : g \in G\}$ is of finite rank.\footnote{This is a technical assumption which ensures, among other desirable consequences, that $Y$-module spanned by $\{T_gf : g \in G\}$ has an orthonormal $Y$-basis.}
	\item Denote by $\cale(\X/\Y)$ the closure in $L^2(X)$ of the subspace of $\Y$-eigenfunctions. In words, $\cale(\X/\Y)$ is called the \emph{$\Y$-eigenfunction space of $\X$}. 
	\een
\end{definition}
The $\Y$-eigenfunction space $\cale(\X/\Y) \subset L^2(X)$ is a $\Y$-module that contains not only $L^2(Y)$ but also the invariant space $\{ f \in L^2(X) : T_gf = f \text{ for all } g \in G\}$. In general, if $L^2(X) = \cale(\X/\Y)$, then $\X$ is called a \emph{compact extension} of $\Y$. Moreover, $\cale(\X/\Y)$ generates a sub-$\sigma$-algebra of $\B$, and the resulting factor of $\X$ is called the \emph{maximal compact extension} of $\Y$ in $\X$. See \cite{diag} and \cite{glasner} for references.

Here is a second round of definitions.
\begin{definition} \label{rel indep joining} Suppose two measure-preserving systems $\X_1 = \left( X_1,\B_1,\mu_1,(T_g^{(1)})_{g\in G} \right)$ and \ $\X_2 = \left( X_2,\B_2,\mu_2,(T_g^{(2)})_{g\in G} \right)$ share a common factor $\Y = (Y, \mathcal{D},\nu,(S_g)_{g\in G})$, and let $\alpha_1 : \X_1 \to \Y$ and $\alpha_2 : \X_2 \to \Y$ be the associated factor maps. The \emph{relatively independent joining of $\X_1$ and $\X_2$ over $\Y$} is the measure-preserving system $\X_1 \times_\Y \X_2 = (X_1 \times_\Y X_2, \B_1 \times_\Y \B_2, \mu_1 \times_\Y \mu_2, (T_g^{(1)} \times T_g^{(2)})_{g\in G})$, where we set $X_1 \times_\Y X_2 := \{ (x_1,x_2) \in X_1 \times X_2 : \alpha_1(x_1) = \alpha_2(x_2) \}$, the $\sigma$-algebra $\B_1 \times_\Y \B_2$ is the restriction of $\B_1 \otimes \B_2$ to $X_1 \times_\Y X_2$, and $\mu_1 \times_\Y \mu_2$ is the measure defined by setting, for every $f_1 \in L^\infty(X_1)$ and $f_2 \in L^\infty(X_2)$,
	\be \int_{X_1 \times_\Y X_2} f_1 \otimes f_2 \ d(\mu_1 \times_\Y \mu_2) \ = \ \int_Y \E{f_1}{\D} \E{f_2}{\D} \ d\nu,
	\ee
	where $f_1 \otimes f_2$ denotes the function $(x_1,x_2) \mapsto f_1(x_1)f_2(x_2)$.
\end{definition}

We will need to compute $\cale(\X_1 \times_\Y \X_2 / \Y)$ in terms of the $\Y$-eigenfunction spaces of $\X_1$ and of $\X_2$. First, let us agree on one more piece of notation: If $M_1 \subset L^2(X_1)$ and $M_2 \subset L^2(X_2)$ are $\Y$-modules, denote by $M_1 \otimes_\Y M_2$ the closed $\Y$-module of $L^2(X_1 \times_\Y X_2)$ spanned by products $f_1 \otimes f_2$, where $f_1 \in M_1$ and $f_2 \in M_2$ are bounded functions. After some effort, one can compute the $\Y$-eigenfunction space of $X_1 \times_\Y X_2$ as follows. 
\begin{thm} \label{Diag77 Thm 7.1} Assuming the setup of Definition \ref{rel indep joining}, we have
	\be
	\cale(\X_1 \times_\Y \X_2/\Y) \ = \ \cale(\X_1/\Y) \otimes_\Y \cale(\X_2/\Y).
	\ee
\end{thm}
\begin{proof} See \cite[Theorem 7.1]{diag} for a proof in the restricted context of $\Z$-systems, and see \cite[Theorem 9.21]{glasner} for a direct quotation in the context of $G$-systems.
\end{proof}
\begin{thm} \label{Diag77 Thm 7.4} Assuming the setup of Definition \ref{rel indep joining}, if moreover $\X_2$ is ergodic, then
	\be
	\cale(\X_1 \times_\Y \X_2/\X_2) \ = \ \cale(\X_1/\Y) \otimes L^2(X_2). 
	\ee
\end{thm}
\begin{proof} See \cite[Theorem 7.4]{diag} for a proof in the context of $\Z$-systems. The argument in the context of $G$-systems is analogous.
\end{proof}

\subsection{A theorem on conditional products}\label{subsec: proof of thm 5.1}

In this subsection, we repeatedly use the following set of assumptions:
\begin{enumerate}
	\item Fix a positive integer $k$.
	\item For each $i \in \{1,\ldots, k\}$, let $\X_i = (X_i, \B_i,\mu_i,(T^{(i)}_g)_{g\in G})$ be a measure-preserving system with a factor $\Y_i = (Y_i,\mathcal{D}_i,\nu_i,(S^{(i)}_g)_{g\in G})$ and associated factor map $\alpha_i : X_i \to Y_i$.
	\item Let $(X,\B)$ denote the product $(\prodo i k X_i, \otimes_{i=1}^{k} \B_i)$, and let $(Y,\mathcal D)$ denote the product $(\prodo i k Y_i, \otimes_{i=1}^{k} \mathcal{D}_i)$.
	\item Let $\alpha : X \to Y$ be defined by $\alpha := (\alpha_1,\ldots, \alpha_k)$.
\end{enumerate}
\begin{definition} A measure $\mu$ on $(X,\B)$ is to said to have \emph{correct marginals} if its image on the $i$th coordinate is $\mu_i$.
\end{definition}
\begin{rem} By definition, a \emph{joining} on $X$ has correct marginals and is invariant under the diagonal action $T_g := T^{(1)}_g \times \cdots \times T^{(k)}_g$, and for some simpler lemmas here, we discuss joinings only when necessary. If $\mu$ is a measure on $X$ with correct marginals, then the pushforward $\alpha_*\mu = \mu \circ \alpha^{-1}$ is a measure on $Y$ with correct marginals.
\end{rem}
\begin{definition} Let $\mu$ be a measure on $(X,\B)$ with correct marginals, and write $\alpha_*\mu$ for its pushforward on $(Y,\mathcal{D})$. We say that $\mu$ is a conditional product measure relative to $Y$ if
	\be
	\mu = \int_Y \mu_{1,y_1} \times \cdots \times \mu_{k,y_k} \ d\alpha_*\mu(y_1,\ldots, y_k),
	\ee
	where $\mu_{i}$ is disintegrated as $\mu_i = \int_{Y_i} \mu_{i,y_i} \ d\nu_i(y_i)$.
	\begin{rem}
		The definition of a conditional product measure is sensible since the set of $(y_1,\ldots,y_k)$'s for which the integrand is defined has full $\alpha_*\mu$-measure since $\alpha_*\mu$ has correct marginals. A \emph{conditional product joining} relative to $Y$ is a conditional product measure relative to $Y$ that is also invariant under the diagonal action.
	\end{rem}
\end{definition}
Here is an equivalent characterization, proven by unfolding definitions.
\begin{lemma} \label{equivalent cpm} A measure $\mu$ on $(X,\B)$ with correct marginals is a conditional product measure relative to $(Y,\mathcal{D})$ if and only if for all $k$-tuples $(f_i)_{i=1}^k$ with $f_i \in L^{\infty}(X_i,\B_i,\mu_i)$, we have
	\be
	\int_X \bigotimes_{i=1}^k f_i \ d\mu \ = \ \int_Y \bigotimes_{i=1}^k \E{f_i}{Y_i} \ d\alpha_*\mu,
	\ee
	where, as usual, we use the notation $\bigotimes_{i=1}^k f_i$ for the function $(x_1,\ldots, x_k) \mapsto f_1(x_1) \cdots f_k(x_k)$.
\end{lemma}
We will need to change the factor with respect to which a measure is a conditional product measure. The following lemma suffices for this purpose.
\begin{lemma}[\cite{diag}, Lemma 9.2] \label{Diag77 Lem 9.2} Suppose for each $i \in \{1,\ldots, k\}$, we have an intermediate factor $(Z_i,\mathcal{C}_i)$ with $X_i \to Z_i \to Y_i$ and factor maps $\beta_{i} : X_i \to Z_i$ and $\beta_{i}' : Z_i \to Y_i$. Suppose that $\mu$ is a measure on $X$ with correct marginals that is also a conditional product measure relative to $Y$. Then $\mu$ is a conditional product measure relative to $Z := \prodo i k Z_i$.
\end{lemma}
\begin{proof}  For each $i$, set $\pi_i := \beta_i' \circ \beta_i$, and write $\beta := (\beta_1,\ldots, \beta_k)$ and similarly for $\pi$. We show the lemma using Lemma~\ref{equivalent cpm}. Fix a $k$-tuple $(f_i)_{i=1}^k$ with $f_i \in L^\infty(X_i)$. We observe that
	\begin{align*}
		\int_Z \bigotimes_{i=1}^k \E{f_i}{Z_i} \ d\beta_*\mu \ & = \ \int_X \bigotimes_{i=1}^k \E{f_i}{\beta_i^{-1}\mathcal{C}_i} \ d\mu \ \overset{*}{=} \ \int_Y \bigotimes_{i=1}^k \E{\E{f_i}{\beta_i^{-1}\mathcal{C}_i}}{Y_i} \ d \pi_*\mu \\
		& = \ \int_X \bigotimes_{i=1}^k \E{\E{f_i}{\beta_i^{-1}\mathcal{C}_i}}{\pi_i^{-1}\mathcal{D}_i} \ d \mu \ \overset{**}{=} \ \int_X \bigotimes_{i=1}^k
		\E{f_i}{\pi_i^{-1}\mathcal{D}_i} \ d\mu \\
		& = \ \int_Y \bigotimes_{i=1}^k \E{f_i}{Y_i} \ d\pi_*\mu \ \overset{*}{=} \ \int_X \bigotimes_{i=1}^k f_i \ d\mu,
	\end{align*}
	where the single-starred equalities hold since $X$ is a conditional product measure relative to $Y$, and the double-starred equality holds since $\pi_i^{-1}\mathcal{D}_i \subset \beta_i^{-1} \mathcal{C}_i$.
\end{proof}
According to the next lemma, the presence of a conditional product measure on a product system $X$ ensures that $X$ is a relatively independent joining of certain relatively independent joinings, a fact which we promise to use.
\begin{lemma} \label{Diag77 Lem 9.3} Let $\X, \Y, \X_i$, and $\Y_i$ be as above. Put $\tilde{\X} := (\X_1 \times_{\Y_1} \Y) \times_\Y (\X_2 \times_{\Y_2} \Y) \times_\Y \cdots \times_\Y (\X_k \times_{\Y_k} \Y)$, and denote by $\tilde{\alpha}$ the function $X \to \tilde{X}$ that maps $x = (x_1,x_2,\ldots, x_k)$ to $\left( (x_1,\alpha(x)),(x_2,\alpha(x)),\ldots,\right.$ $\left.(x_k,\alpha(x)) \right)$. If $\mu$ is a conditional product measure relative to $Y$, then $\tilde{\alpha}$ is a measurable isomorphism of $\X$ with $\tilde{\X}$.
\end{lemma}
\begin{proof} Since $\X$ and $\tilde \X$ share the factor $\Y$, it suffices to show that $\mu$ and the measure on $\tilde X$ disintegrate into the same measures with respect to the factor $\Y$. For reference, see the proof in \cite{diag} of Lemma 9.3.
\end{proof}
There is one more basic result we need.
\begin{lemma}\label{Diag77 Lem 5.3} If each $(X_i,\B_i,\mu_i,(T^{(i)}_g)_{g\in G})$ is ergodic and $\mu$ is a joining on $X$, then almost every ergodic component of $\mu$ has correct marginals.
\end{lemma}
\begin{proof} Using the coordinate projection onto $X_i$, push the ergodic decomposition of $\mu$ forward to get a decomposition of $\mu_i$ into invariant measures, each of which is a pushforward of an ergodic component of $\mu$. As $\mu_i$ is already ergodic, these invariant measures are almost always $\mu_i$. We know this for each $i$.
\end{proof}

The following is a generalized form of \cite[Theorem 9.4]{diag}.
\begin{thm} \label{pre-thm: joining compact ext}
	Assume each $(Y_i, \mathcal{D}_i,\nu_i,(T^{(i)}_g)_{g\in G})$ is ergodic. For each $i$, let $(\hat{Y}_i,\hat{\mathcal{D}}_i,\hat{\nu}_i,(T^{(i)}_g)_{g\in G})$ be the maximal compact extension of $(Y_i, \mathcal{D}_i, \nu_i,(S^{(i)}_g)_{g\in G})$ in $(X_i,\B_i, \mu_i,(T^{(i)}_g)_{g\in G})$. Let $G$ act on $X$ by $T_g := T^{(1)}_g \times \cdots \times T^{(k)}_g$. If $\mu$ is a conditional product joining on $X$ relative to $Y$, then almost all ergodic components of $\mu$ are conditional product measures relative to $\hat{Y} := \prodo i k \hat{Y}_i$.
\end{thm}
\begin{proof}\renewcommand{\qedsymbol}{} Let $\mu' = \alpha_*\mu$. Without loss of generality, $\mu'$ is ergodic.
	First, we need the following:
	\begin{cla} The subspace of $T$-invariant functions $L^2(X, \I_T,\mu)$ is contained in $\bigotimes_{i=1}^k \cale (\X_i/\Y_i)$.
	\end{cla}
	\begin{proof}[Proof of Claim]\renewcommand{\qedsymbol}{} Applying Lemma~\ref{Diag77 Lem 9.3}, $\tilde{\alpha} : X \to \tilde{X}$ is an isomorphism of measure-preserving systems, where $\tilde{X} := (X_1 \times_{Y_1} Y) \times_Y (X_2 \times_{Y_2} Y) \times_Y \cdots \times_Y (X_k \times_{Y_k} Y)$. For each $i$ the base space of $X_i \times_{Y_i} Y$ is a subset of $X_i \times Y$; thus the base space of $\tilde{X}$ is $X \times Y^k$. Moreover, all of the coordinates that live in $Y$ are the same, so we can even view $\tilde{X}$ as a subset of $X \times Y$.
		
		Now, the invariant subspace $L^2(\tilde{X}, \tilde{\B}_T,\tilde{\mu})$ is a subspace of $\cale(\tilde{X}/Y)$. By Theorem \ref{Diag77 Thm 7.1}, \be
		\cale(\tilde{\X}/\Y) \ = \ \cale(\X_1 \times_{\Y_1} \Y/\Y) \otimes_Y \cdots \otimes_Y \cale(\X_k \times_{\Y_k} \Y / \Y).
		\ee
		Then, for each $i$, by Theorem \ref{Diag77 Thm 7.4} (valid since $\mu'$ is ergodic), we have $\cale(\X_i \times_{\Y_i} \Y/\Y) = \cale(\X_1/\Y_i) \otimes L^2(Y)$, which implies that 
		\be \label{to be pulled back}
		L^2(\tilde{X}, \tilde{\B}_T, \mu) \ \subset \ \left( \bigotimes_{i=1}^k \cale (\X_i/\Y_i) \right) \otimes L^2(Y). 
		\ee
		Viewing $\tilde{X}$ as a subset of $X \times Y$, we pull back to transform Equation \eqref{to be pulled back} into the desired containment
		\be
		L^2(X, \I_T,\mu) \ \subset \ \bigotimes_{i=1}^k \cale (\X_i/\Y_i).
		\ee
	\end{proof}
	
	Having proved the claim, we proceed as follows. We have the ergodic decomposition $\mu = \int \mu_x d\mu(x)$; every $f \in L^1(X)$ satisfies $\E{f}{\I_T}(x) = \int f \ d \mu_x$ for $\mu$-a.e. $x \in X$. Since $\mu$ is a joining, Lemma~\ref{Diag77 Lem 5.3} implies that almost every $\mu_x$ has correct marginals.
	
	Via the equivalent formulation in Lemma \ref{equivalent cpm}, we will show that almost every component $\mu_x$ is a conditional product measure with respect to $\hat{Y}$. Thus, we need to show that for every $k$-tuple $(f_i)_{i=1}^k$ with $f_i \in L^\infty(X_i,\mu_i)$, the equation
	\be \label{cpm in 9.4}
	\int_X \bigotimes_{i=1}^k f_i \ d\mu_x \ = \ \int_{\hat{Y}} \bigotimes_{i=1}^k \E{f_i}{\hat{Y}_i} \ d \beta_*\mu_x,
	\ee
	holds for a.e. $x$, where $\beta : X \to \hat{Y}$ is the composition $\beta := (\beta_1,\ldots,\beta_k)$ of the factor maps $\beta_i : X_i \to \hat{Y}_i$.
	This follows from an easy approximation argument whose ingredients are the following four facts, the last of which we prove.
	\begin{lemma} For each $i$, let $f_i, f_i' \in L^\infty(X_i,\mu_i)$ be functions, and let $M$ bound all of $||f_1||_{L^\infty(X_1,\mu_1)}, \ \ldots,$ $\ ||f_k'||_{L^\infty(X_k,\mu_k)}$. Then for any measure $\theta$ on $X$ with correct marginals, we have
		\be
		\left\vert\left\vert \bigotimes_{i=1}^k f_i - \bigotimes_{i=1}^k f_i'\right\vert\right\vert_{L^1(X,\theta)} \ \leq \ M^{k-1} \sum_{i=1}^k || f_i - f_i'||_{L^1(X_i,\mu_i)}.
		\ee
	\end{lemma}
	\begin{lemma} The operator $\E{\cdot}{\beta_i^{-1} \hat{\mathcal{D}}_i}$ is a contraction $L^1(X_i, \B_i,\mu_i) \to L^1(X_i, \beta_i^{-1} \hat{\mathcal{D}}_i,\mu_i)$.
	\end{lemma}
	\begin{lemma} There is a countable subset $\mathcal{F}_i \subset L^\infty(X_i,\mu_i)$ that is dense in $L^1(X_i,\mu_i)$.
	\end{lemma}
	\begin{cla} For a.e $x$, for all $k$-tuples $(f_i)_{i=1}^k$ with $f_i \in \mathcal{F}_i$, we have
		\be \int_X \bigotimes_{i=1}^k f_i \ d\mu_x \ = \ \int_{\hat{Y}} \bigotimes_{i=1}^k \E{f_i}{\hat{Y}_i} \ d\beta_*\mu_x.
		\ee
	\end{cla}
	\begin{proof}[Proof of Claim]
		Fix a $k$-tuple $(f_i)_{i=1}^k$ with $f_i \in \mathcal{F}_i \subset L^\infty(X_i)$. By Lemma \ref{Diag77 Lem 9.2}, since $\mu$ is a conditional product measure relative to $Y$, it is also a conditional product measure relative to $\hat{Y}$. Thus, if $(f_i')_{i=1}^k$ is a $k$-tuple of functions with $f_i' \in \cale(\X_i/\Y_i)$, it follows that
		\be
		\int_X \bigotimes_{i=1}^k f_i f_i' \ d\mu \ = \ \int_{\hat{Y}} \bigotimes_{i=1}^k \E{f_if_i'}{\hat{Y}_i} \ d\beta_*\mu \ = \ \int_X \bigotimes_{i=1}^k \E{f_i f_i'}{\beta_i^{-1}\hat{\mathcal{D}}_i} \ d\mu \ = \ \int_X \bigotimes_{i=1}^k f_i' \E{f_i}{\beta_i^{-1}\hat{\mathcal{D}}_i} \ d\mu,
		\ee 
		which implies that for any $f' \in \bigotimes_{i=1}^k \cale(\X_i/\Y_i)$, we have
		\be
		\int_X f' \cdot \bigotimes_{i=1}^k f_i \ d\mu \ = \ \int_X f' \cdot \bigotimes_{i=1}^k \E{f_i}{\beta_i^{-1}\hat{\mathcal{D}}_i} \ d\mu.
		\ee
		Therefore, by the previous claim that $\bigotimes_{i=1}^k \cale(\X_i/\Y_i)$ contains the invariant subspace $L^2(X,\I_T,\mu)$, we conclude
		\be \label{the previous equation}
		\mathbb{E}\left.\left[ \bigotimes_{i=1}^k f_i \ \right\vert \ \I_T \right] \ = \ \mathbb{E} \left. \left[\bigotimes_{i=1}^k \E{f_i}{\beta_i^{-1}\hat{\mathcal{D}}_i} \ \right\vert \ \I_T\right].
		\ee
		Finally, since every $f \in L^2(X)$ satisfies $\int f \ d\mu_x = \E{f}{\I_T}(x)$ for a.e. $x$, it follows from Equation \eqref{the previous equation} that the equation
		\be
		\int_X \bigotimes_{i=1}^k f_i \ d\mu_x \ = \ \int_X \bigotimes_{i=1}^k \E{f_i}{\beta_i^{-1}\hat{\mathcal{D}}_i} \ d\mu_x \ = \ \int_{\hat{Y}} \bigotimes_{i=1}^k \E{f_i}{\hat{Y}_i} \ d\beta_*\mu_x
		\ee
		holds for a.e. $x$. Since the collections $\mathcal{F}_i$ are countable, the claim holds as stated. $\square$
	\end{proof}
\end{proof}
\vspace{-.3in}We can slightly strengthen the previous theorem to get the following theorem, which generalizes \cite[Theorem 5.1]{fw}. This theorem is stated with the assumptions of this subsection made explicit, for ease of use elsewhere.
\begin{thm} \label{thm: joining compact ext}
	For $1 \le i \le k$, let $\Y_i = \left( Y_i, \D_i, \nu_i, \left( S^{(i)}_g \right)_{g \in G} \right)$
	be a factor of $\X_i = \left( X_i, \B_i, \mu_i, \left( T^{(i)}_g \right)_{g \in G} \right)$
	with maximal compact extension $\hat{\Y_i}$ in $\X_i$.
	Suppose each $\X_i$ has finitely many ergodic components.
	If $\mu$ is a conditional product joining on $\prod_{i=1}^k{X_i}$ relative to $\prod_{i=1}^k{Y_i}$,
	then almost every ergodic component of $\mu$ is a conditional product relative to $\prod_{i=1}^k{\hat{Y_i}}$.
\end{thm}
\begin{proof} Since each $X_i$ has finitely many components, it follows that each $Y_i$ has finitely many components. The conclusion follows from the argument in Theorem \ref{pre-thm: joining compact ext} in the same way that the proof of Theorem 9.5 follows from the proof of Theorem 9.4 in \cite{diag}.
\end{proof}
\begin{rem} The conclusion of this theorem can also be stated as follows: if a function $F \in L^2(\prodo i k X_i)$ is invariant under $T_1 \times \cdots \times T_k$, then there is a function $D \in L^2(\prodo i k \hat{Y_i})$ such that $F(x) = D(k)$, where $x \mapsto k$ is the projection from $\prodo i k X_i$ to $\prodo i k \hat{Y_i}$.
\end{rem}

\subsection{A lemma for the diagonal measure}
We now state and prove a lemma for the diagonal measure, which is needed to apply Theorem \ref{thm: joining compact ext} as a step in the proof of Theorem \ref{thm: compact ext}.
\begin{lem} \label{lem: diag meas}
	Let $\X = (X, \B, \mu, (T_g)_{g \in G})$ be an ergodic measure-preserving system,
	and suppose $(Z_k, \nu)$ is characteristic for averages \eqref{eq: multiple average}.
	Let $\{\varphi_1, \dots, \varphi_{k+1}\}$ be an admissible family.
	Let $S_g := T_{\varphi_1(g)} \times \cdots \times T_{\varphi_{k+1}(g)}$
	be the corresponding diagonal action on $X^{k+1}$.
	Then there are measures $\tilde{\mu}$ on $X^{k+1}$ and $\tilde{\nu}$ on $Z_k^{k+1}$ such that
	\begin{enumerate}[1.]
		\item	$\tilde{\mu}$ is an $S$-invariant joining of $\mu$ with itself;
		\item	$\tilde{\nu}$ is an $S$-invariant joining of $\nu$ with itself;
		\item	$\tilde{\mu}$ is a conditional product joining relative to $(Z_k^{k+1}, \tilde{\nu})$; and
		\item	for all $f_1, f_2, \dots, f_{k+1} \in L^{\infty}(\mu)$,
		\begin{align*}
			\UClim_{g \in G}{\int_X{\prod_{i=1}^{k+1}{T_{\varphi_i(g)}f_i}~d\mu}}
			& = \int_{X^{k+1}}{\bigotimes_{i=1}^{k+1}{f_i}~d\tilde{\mu}}, \\
			\UClim_{g \in G}{\int_{Z_k}{\prod_{i=1}^{k+1}{T_{\varphi_i(g)}\pi_kf_i}~d\nu}}
			& = \int_{Z_k^{k+1}}{\bigotimes_{i=1}^{k+1}{\pi_k f_i}~d\tilde{\nu}}, \\
		\end{align*}
		where $\pi_k$ is the projection $\pi_k : L^2(X) \to L^2(Z_k)$.
	\end{enumerate}
\end{lem}
\begin{proof}
	Define a functional $\Phi : C(X^{k+1}) \to \C$ by
	\begin{align} \label{eq: diag functional} 
		\Phi(F) := \UClim_{g \in G}{\int_{X^{k+1}}{S_g F~d\sigma}},
	\end{align}
	where $\sigma$ is the diagonal measure
	\begin{align}
		\int_{X^{k+1}}{F(x_1, x_2, \dots, x_{k+1})~d\sigma(x_1, x_2, \dots, x_{k+1})}
		:= \int_{X}{F(x, x, \dots, x)~d\mu(x)}.
	\end{align}
	The limit in \eqref{eq: diag functional} exists (see \cite{austin, zorin}), and it is controlled by the factor $\mathbf{Z}_k$.\footnote{This is because $T$ is measure-preserving, so $\Phi(F) = \UClim_{g \in G}{\int_X{F(x, T_{(\varphi_2-\varphi_1)(g)}x, \dots, T_{(\varphi_{k+1}-\varphi_1)(g)}x)~d\mu(x)}}$. Approximating $F$ by linear combinations of simple tensors $\otimes_{i=1}^{k+1}{f_i}$ shows that this limit is controlled by the characteristic factor for the family $\{\varphi_2 - \varphi_1, \dots, \varphi_{k+1} - \varphi_1\}$, which is an admissible family with $k$ elements.}
	Observe that $\Phi$ is a positive linear functional and $|\Phi(F)| \le \|F\|_{\sup}$.
	Thus, by the Riesz representation theorem, there is a (positive) Borel measure
	$\tilde{\mu}$ on $X^{k+1}$ such that $\Phi(F) = \int_{X^{k+1}}{F~d\tilde{\mu}}$.
	Moreover, since $\Phi(\ind) = 1$, $\tilde{\mu}$ is a probability measure.
	
	Note that $\tilde{\mu}$ is invariant with respect to the diagonal action $S$ by the F{\o}lner property:
	for any F{\o}lner sequence $(F_N)_{N \in \N}$, we have
	\begin{align*}
		\left| \int_{X^{k+1}}{ S_g F~d\tilde{\mu}} - \int_{X^{k+1}}{F~d\tilde{\mu}} \right|
		= & \left| \lim_{M \to \infty}{\frac{1}{|F_M|} \sum_{h \in F_M}{\int_X{S_{g+h}F~d\sigma}}}
		- \lim_{M \to \infty}{\frac{1}{|F_M|} \sum_{h \in F_M}{\int_X{S_h F~d\sigma}}} \right| \\
		= & \left| \lim_{M \to \infty}{\frac{1}{|F_M|} \sum_{h \in F_M+g}{\int_X{S_h F~d\sigma}}}
		- \lim_{M \to \infty}{\frac{1}{|F_M|} \sum_{h \in F_M}{\int_X{S_h F~d\sigma}}} \right| \\
		\le & \lim_{M \to \infty}{\frac{1}{|F_M|} \sum_{h \in (F_M+g) \triangle F_M}{
				\int_X{|S_h F|~d\sigma}}} \\
		\le & \lim_{M \to \infty}{\frac{ \left| (F_M+g) \triangle F_M \right|}{|F_M|}} \|F\|_{\sup} = 0.
	\end{align*}
	Moreover, it is easy to check from the definition that $\tilde{\mu}$ is a joining of $\mu$ with itself,
	since $\mu$ is $T$-invariant.
	
	By repeating this construction on the factor $\mathbf{Z}_k$,
	we may define an invariant probability measure $\tilde{\nu}$ on $Z_k^{k+1}$ satisfying
	\begin{align*}
		\int_{Z_k^{k+1}}{F~d\tilde{\nu}}
		= \UClim_{g \in G}{\int_{Z_k}{F \left( S_g(z, z, \dots, z) \right)~d\nu(z)}}.
	\end{align*}
	
	Now, we claim that $\tilde{\mu}$ is a conditional product joining
	relative to the measure $\tilde{\nu}$ on $Z_k^{k+1}$.
	Let $\pi_k : L^2(X) \to L^2(Z_k)$ be the projection onto the factor $\mathbf{Z}_k$.
	Since $\mu$ is $T$-invariant and $Z_k$ is a $T$-invariant sub-$\sigma$-algebra,
	we can use the fact that $\mathbf{Z}_k$ is characteristic for length $k$ averages
	to compute the integral with respect to $\tilde{\mu}$:
	\begin{align*}
		\int_{X^{k+1}}{\bigotimes_{i=1}^{k+1}{f_i}~d\tilde{\mu}}
		& = \UClim_{g \in G}{\int_X{\prod_{i=1}^{k+1}{f_i(T_{\varphi_i(g)}x)}~d\mu(x)}} \\
		& = \UClim_{g \in G}{\int_X{f_1(x) \prod_{i=2}^{k+1}{f_i(T_{(\varphi_i-\varphi_1)(g)}x)}~d\mu(x)}} \\
		& = \UClim_{g \in G}{
			\int_X{f_1(x) \prod_{i=2}^{k+1}{\E{f_i}{Z_k}(T_{(\varphi_i-\varphi_1)(g)}x)}~d\mu(x)}} \\
		& = \UClim_{g \in G}{
			\int_X{f_1(T_{\varphi_1(g)}x) \prod_{i=2}^{k+1}{\E{f_i}{Z_k}(T_{\varphi_i(g)}x)}~d\mu(x)}} \\
		& = \UClim_{g \in G}{\int_{Z_k}{\prod_{i=1}^{k+1}{\pi_k f_i(T_{\varphi_i(g)}z)}~d\nu(z)}} \\
		& = \int_{Z_k^{k+1}}{\bigotimes_{i=1}^{k+1}{\pi_kf_i}~d\tilde{\nu}}.
	\end{align*}
\end{proof}

\begin{rem} \label{random remark on lem: diag meas} 
	When $k=2$, so that $\mathbf{Z}_2 = Z$ is the Kronecker factor, we can see the measure $\tilde{\nu}$
	as the Haar measure on the subgroup
	$W_{\varphi,\psi,\theta} := \overline{\{(z + \hat{\varphi(g)}, z + \hat{\psi(g)}, z + \hat{\theta(g)})
		: z \in Z, g \in G\}} \subseteq Z^3$.
	If $\varphi \Lambda + \psi \Lambda + \theta \Lambda \subseteq \Lambda$,
	then $W_{\varphi, \psi, \theta}$ has the simpler description
	$W_{\varphi, \psi, \theta} = \{(z + \varphi w, z + \psi w, z + \theta w) : z, w \in Z\}
	= Z_{id, id, id} + Z_{\varphi, \psi, \theta}$.
\end{rem}

\subsection{Proof of Theorem~\ref{thm: compact ext}}
With all of this machinery, we now prove Theorem \ref{thm: compact ext}.
Let $\hat{\mathbf{Z}_k}$ be the maximal compact extension of $\mathbf{Z}_k$.
Let $u_g := \prod_{i=1}^{k+1}T_{\varphi_i(g)} f_i$.
By symmetry, it suffices to show that
\begin{align} \label{eq: multi-avg=0}
	\UClim_{g \in G}{u_g} = 0
\end{align}
when $\E{f_1}{\hat{Z_k}} = 0$.

We will assume $\E{f_1}{\hat{Z_k}} = 0$ and show that \eqref{eq: multi-avg=0}
holds using the van der Corput trick (Lemma \ref{lem: vdC}).
We can use the diagonal measure from Lemma \ref{lem: diag meas} to compute the limit
\begin{align*}
	\gamma_h & := \UClim_{g \in G}{\innprod{u_{g+h}}{u_g}} \\
	& = \int_{X^{k+1}}{\bigotimes_{i=1}^{k+1}{\left( \overline{f}_iT_{\varphi_i(h)}f_i \right)}~d\tilde{\mu}} \\
	& = \int_{X^{k+1}}{\left( \bigotimes_{i=1}^{k+1}{\overline{f}_i} \right)
		\left( \bigotimes_{i=1}^{k+1}{T_{\varphi_i(h)}f_i} \right)~d\tilde{\mu}}.
\end{align*}
Thus, letting $F = \bigotimes_{i=1}^{k+1}{f_i}$
and $S_h = T_{\varphi_1(h)} \times T_{\varphi_2(h)} \times \cdots \times T_{\varphi_{k+1}(h)}$
the diagonal action, we have by the ergodic theorem that
\begin{align*}
	\UClim_{h \in G}{\gamma_h}
	& = \UClim_{h \in G}{\int_{X^{k+1}}{\overline{F}~S_hF~d\tilde{\mu}}}. \\
	& = \int_{X^{k+1}}{\overline{F}~\E{F}{\I_S}~d\tilde{\mu}},
\end{align*}
where $\I_S$ is the $\sigma$-algebra of $S$-invariant sets.
Now, Lemma \ref{lem: diag meas} showed that the measure $\tilde{\mu}$ is a conditional product joining.
Furthermore, since $\{\varphi_1, \dots, \varphi_{k+1}\}$ is an admissible family,
the actions $\left( Z_k, (T_{\varphi_i(h)})_{h \in G} \right)$ each have finitely many ergodic components.
So Theorem \ref{thm: joining compact ext} applies, and we can write
$\E{F}{\I_S}(x) = D(k)$ for some function
$D \in L^2(\hat{Z_k}^{k+1})$, where $x \mapsto k$ is the projection $X^{k+1} \to \hat{Z_k}^{k+1}$.
Thus,
\begin{align} \label{eq: lim}
	\UClim_{h \in G}{\gamma_h} = \int_{X^{k+1}}{\overline{F} D~d\tilde{\mu}}.
\end{align}
It remains to show that the quantity in \eqref{eq: lim} is in fact equal to zero.

To do so, we define a functional $\Psi : L^2(\hat{Z_k}^{k+1}) \to \C$ by
\begin{align*}
	\Psi(H) := \int_{X^{k+1}}{\overline{F} H~d\tilde{\mu}}
\end{align*}
That is, $\Psi = \innprod{\cdot}{F}_{\tilde{\mu}}$.
We want to show that $\Psi \equiv 0$.
Since $\Psi$ is continuous and linear and $L^2(\hat{Z_k}^{k+1}) = \bigotimes_{i=1}^{k+1} L^2(\hat{Z_k})$,
it suffices to show that $\Psi(H) = 0$ for $H$ of the form
$H = \bigotimes_{i=1}^{k+1}{h_i}$ with $h_i \in L^2(\hat{Z_k})$.
By Lemma \ref{lem: diag meas}, $\tilde{\mu}$ is a conditional product joining
with respect to the measure $\tilde{\nu}$ on $Z_k^{k+1}$, so
\begin{align*}
	\Psi \left( \bigotimes_{i=1}^{k+1}{h_i} \right)
	& = \int_{X^{k+1}}{\bigotimes_{i=1}^{k+1}{\left( \overline{f}_i h_i \right)}~d\tilde{\mu}} \\
	& = \int_{Z_k^{k+1}}{\bigotimes_{i=1}^{k+1}{\E{\overline{f}_i h_i}{Z_k}}~d\tilde{\nu}}.
\end{align*}
Now, since $\hat{\mathbf{Z}_k}$ is an extension of $\mathbf{Z}_k$,
\begin{align*}
	\E{\overline{f}_1 h_1}{Z_k}
	= \E{\E{\overline{f}_1 h_1}{\hat{Z_k}}}{Z_k}
	= \E{h_1 \overline{\E{f_1}{\hat{Z_k}}}}{Z_k} = 0,
\end{align*}
since we assumed that $\E{f_1}{\hat{Z_k}} = 0$.
Thus, $\Psi \left( \bigotimes_{i=1}^{k+1}{h_i} \right) = 0$.
This proves that $\Psi \equiv 0$, so in particular $\Psi(D) = 0$.
Therefore, by Lemma \ref{lem: vdC}, equation \eqref{eq: multi-avg=0} holds.
We have shown that $\hat{\mathbf{Z}_k}$ is an extension of $\mathbf{Z}_{k+1}$.

Since $\hat{\mathbf{Z}_k}$ is an extension of $\mathbf{Z}_{k+1}$, which is an extension of $\mathbf{Z}_k$, and $\hat{\mathbf{Z}_k}$ is a compact extension of $\mathbf{Z}_k$, it follows by, e.g., \cite[Lemma 9.12]{glasner} that $\mathbf{Z}_{k+1}$ is a compact extension of $\mathbf{Z}_k$. \qed


\section{From compact extensions to abelian group extensions} \label{sec: abelian ext}

We now move to analyzing the multiple ergodic averages
\begin{align} \label{eq: triple avg}
	\UClim_{g \in G}{T_{\varphi(g)}f_1 \cdot T_{\psi(g)}f_2 \cdot T_{\theta(g)}f_3}.
\end{align}
Later, we will restrict to the case $\theta = \varphi + \psi$.
By Theorem \ref{thm: compact ext}, the average \eqref{eq: triple avg} is controlled
by a compact extension of the Kronecker factor.
Utilizing a technique from \cite{cl} and \cite{fw}, we can obtain a more refined description of this factor.

First, a general compact extension of the Kronecker factor is a skew product of the form $Z \times_{\rho} K/L$,
where $K$ is a compact group, $L$ a closed subgroup, and $\rho_g : Z \to K$ a cocycle
(see \cite[Theorem 9.14]{glasner}).
That is, $(\rho_g)_{g \in G}$ satisfies the cocycle equation
\begin{align}
	\rho_{g+h}(z) = \rho_g(z + \hat{h})\rho_h(z),
\end{align}
and $T$ acts on $Z \times K/L$ by $T_g(z,kL) = \left( z + \hat{g}, \rho_g(z) k L \right)$.

It is easiest to deal with the case that $L$ is the trivial group, so that the compact extension
is just a group extension.
Following \cite{fw}, we call a system $\X$ \emph{normal} if
the maximal compact extension of the Kronecker factor is a group extension.
Fortunately, it is enough to study normal systems:
\begin{prop}[\cite{fw}, Theorem 8.8] \label{prop: normal ext}
	Every ergodic system is a factor of a normal ergodic system.
\end{prop}

Furstenberg and Weiss prove this in the case $G = \Z$,
but their proof relies only on general facts from group theory
and categorical arguments with commutative diagrams
that can be interpreted in the category of $G$-systems with no added difficulty
(see \cite[Sections 7 and 8]{fw}).
The one exception is \cite[Lemma 8.3]{fw}, which states that a homomorphism $\theta : \X \to \X'$
between measure-preserving systems induces a homomorphism of pairs
$(\X, \mathbf{Z}) \to (\X, \mathbf{Z}')$,
where $\mathbf{Z}$ and $\mathbf{Z}'$ are the Kronecker factors of $\X$ and $\X'$ respectively.
To prove this in the setting of general countable abelian groups requires a few modifications.
For completeness, we provide a proof of this fact in our setting:

\begin{lem}[c.f. \cite{fw}, Lemma 8.3]
	Let $\X = \left( X, \B, \mu, (T_g) \right)$ be an ergodic $G$-system with Kronecker factor $\mathbf{Z}$.
	Let $\X'$ be another $G$-system with Kronecker factor $\mathbf{Z}'$.
	Suppose $\theta : \X \to \X'$ is a homomorphism.
	Then $\theta$ defines a homomorphism of pairs\footnote{This means that
		$\theta$ induces a homomorphism $\mathbf{Z} \to \mathbf{Z}'$.}
	$(\X, \mathbf{Z}) \to (\X', \mathbf{Z}')$.
\end{lem}
\begin{proof}
	Since all of the systems involved here are homomorphic images of $\X$,
	we deal with the space $L^2(X)$ and view the other systems as sub-$\sigma$-algebras of $\B$
	(or equivalently, as $L^2$-closures of sub-algebras of $L^{\infty}(X) \subseteq L^2(X)$).
	Let $f \in L^2(X')$.
	We want to show $\E{f}{Z} \in L^2(Z')$.
	
	We know $Z$ is a compact abelian group and $L^2(Z)$ is spanned by characters on $Z$.
	Let $\Lambda = \hat{Z}$.
	We may think of $\Lambda$ as a subgroup of $\hat{G}$.
	Hence, $\E{f}{Z} = \sum_{\lambda \in \Lambda}{\innprod{f}{e_{\lambda}} e_{\lambda}}$.
	Suppose $\innprod{f}{e_{\lambda}} \ne 0$.
	Then
	\begin{align*}
		\innprod{\overline{\lambda(g)} T_gf}{e_{\lambda}}
		= \innprod{T_gf}{T_ge_{\lambda}}
		= \innprod{f}{e_{\lambda}} \ne 0.
	\end{align*}
	Thus, $h := \UClim_{g \in G}{\overline{\lambda}(g) T_gf} \ne 0$.
	But $h$ is an eigenfunction in $L^2(X')$ with eigenvalue $\lambda$.
	Therefore, $h \in L^2(Z')$.
	It follows that $\E{f}{Z} \in L^2(Z')$ as desired.
\end{proof}

The lemmas in \cite[Section 9]{fw} are again of a general nature and allow for further reduction
to abelian group extensions (assuming $\X$ is normal).

\begin{thm} \label{thm: abelian ext}
	Let $\X$ be a normal ergodic system.
	There is a factor $\Y$ of $\X$ that is an abelian group extension of the Kronecker factor,
	$Y = Z \times_{\rho} H$, such that $\Y$ is characteristic for the averages \eqref{eq: triple avg}.
\end{thm}


\section{Conze--Lesigne factors} \label{sec: CL}

Now we restrict to the case $\theta = \varphi + \psi$ in \eqref{eq: triple avg}.
Instead of handling the cocycle $\rho_g : Z \to H$ appearing in Theorem \ref{thm: abelian ext} directly,
we will deal with the family of cocycles $\sigma_g = \chi \circ \rho_g : Z \to S^1$, where $\chi \in \hat{H}$ is a character.

The next step in \cite{fw} is to use a clever argument involving the notion of Mackey groups to reduce to an identity
for the resulting cocycles $\sigma_g : Z \to S^1$:
\begin{align} \label{eq: cocycle}
	\sigma_{\varphi(g)}(z_1) p_g(z_2) q_g(z_3)
	= \frac{F\left( z_1 + \hat{\varphi(g)}, z_2 + \hat{\psi(g)}, z_3 + \hat{(\varphi+\psi)(g)} \right)}
	{F(z_1, z_2, z_3)},
\end{align}
where $p_g, q_g : Z \to S^1$, $z_1, z_2, z_3 \in W_{\varphi, \psi, \varphi + \psi}$,
and $F : W_{\varphi, \psi, \varphi + \psi} \to S^1$.
This argument is contained in \cite[Section 10]{fw} for $\Z$-systems.
The required technical results are all of a general group-theoretic nature,
so the same argument works for our setting.
We proceed from this identity following the approach in \cite{fw}.
Several modifications are required to push the results from $\Z$-systems to general $G$-systems,
but the skeleton of the argument is the same.

For $\varphi, \psi : G \to G$, recall
$Z_{\varphi,\psi} := \overline{\left\{ \left( \hat{\varphi(g)}, \hat{\psi(g)} \right) : g \in G \right\}}$.
Let $\Delta = Z_{id,id} \cap Z_{\varphi,\psi}$ and $W_{\varphi,\psi} = Z_{id,id} + Z_{\varphi,\psi}$.
By the second isomorphism theorem, $W_{\varphi,\psi} / Z_{\varphi,\psi} \cong Z_{id,id} / \Delta$.

Let $\lambda_0 : Z_{id,id} / \Delta \to Z_{id,id}$ be a measurable cross-section of the quotient homomorphism.
Using $W_{\varphi,\psi} / Z_{\varphi,\psi} \cong Z_{id,id} / \Delta$ and $Z_{id,id} \cong Z$,
lift this to a measurable map $\lambda : W_{\varphi,\psi} \to Z$.
Since this comes from the cross-section, we have
$\left( z_1 - \lambda(z_1, z_2), z_2 - \lambda(z_1, z_2) \right) \in Z_{\varphi,\psi}$
for every $(z_1, z_2) \in W_{\varphi,\psi}$,
and $\lambda(z_1 + u, z_2+ v) = \lambda(z_1, z_2)$ for $(u,v) \in Z_{\varphi,\psi}$.

For $(z_1, z_2) \in W_{\varphi,\psi}$, define $u(z_1, z_2) := z_1 + z_2 - \lambda(z_1, z_2)$.
\begin{lem} \label{lem: map from r,s to r,s,t}
	For $(z_1, z_2) \in W_{\varphi,\psi}$, $(z_1, z_2, u(z_1, z_2)) \in W_{\varphi,\psi,\theta}$.
\end{lem}
\begin{proof}
	Since $\left( z_1 - \lambda(z_1, z_2), z_2 - \lambda(z_1, z_2) \right) \in Z_{\varphi,\psi}$,
	we can find a sequence $(g_k)_{k \in \N}$ in $G$ so that
	\begin{align*}
		z_1 - \lambda(z_1, z_2) = \lim_{k \to \infty}{\hat{\varphi(g_k)}}; \\
		z_2 - \lambda(z_1, z_2) = \lim_{k \to \infty}{\hat{\psi(g_k)}}.
	\end{align*}
	Thus, since $\varphi + \psi = \theta$, we have
	\begin{align*}
		u(z_1, z_2) - \lambda(z_1, z_2)
		& = \left( z_1 - \lambda(z_1, z_2) \right) + \left( z_2 - \lambda(z_1, z_2) \right) \\
		& = \lim_{k \to \infty}{\hat{\varphi(g_k)}} + \lim_{k \to \infty}{\hat{\psi(g_k)}} \\
		& = \lim_{k \to \infty}{\left( \hat{\varphi(g_k)} + \hat{\psi(g_k)} \right)} \\
		& = \lim_{k \to \infty}{\hat{\theta(g_k)}}.
	\end{align*}
	Therefore,
	\begin{align*}
		(z_1, z_2, u(z_1, z_2))
		= (\lambda(z_1, z_2), \lambda(z_1, z_2), \lambda(z_1, z_2))
		+ \lim_{k \to \infty}{ \left( \hat{\varphi(g_k)}, \hat{\psi(g_k)}, \hat{\theta(g_k)} \right)}
		\in W_{\varphi, \psi, \theta}.
	\end{align*}
\end{proof}

Now let $\Delta' := \{\delta \in Z : (\delta, \delta) \in \Delta\}$.
For $\delta \in \Delta'$, the map $\lambda_{\delta}(z_1, z_2) := \lambda(z_1, z_2) + \delta$
is another cross-section, so Lemma \ref{lem: map from r,s to r,s,t} also holds for
$u_{\delta}(z_1, z_2) := z_1 + z_2 - \lambda_{\delta}(z_1, z_2)$.

Now define a map $\tilde{u} : W_{\varphi,\psi} \times \Delta' \to W_{\varphi,\psi,\theta}$ by
$\tilde{u}(z_1, z_2, \delta) := \left( z_1, z_2, u_{\delta}(z_1, z_2) \right)$.
\begin{lem}
	The map $\tilde{u} : W_{\varphi,\psi} \times \Delta' \to W_{\varphi,\psi,\theta}$
	is onto and preserves the Haar measure.
\end{lem}
\begin{proof}
	Suppose $(z_1, z_2, z_3) \in W_{\varphi,\psi,\theta}$.
	Then there is a sequence $(g_k)_{k \in \N}$ in $G$ and an element $z \in Z$ such that
	\begin{align*}
		z_1 = z + \lim_{k \to \infty}{\hat{\varphi(g_k)}}; \\
		z_2 = z + \lim_{k \to \infty}{\hat{\psi(g_k)}}; \\
		z_3 = z + \lim_{k \to \infty}{\hat{\theta(g_k)}}.
	\end{align*}
	Let $\delta := z - \lambda(z,z) \in \Delta'$. 
	We have:
	\begin{align*}
		u_{\delta}(z_1, z_2) & = z_1 + z_2 - \lambda(z_1, z_2) - \delta \\
		& = 2z + \lim_{k \to \infty}{\left( \hat{\varphi(g_k)} + \hat{\psi(g_k)} \right)} - \lambda(z, z) - \delta \\
		& = 2z + (z_3 - z) - \lambda(z, z) - \left( z - \lambda(z, z) \right) \\
		& = z_3.
	\end{align*}
	Hence $\tilde{u}(z_1, z_2, \delta) = (z_1, z_2, z_3)$, so $\tilde{u}$ is onto.
	
	Moreover, for each $\delta \in \Delta'$, the map
	$(z_1, z_2) \mapsto \left( z_1, z_2, u_{\delta}(z_1, z_2) \right)$ is a cross-section
	of the homomorphism $W_{\varphi,\psi,\theta} \to W_{\varphi,\psi}$,
	so $\tilde{u}$ preserves the Haar measure.
\end{proof}

\begin{lem} \label{lem: u compatibility}
	For $(z_1, z_2) \in W_{\varphi,\psi}$, $g \in G$, and $\delta \in \Delta'$:
	\begin{align*}
		u_{\delta}(z_1 + \hat{\varphi(g)}, z_2 + \hat{\psi(g)}) = u_{\delta}(z_1, z_2) + \hat{\theta(g)}.
	\end{align*}
\end{lem}
\begin{proof}
	\begin{align*}
		u_{\delta}(z_1 + \hat{\varphi(g)}, z_2 + \hat{\psi(g)})
		& = z_1 + \hat{\varphi(g)} + z_2 + \hat{\psi(g)}
		- \lambda_{\delta}(z_1 + \hat{\varphi(g)}, z_2 + \hat{\psi(g)}) \\
		& = z_1 + z_2 + \hat{\theta(g)} - \lambda_{\delta}(z_1, z_2) \\
		& = u_\delta(z_1, z_2) + \hat{\theta(g)}.
	\end{align*}
\end{proof}

Recall equation \eqref{eq: cocycle}:
\begin{align*}
	\sigma_{\varphi(g)}(z_1) p_g(z_2) q_g(z_3)
	= \frac{F\left( z_1 + \hat{\varphi(g)}, z_2 + \hat{\psi(g)}, z_3 + \hat{(\varphi+\psi)(g)} \right)}
	{F(z_1, z_2, z_3)}.
\end{align*}
We now wish to eliminate the variable $z_3$.

First, for $(u,v,w) \in Z_{\varphi,\psi,\theta}$, we replace $(z_1, z_2, z_3)$ by
\begin{align*}
	(z_1 + (u-w), z_2 + (v-w), z_3) = (z_1 + u, z_2 + v, z_3 + w) - (w, w, w) \in W_{\varphi,\psi,\theta}.
\end{align*}
Substituting into \eqref{eq: cocycle}, we get
\begin{align*}
	\sigma_{\varphi(g)}(z_1 + (u-w)) p_g(z_2 + (v-w)) q_g(z_3)
	= \frac{F\left( z_1 + (u-w) + \hat{\varphi(g)}, z_2 + (v-w) + \hat{\psi(g)}, z_3 + \hat{\theta(g)} \right)}
	{F(z_1 + (u-w), z_2+(v-w), z_3)}.
\end{align*}
Dividing by the original expression gives
\begin{align} \label{eq: cocycle2}
	\frac{\sigma_{\varphi(g)}(z_1 + (u-w))}{\sigma_{\varphi(g)}(z_1)} \frac{p_g(z_2 + (v-w))}{p_g(z_2)}
	= \frac{F_{u,v,w}\left( z_1 + \hat{\varphi(g)}, z_2 + \hat{\psi(g)}, z_3 + \hat{\theta(g)} \right)}
	{F_{u,v,w}(z_1, z_2, z_3)},
\end{align}
where
\begin{align*}
	F_{u,v,w}(z_1, z_2, z_3) := \frac{F\left( z_1 + u-w, z_2 + v-w, z_3 \right)}{F(z_1, z_2, z_3)}.
\end{align*}

We now want to replace $z_3$ by $u_{\delta}(z_1, z_2)$.
The equation \eqref{eq: cocycle2} holds for almost every $(z_1, z_2, z_3) \in W_{\varphi, \psi, \theta}$
and $(u,v,w) \in Z_{\varphi,\psi,\theta}$.
By Lemma 5.2, we can decompose $W_{\varphi,\psi,\theta}$ (up to isomorphism)
as the Cartesian product $W_{\varphi,\psi} \times \Delta'$.
By Fubini's theorem, it follows that \eqref{eq: cocycle2} holds on almost every slice
$(z_1, z_2, z_3) \in \tilde{u}\left( W_{\varphi,\psi} \times \{\delta\} \right)$.
That is, for almost every $\delta \in \Delta'$, \eqref{eq: cocycle2} holds almost everywhere upon replacing
$z_3$ by $u_{\delta}(z_1, z_2)$.
Fix any such $\delta$, and let $H_{u,v,w}(z_1, z_2) := F_{u,v,w} \left( z_1, z_2, u_{\delta}(z_1, z_2) \right)$.
Applying Lemma \ref{lem: u compatibility}, we now have
\begin{align} \label{eq: cocycle 2-var}
	\frac{\sigma_{\varphi(g)}(z_1 + (u-w))}{\sigma_{\varphi(g)}(z_1)} \frac{p_g(z_2 + (v-w))}{p_g(z_2)}
	= \frac{H_{u,v,w}\left( z_1 + \hat{\varphi(g)}, z_2 + \hat{\psi(g)} \right)}{H_{u,v,w}(z_1, z_2)}.
\end{align}

We want to use this identity to get a similar identity for the cocycle $\sigma$ on its own.
To do so, we need the following lemma:
\begin{lem} \label{lem: separate variables}
	Let $\X = \left( X, \B, \mu, (T_g)_{g \in G} \right)$ and $\Y = \left( Y, \D, \nu, (S_g)_{g \in G} \right)$
	be ergodic systems.
	Suppose $(\alpha_g)_{g \in G}$ and $(\beta_g)_{g \in G}$ are cocycles for $\X$ and $\Y$ respectively,
	each taking values in $S^1$.
	Suppose $L : X \times Y \to \C$ is not 0 a.e. and satisfies for all $g \in G$
	\begin{align}
		\alpha_g(x) \beta_g(y) L(x, y) = L(T_g x, S_g y).
	\end{align}
	Then there are functions $M : X \to S^1$ and $N : Y \to S^1$ and characters $c', c'' \in \hat{G}$ so that
	\begin{align} \label{eq: separate variables}
		\alpha_g(x) = c'(g) \frac{M(T_g x)}{M(x)}, \quad
		\beta_g(y) = c''(g) \frac{N(S_g y)}{N(y)}.
	\end{align}
\end{lem}
\begin{proof}
	Form $S^1$-extensions by defining
	\begin{align*}
		\tilde{T}_g(x, \zeta) = \left( T_g x, \alpha_g(x)^{-1} \zeta \right), \quad
		\tilde{S}_g(y, \eta) = \left( S_g y, \beta_g(y)^{-1} \eta \right).
	\end{align*}
	These are actions because $\alpha$ and $\beta$ are cocycles for $T$ and $S$ respectively.
	Denote by $\tilde{X}$ and $\tilde{Y}$ the product spaces
	$\tilde{X} = X \times S^1$ and $\tilde{Y} = Y \times S^1$.
	Define $\tilde{L} : \tilde{X} \times \tilde{Y} \to \C$ by
	\begin{align*}
		\tilde{L}(x, \zeta; y, \eta) := \zeta \eta L(x,y).
	\end{align*}
	By construction, $\tilde{L}$ is $\tilde{T} \times \tilde{S}$ invariant:
	\begin{align*}
		\left( \tilde{T}_g \times \tilde{S}_g \right) \tilde{L}(x, \zeta; y, \eta)
		& = \left( \alpha_g(x)^{-1}\zeta \right) \left( \beta_g(y)^{-1}\eta \right) L(T_g x, S_g y) \\
		& = \alpha_g(x)^{-1} \zeta \beta_g(y)^{-1} \eta \alpha_g(x) \beta_g(y) L(x,y) \\
		& = \zeta \eta L(x,y) \\
		& = \tilde{L}(x, \zeta; y, \eta).
	\end{align*}
	It follows that $\tilde{L}$ comes from the product of the Kronecker factors for
	$\tilde{T}$ and $\tilde{S}$, so we can write $\tilde{L}$ as a sum of products of eigenfunctions:
	\begin{align} \label{eq: decomp}
		\zeta \eta L(x,y) = \sum_i{H_i(x, \zeta) K_i(y, \eta)},
	\end{align}
	where $\tilde{T}_g H_i = c'_i(g) H_i$ and $\tilde{S}_g K_i = c'_i(g)^{-1} K_i$
	for some character $c'_i \in \hat{G}$.
	Now we can write
	\begin{align*}
		H_i(x, \zeta) = \sum_j{h_{ij}(x) \zeta^j}, \quad K_i(y, \eta) = \sum_j{k_{ij}(y) \eta^j}.
	\end{align*}
	Substituting back into \eqref{eq: decomp} gives
	\begin{align*}
		\zeta \eta L(x,y) = \sum_{i,j,j'}{\zeta^j \eta^{j'} h_{ij}(x) k_{ij'}(y)}
	\end{align*}
	But $L$ does not depend on $\zeta$ or $\eta$, so only the $j=j'=1$ term contributes,
	and evaluating at $\zeta = \eta = 1$, we have
	\begin{align*}
		L(x,y) = \sum_i{h_{i1}(x) k_{i1}(y)}
	\end{align*}
	Therefore, for some $i$, $h_{i1} \not\equiv 0$ and $k_{i1} \not\equiv 0$.
	
	But
	\begin{align*}
		\tilde{T}_g \left( \sum_j{h_{ij}(x)\zeta^j} \right) = c'_i(g) \sum_j{h_{ij}(x)\zeta^j},
	\end{align*}
	so
	\begin{align*}
		\sum_j{h_{ij}(T_gx) \alpha_g(x)^{-j} \zeta^j} = c'_i(g) \sum_j{h_{ij}(x)\zeta^j}.
	\end{align*}
	Matching powers of $\zeta$, we get $h_{i1}(T_gx) \alpha_g(x)^{-1} = c'_i(g) h_{i1}(x)$.
	But $T$ is ergodic, so taking the absolute value of both sides,
	we find that $|h_{i1}|$ is constant, say $|h_{i1}| = \lambda \ne 0$.
	Let $M(x) = \lambda^{-1} h_{i1}(x)$ and $c'(g) = c'_i(g)^{-1}$.
	Then $\alpha_g(x)$ has the desired form \eqref{eq: separate variables}.
	This process can be repeated to show that $\beta_g(y)$ also has this form.
\end{proof}

For $(u,v,w) \in Z_{\varphi,\psi,\theta}$, we can define a cocycle
\begin{align*}
	\alpha_g(x) := \frac{\sigma_{\varphi(g)}(x + (u-w))}{\sigma_{\varphi(g)}(x)}.
\end{align*}
However, Lemma \ref{lem: separate variables} does not immediately apply, since the maps
$z_1 \mapsto z_1 + \hat{\varphi(g)}$ and $z_2 \mapsto z_2 + \hat{\psi(g)}$ may not be ergodic,
and \eqref{eq: cocycle 2-var} only holds on $W_{\varphi,\psi} \subseteq Z \times Z$.
We solve the first problem by taking ergodic components $\hat{j} + Z_\varphi$ and $\hat{j} + Z_\psi$.
For the latter issue, we need the following lemma
(note that $\xi$ will have finite index image, since $(\varphi, \psi)$ is an admissible pair):
\begin{lem}
	Let $\xi = \psi - \varphi$.
	Then $Z_\xi \times Z_\xi \subseteq W_{\varphi,\psi}$.
\end{lem}
\begin{proof}
	It suffices to show $(\hat{\xi(g)}, \hat{\xi(h)}) \in W_{\varphi,\psi}$ for all $g, h \in G$,
	since $W_{\varphi,\psi}$ is a closed subgroup of $Z \times Z$.
	Suppose $g, h \in G$.
	Then set $k := h - g$ and $z := \hat{\psi(g)} - \hat{\varphi(h)}$.
	Then
	\begin{align*}
		(\hat{\xi(g)}, \hat{\xi(h)}) & = (\hat{\psi(g)} - \hat{\varphi(g)}, \hat{\psi(h)} - \hat{\varphi(h)}) \\
		& = (\hat{\psi(g)} - \hat{\varphi(h)} + \hat{\varphi(h)} - \hat{\varphi(g)},
		\hat{\psi(h)} - \hat{\psi(g)} + \hat{\psi(g)} -\hat{\varphi(h)})\\
		& = (z + \hat{\varphi(k)}, z + \hat{\psi(k)}) \in W_{\varphi,\psi}.
	\end{align*}
\end{proof}

From this lemma, we can conclude that
\begin{align*}
	\left( \hat{j} + Z_{\xi \circ \varphi} \right) \times \left( \hat{j} + Z_{\xi \circ \psi} \right)
	\subseteq W_{\varphi,\psi} \cap \left( \left( \hat{j} + Z_\varphi \right)
	\times \left( \hat{j} + Z_\psi \right) \right).
\end{align*}
But $Z_{\xi \circ \varphi}$ and $Z_{\xi \circ \psi}$ have finite index in $Z$,
so $W_{\varphi,\psi}$ has positive measure intersection with the set
$\left( \hat{j} + Z_\varphi \right) \times \left( \hat{j} + Z_\psi \right)$.
Thus, if we define $H'_{u,v,w}$ on $\left( \hat{j} + Z_\varphi \right) \times \left( \hat{j} + Z_\psi \right)$ by
\begin{align*}
	H'_{u,v,w}(z_1, z_2) = \begin{cases}
		H_{u,v,w}(z_1, z_2), & (z_1, z_2) \in W_{\varphi,\psi}; \\
		0, & (z_1, z_2) \notin W_{\varphi,\psi},
	\end{cases}
\end{align*}
then $H'_{u,v,w}$ is not zero almost everywhere.
Moreover,
\begin{align*}
	H'_{u,v,w}(z_1 + \hat{\varphi(g)}, z_2 + \hat{\psi(g)})
	= \alpha_g(z_1) \beta_g(z_2) H'_{u,v,w}(z_1, z_2)
\end{align*}
for $(z_1, z_2) \in \left( \hat{j} + Z_\varphi \right) \times \left( \hat{j} + Z_\psi \right)$,
where $\alpha$ and $\beta$ are the cocycles given by
\begin{align*}
	\alpha_g(x) := \frac{\sigma_{\varphi(g)}(x + (u-w))}{\sigma_{\varphi(g)}(x)}, \quad
	\beta_g(x) := \frac{p_{g}(x + (v-w))}{p_{g}(x)}.
\end{align*}
Hence, applying Lemma \ref{lem: separate variables}, we have
\begin{align*}
	\frac{\sigma_{\varphi(g)}(z + (u-w))}{\sigma_{\varphi(g)}(z)}
	= c_j(g) \frac{K_j(u,v,w; z + \hat{\varphi(g)})}{K_j(u,v,w; z)}
\end{align*}
for each $j \in G$ and $z \in \hat{j} + Z_\varphi$.
Observe that the left-hand side is parametrized by $u-w$ with $u \in Z_{\varphi}$ and $w \in Z_{\theta}$.
Since $\theta = \varphi + \psi$, we have $Z_{\varphi} - Z_{\theta} \supseteq Z_{-\psi} = Z_{\psi}$.
Thus, we can restrict to $u \in Z_{\psi}$, and choose the functions $K_j$ so that
\begin{align*}
	\frac{\sigma_{\varphi(g)}(z + u)}{\sigma_{\varphi(g)}(z)}
	= c_j(g) \frac{K_j(u, z + \hat{\varphi(g)})}{K_j(u, z)}
\end{align*}
for $g \in G$, $z \in \hat{j} + Z_\varphi$, and $u \in Z_{\psi}$.
Now since cosets are disjoint and cover all of $Z$, we can patch together $c_j$ and $K_j$
for different values of $j$ to get a new expression
\begin{align}
	\frac{\sigma_{\varphi(g)}(z + u)}{\sigma_{\varphi(g)}(z)}
	= \Lambda_u(z + Z_\varphi)(g) \frac{K_u(z + \hat{\varphi(g)})}{K_u(z)}
\end{align}
that holds for $z \in Z$, $u \in Z_{\psi}$, and $g \in G$.
The functions $\Lambda_u$ and $K_u$ can be chosen to depend measurably on $u$
(see \cite[Proposition 2]{lesigne} and \cite[Proposition 10.5]{fw}).

The foregoing discussion motivates the following definition (see \cite{cl} and \cite{fw} for the case $G = \Z$):
\begin{defn} \label{defn: CL cocycle}
	Let $\X$ be an ergodic system with Kronecker factor $Z$.
	A cocycle $\rho_g : Z \to S^1$ is a \emph{$(\varphi,\psi)$-Conze--Lesigne cocycle}
	($(\varphi,\psi)$-CL cocycle for short) if there are measurable functions
	$\Lambda : Z_{\psi} \times \left( Z/Z_{\varphi} \right) \to \hat{G}$
	and $K : Z_{\psi} \times Z \to S^1$ such that
	\begin{align} \label{eq: CL}
		\frac{\rho_{\varphi(g)}(z + u)}{\rho_{\varphi(g)}(z)}
		= \Lambda_u(z + Z_{\varphi})(g) \frac{K_u(z + \hat{\varphi(g)})}{K_u(z)}
	\end{align}
	for almost all $z \in Z$, $u \in Z_{\psi}$, and all $g \in G$.
	The set of all $(\varphi,\psi)$-CL cocycles is denoted by $CL_{\X}(\varphi,\psi)$.
\end{defn}

If $\varphi(g) = rg$ and $\psi(g) = sg$ with $r, s \in \Z$, we will denote $CL_{\X}(\varphi, \psi)$ by $CL_{\X}(r,s)$.

\begin{defn}
	Let $\X$ be an ergodic system with Kronecker factor $Z$.
	We call a function $f : X \to S^1$ a \emph{$(\varphi,\psi)$-CL function}
	if there is a cocycle $\rho \in CL_{\X}(\varphi,\psi)$ such that
	\begin{align*}
		f(T_gx) = \rho_g(z) f(x),
	\end{align*}
	where $z = \pi(x)$ for $\pi : X \to Z$ the projection onto the Kronecker factor.
\end{defn}
Since $CL_{\X}(\varphi,\psi)$ is a group, it follows immediately that the span of $(\varphi,\psi)$-CL functions
forms a subalgebra of $L^{\infty}(\mu)$.
We call the corresponding factor the \emph{$(\varphi,\psi)$-CL factor}, denoted by $\B_{CL(\varphi,\psi)}$.
By the computations in this section, we have reduced the general cocycle in Theorem \ref{thm: abelian ext}
to cocycles satisfying the Conze--Lesigne equation \eqref{eq: CL}.
We can therefore summarize the results of this section by the following theorem:
\begin{thm} \label{thm: CL characteristic}
	Let $\X$ be a normal ergodic system, and let $\{\varphi,\psi\}$ be an admissible pair of homomorphisms.
	Then
	\begin{align*}
		\UClim_{g \in G} & \ {T_{\varphi(g)} f_1 \cdot T_{\psi(g)} f_2 \cdot T_{(\varphi+\psi)(g)} f_3} \\
		= &~\UClim_{g \in G}{T_{\varphi(g)} \E{f_1}{\B_{CL(\varphi,\psi)}}
			\cdot T_{\psi(g)} \E{f_2}{\B_{CL(\varphi,\psi)}}
			\cdot T_{(\varphi+\psi)(g)} \E{f_3}{\B_{CL(\varphi,\psi)}}}
	\end{align*}
	in $L^2(\mu)$.
\end{thm}


\section{Limit formula for triple averages} \label{sec: limit formula 3}

The goal of this section is to prove a formula for the limit of the multiple ergodic averages
\begin{align*}
	\UClim_{g \in G}{f_1(T_{rg}x) f_2(T_{sg}x) f_3(T_{(r+s)g}x)},
\end{align*}
where $T$ is an ergodic action and $\{r, s, r+s\}$ is an admissible triple.
Our approach follows closely the method used by Host and Kra in \cite{hk} in establishing a formula for $\Z$-actions (see \cite[Theorem 12]{hk}).
In order to extend this method to our more general setting of countable discrete abelian groups, we need to restrict our attention to a class of systems called \emph{quasi-affine systems} (defined in Section \ref{sec: QA}).

\begin{restatable}{thm}{LimitFormula} \label{thm: limit formula}
	Let $G$ be a countable discrete abelian group.
	Let $r, s \in \Z$ such that $rG$, $sG$, and $(r \pm s)G$ have finite index in $G$.
	Let $k'_1 = - rs(r+s)$, $k'_2 = rs(r+s)$, and $k'_3 = -rs(s-r)$.
	Set $D := \gcd(k'_1, k'_2, k'_3) = rs \gcd(r+s, s-r)$ and $k_i = \frac{k'_i}{D}$.
	Let $b_1, b_2, b_3 \in \Z$ so that $\sum_{i=1}^3{k_ib_i} = 1$.
	
	Let $\X = \mathbf{Z} \times_{\sigma} H$ be an ergodic quasi-affine system.
	There is a function $\psi : Z \times Z \to H$ such that $\psi(0,z) = 0$, $t \mapsto \psi(t,\cdot)$
	is a continuous map from $Z$ to $\M(Z,H)$,\footnote{We denote by $\M(Z,H)$ the space of measurable functions
		$Z \to H$ in the topology of convergence in measure.}
	and for every $f_1, f_2, f_3 \in L^{\infty}(\mu)$,
	\begin{align} \label{eq: limit formula}
		\UClim_{g \in G}{f_1(T_{rg}x) f_2(T_{sg}x) f_3(T_{(r+s)g}x)}
		= \int_{Z \times H^2}{\prod_{i=1}^3{f_i(z + a_it, h + a_iu + a_i^2v + b_i \psi(t,z))}~du~dv~dt},
	\end{align}
	in $L^2(\mu)$, where $x = (z,h) \in Z \times H$,
	and $a_1 = r, a_2 =s, a_3 = r+s$.
\end{restatable}


\subsection{Cohomology for abelian group extensions}

To enable our usage of cocycles, we introduce cohomology for abelian group extensions.

\begin{defn}
	Let $\X = \left( X, \B, \mu, (T_g)_{g \in G} \right)$ be a measure-preserving system
	and $(H, +)$ a compact abelian group.
	\begin{enumerate}[1.]
		\item	Denote by $Z^1_G(\X, H)$ the set of \emph{cocycles},
		i.e. measurable functions $\sigma : G \times X \to H$ satisfying the cocycle equation
		\begin{align*}
			\sigma_{g+h}(x) = \sigma_g(T_h x) + \sigma_h(x).
		\end{align*}
		\item	A cocycle $\sigma \in Z^1_G(\X,H)$ is a \emph{coboundary}
		if there is a measurable function $F : X \to H$ such that
		\begin{align*}
			\sigma_g(x) = F \left( T_gx \right) - F(x).
		\end{align*}
		We denote the set of all coboundaries by $B^1_G(\X,H)$.
		\item	Two cocycles $\sigma$ and $\sigma'$ are \emph{cohomologous}, denoted $\sigma \sim \sigma'$,
		if $\sigma - \sigma' \in B^1_G(\X,H)$.
	\end{enumerate}
\end{defn}

Recall that, given a cocycle $\sigma \in Z^1_G(\X,H)$, we denote by $\X \times_{\sigma} H$ the $G$-system
\begin{align*}
	T^{\sigma}_g(x,y) = (T_gx, y+\sigma_g(x))
\end{align*}
for $(x,y) \in X \times H$ and $g \in G$.

The set of cocycles $Z^1_G(\X,H)$ and the set of coboundaries $B^1_G(\X,H)$ both form groups
under pointwise addition $(\sigma + \sigma')(g,x) = \sigma(g,x) + \sigma'(g,x)$.
Moreover, $B^1_G(\X, H)$ is a subgroup of $Z^1_G(\X, H)$, and if $\sigma \sim \sigma'$,
then $\X \times_{\sigma} H \cong \X \times_{\sigma'} H$.
As a result, the isomorphism class of the extension $\X \times_{\sigma} H$ depends only
on the congruence class of $\sigma$ in the \emph{cohomology group} $H^1_G(\X, H) := Z^1_G(\X,H)/B^1_G(\X,H)$.

Since we are ultimately interested in studying the Conze--Lesigne factor,
which is an extension of the Kronecker factor by a compact abelian group,
we need the following notions for cocycles defined on the Kronecker factor:

\begin{defn}
	Suppose $\mathbf{Z}$ is an ergodic Kronecker system.\footnote{By this, we mean that $\mathbf{Z}$
		is measurably isomorphic to an ergodic action of $G$ by rotations on a compact abelian group.}
	Let $H$ be a compact abelian group, and let $\sigma : G \times Z \to H$ be a cocycle.
	\begin{enumerate}[1.]
		\item	$\sigma$ is \emph{ergodic} if $\mathbf{Z} \times_{\sigma} H$ is ergodic.
		\item $\sigma$ is \emph{weakly mixing} if the extension by $H$ over $\sigma$ is relatively weakly mixing.
		That is, $\mathbf{Z} \times_{\sigma} H$ is ergodic with Kronecker factor equal to $\mathbf{Z}$.
	\end{enumerate}
\end{defn}

It is helpful to study cocycles $\sigma \in Z^1_G(\X,H)$ by considering the family of cocycles
$\chi \circ \sigma : G \times X \to S^1$ for $\chi \in \hat{H}$.
Here we will use multiplicative notation, rather than the additive notation above.
For $S^1$-valued cocycles, we introduce an additional definition:

\begin{defn}
	Let $\X = \left( X, \B, \mu, (T_g)_{g \in G} \right)$ be a measure-preserving system.
	A cocycle $\rho$ is \emph{cohomologous to a character} if there is a character $\gamma \in \hat{G}$
	such that $\rho_g(x) \sim \gamma(g)$.
	That is, for some measurable function $F : X \to S^1$,
	\begin{align*}
		\rho_g(x) = \gamma(g) \frac{F \left( T_gx \right)}{F(x)}
	\end{align*}
	for every $g \in G$ and almost every $x \in X$.
\end{defn}

The following proposition characterizes dynamical properties of a cocycle $\sigma : G \times Z \to H$
in terms of the behavior of the cocycles $\chi \circ \sigma$ for $\chi \in \hat{H}$:

\begin{prop} \label{prop: HK P4}
	Suppose $\mathbf{Z}$ is an ergodic Kronecker system.
	Let $H$ be a compact abelian group, and let $\sigma : G \times Z \to H$ be a cocycle.
	\begin{enumerate}[(1)]
		\item	$\sigma$ is ergodic if and only if, for every $\chi \in \hat{H} \setminus \{1\}$,
		$\chi \circ \sigma$ is not a coboundary.
		\item	$\sigma$ is weakly mixing if and only if, for every $\chi \in \hat{H} \setminus \{1\}$,
		$\chi \circ \sigma$ is not cohomologous to a character.
	\end{enumerate}
\end{prop}
\begin{proof}
	(1).
	Suppose $\sigma$ is ergodic and $\chi \circ \sigma$ is a coboundary for some $\chi \in \hat{H}$.
	We want to show $\chi = 1$.
	Write
	\begin{align} \label{eq: chi sigma coboundary}
		\chi \left( \sigma_g(z) \right) = \frac{F \left( z + \hat{g} \right)}{F(z)}
	\end{align}
	for some function $F : Z \to S^1$.
	Define $f : Z \times H \to S^1$ by $f(z,x) := \overline{F}(z) \chi(x)$.
	The coboundary identity \eqref{eq: chi sigma coboundary} implies that $f$ is invariant for
	$\mathbf{Z} \times_{\sigma} H$.
	Since $\sigma$ is ergodic, it follows that $f$ is constant.
	Therefore, $\chi = 1$.
	
	Conversely, suppose that for every $\chi \in \hat{H} \setminus \{1\}$, $\chi \circ \sigma$
	is not a coboundary.
	Suppose $f$ is invariant for $\mathbf{Z} \times_{\sigma} H$.
	We will show $f$ is constant.
	Write $f(z,x) = \sum_{\chi \in \hat{H}}{c_{\chi}(z) \chi(x)}$.
	Invariance means
	\begin{align*}
		\sum_{\chi \in \hat{H}}{c_{\chi} \left( z + \hat{g} \right) \chi(\sigma_g(z)) \chi(x)}
		= \sum_{\chi \in \hat{H}}{c_{\chi}(z) \chi(x)}.
	\end{align*}
	So, for every $\chi \in \hat{H}$,
	\begin{align*}
		c_{\chi} \left( z + \hat{g} \right) \chi \left( \sigma_g(z) \right) = c_{\chi}(z).
	\end{align*}
	By ergodicity of $\mathbf{Z}$, $|c_{\chi}|$ is constant, say $|c_{\chi}| = C_{\chi}$.
	Thus, if $c_{\chi} \ne 0$, then
	\begin{align*}
		\chi \left( \sigma_g(z) \right)
		= \frac{C_{\chi}^{-1} \overline{c_{\chi}} \left( z+ \hat{g} \right)}{C_{\chi}^{-1} \overline{c_{\chi}}(z)},
	\end{align*}
	so $\chi \circ \sigma_g$ is a coboundary.
	It follows that $c_{\chi} = 0$ for $\chi \ne 1$, so $f(z,x) = c_1(z)$.
	But $\mathbf{Z}$ is ergodic, so $f$ is constant.
	
	\bigskip
	
	(2).
	Suppose $\sigma$ is weakly mixing.
	Let $\chi \in \hat{H}$, and assume $\chi \circ \sigma$ is cohomologous to a character.
	We will show $\chi = 1$.
	Let $F : Z \to S^1$ and $\gamma \in \hat{G}$ so that
	\begin{align*}
		\chi \left( \sigma_g(z) \right) = \gamma(g)\frac{F \left( z + \hat{g} \right)}{F(z)}.
	\end{align*}
	Then $f(z,x) := \overline{F}(z) \chi(x)$ is an eigenfunction of $\mathbf{Z} \times_{\sigma} H$
	with eigenvalue $\gamma$.
	Since $\sigma$ is weakly mixing, it follows that $f$ is measurable over $\mathbf{Z}$.
	Therefore, $\chi$ is constant, so $\chi = 1$.
	
	Conversely, suppose that for every $\chi \in \hat{H} \setminus \{1\}$,
	$\chi \circ \sigma$ is not cohomologous to a character.
	Suppose $f$ is an eigenfunction of $\mathbf{Z} \times_{\sigma} H$ with eigenvalue $\gamma \in \hat{G}$.
	We will show that $f$ is measurable over $Z$.
	Write $f(z,x) = \sum_{\chi \in \hat{H}}{c_{\chi}(z) \chi(x)}$.
	Then
	\begin{align*}
		\sum_{\chi \in \hat{H}}{c_{\chi} \left( z + \hat{g} \right) \chi(\sigma_g(z)) \chi(x)}
		= \gamma(g) \sum_{\chi \in \hat{H}}{c_{\chi}(z) \chi(x)}.
	\end{align*}
	Hence, for every $\chi \in \hat{H}$,
	\begin{align*}
		c_{\chi} \left( z + \hat{g} \right) \chi \left( \sigma_g(z) \right) = \gamma(g) c_{\chi}(z),
	\end{align*}
	so $\chi \circ \sigma \sim \gamma$ for any $\chi \in \hat{H}$ with $c_{\chi} \ne 0$.
	It follows that $f(z,x) = c_1(z)$.
\end{proof}

The characterization of ergodicity in Proposition \ref{prop: HK P4} motivates the following definition:
\begin{defn}
	Let $\mathbf{Z}$ be an ergodic Kronecker system, $H$ a compact abelian group,
	and $\sigma : G \times Z \to H$ a cocycle.
	The \emph{Mackey group associated to $\sigma$} is the subgroup
	\begin{align*}
		M := \left\{ x \in H : \chi(x) = 1~\text{for all}~\chi \in \hat{H}
		~\text{such that}~\chi \circ \sigma~\text{is a coboundary} \right\}
	\end{align*}
\end{defn}

Note that the annihilator\footnote{The \emph{annihilator} of a set $A$ in a group $H$ is the set
	$A^{\perp} := \left\{ \chi \in \hat{H} : \chi(a) = 1~\text{for every}~a \in A \right\}$.} of $M$
is given by $M^{\perp} = \left\{ \chi \in \hat{H} : \chi \circ \sigma~\text{is a coboundary} \right\}$.
Using the isomorphism $M^{\perp} = \hat{H/M}$,
Proposition \ref{prop: HK P4} says that $\sigma$ is ergodic if and only if $M = H$ if and only if $M^{\perp} = \{1\}$.
The following proposition refines this criterion to describe the splitting of $L^2(Z \times H)$ into invariant and ergodic functions:

\begin{prop} \label{prop: HK P5}
	Let $\mathbf{Z}$ be an ergodic Kronecker system, $H$ a compact abelian group,
	and $\sigma : G \times Z \to H$ a cocycle.
	Let $M$ be the Mackey group associated to $\sigma$.
	Let $f \in L^2(Z \times H)$.
	Suppose for every $\chi \in M^{\perp}$,
	\begin{align*}
		\int_H{f(z,x) \overline{\chi(x)}~dx} = 0
	\end{align*}
	for almost every $z \in Z$.
	Then
	\begin{align*}
		\UClim_{g \in G}{f \left( z + \hat{g}, x + \sigma_g(z) \right)} = 0
	\end{align*}
	in $L^2(Z \times H)$.
\end{prop}
\begin{proof}
	Let $T^{\sigma}_g(z,x) := \left( z + \hat{g}, x + \sigma_g(z) \right)$.
	Write $f = f_1 + f_2$, where $f_1$ is $T^{\sigma}$-invariant and $f_2$ is in the ergodic subspace for $T^{\sigma}$.
	We want to show $f_1 = 0$.
	We can expand $f_1(z,x) = \sum_{\chi \in \hat{H}}{c_{\chi}(z) \chi(x)}$.
	Using $T^{\sigma}$-invariance, we have
	\begin{align*}
		\sum_{\chi \in \hat{H}}{c_{\chi}\left( z + \hat{g} \right) \chi \left( \sigma_g(z) \right) \chi(x)}
		= \sum_{\chi \in \hat{H}}{c_{\chi}(z) \chi(x)}.
	\end{align*}
	Thus,
	\begin{align*}
		c_{\chi}\left( z + \hat{g} \right) \chi \left( \sigma_g(z) \right) 
		= c_{\chi}(z)
	\end{align*}
	for every $\chi \in \hat{H}$.
	If $c_{\chi} \ne 0$, this implies $\chi \circ \sigma$ is a coboundary, so $\chi \in M^{\perp}$.
	But for $\chi \in M^{\perp}$,
	\begin{align*}
		c_{\chi}(z) = \int_H{f(z,x) \overline{\chi(x)}~dx} = 0.
	\end{align*}
	Therefore, $f_1 = 0$.
\end{proof}


\subsection{Mackey groups for admissible triples} \label{sec: Mackey triples}

When dealing with multiple recurrence, one needs a modified version of the Mackey group.
First we define several notations.
Let $\mathbf{Z}$ be an ergodic Kronecker system and $H$ a compact abelian group.
Fix an admissible triple $\{a_1, a_2, a_3\} \subseteq \Z$ and a cocycle $\sigma : G \times Z \to H$.
Let $W \subseteq Z^3$ be the subgroup
\begin{align*}
	W := W_{a_1,a_2,a_3} = \left\{ (z + a_1t, z + a_2t, z + a_3t) : z, t \in Z \right\}.
\end{align*}
For $g \in G$, let $\alpha_g := (a_1\hat{g}, a_2\hat{g}, a_3\hat{g}) \in W$,
and define $\tilde{S}_g : W \to W$ by $\tilde{S}_g(w) = w + \alpha_g$.
Let $\tilde{\sigma} : G \times W \to H^3$ be the cocycle
\begin{align*}
	\tilde{\sigma}_g(w_1,w_2,w_3) = \left( \sigma_{a_1g}(w_1), \sigma_{a_2g}(w_2), \sigma_{a_3g}(w_3) \right).
\end{align*}
Finally, let $M = M(a_1, a_2, a_3) \subseteq H^3$ be the Mackey group associated to the cocycle $\tilde{\sigma}$.

The action $\tilde{S}$ is not ergodic, but its ergodic decomposition is easy to describe.
For each $z \in Z$, let $W_z := \left\{ (z + a_1t, z + a_2t, z + a_3t) : t \in Z \right\}$.
Note that $W_0$ is an $\tilde{S}$-invariant subgroup of $Z^3$,
and $(W_0, \tilde{S})$ is uniquely ergodic by Lemma \ref{lem: compact group rotations}.
Now, for each $z \in Z$, $(W_z, \tilde{S})$ is an isomorphic topological dynamical system
supporting a unique invariant measure $m_z$, which is a shift of the Haar measure $m_0$ on $W_0$.
We claim that the Haar measure $m_W$ on $W$ can be decomposed as
\begin{align*}
	m_W = \int_Z{m_z~dz}.
\end{align*}
To see this, let $f : W \to \C$ be a continuous function and $w_0 = (u + a_1v, u + a_2v, u + a_3v) \in W$.
Then
\begin{align*}
	\int_Z{\int_W{f(w+w_0)~dm_z(w)}~dz}
	& = \int_Z{\int_W{f \left( w + (u,u,u) \right)~dm_z(w)}~dz} & (m_z~\text{is}~W_0\text{-invariant}) \\
	& = \int_Z{\int_W{f(w)~dm_{z+u}(w)}~dz} & (W_z + (u,u,u) = W_{z+u}) \\
	& = \int_Z{\int_W{f(w)~dm_z(w)}~dz}. & (\text{Haar measure on}~Z~\text{is shift-invariant})
\end{align*}
Thus, the measure $\int_Z{m_z~dz}$ is shift-invariant on $W$.
By the uniqueness of Haar measure, $m_W = \int_Z{m_z~dz}$ as claimed.

\begin{rem}
	The decomposition given here is not the standard ergodic decomposition,
	since the cosets $W_z$ may coincide for different values of $z \in Z$.
	However, this form will be more convenient for our purposes.
\end{rem}

For each $z \in Z$, $\tilde{\sigma}$ is defined $m_z$-a.e., so we can define the Mackey groups
\begin{align*}
	M_z^{\perp}
	:= \left\{ \tilde{\chi} \in \hat{H}^3 : \tilde{\chi} \circ \tilde{\sigma}
	~\text{is a coboundary for}~(W_z, m_z, \tilde{S}) \right\}.
\end{align*}
Ergodicity of the base system $\mathbf{Z}$ allows us to prove the following:

\begin{prop} \label{prop: constant Mackey}
	For almost every $z \in Z$, $M_z = M$.
\end{prop}
\begin{proof}
	Let $\stackrel{\triangle}{g} = \left( \hat{g}, \hat{g}, \hat{g} \right)$,
	and let $\stackrel{\triangle}{S}_g(w) = w + \stackrel{\triangle}{g}$.
	Then for each $g \in G$, $\stackrel{\triangle}{S}_g$ gives an isomorphism
	$(W_z, m_z, \tilde{S}) \cong (W_{z+\hat{g}}, m_{z+\hat{g}}, \tilde{S})$.
	Moreover, by the cocycle equation,
	\begin{align*}
		\sigma_{a_ig} \left( w_i + \hat{h} \right)
		= \sigma_{a_ig}(w_i) + \sigma_h \left( w_i + a_i\hat{g} \right) - \sigma_h(w_i),
	\end{align*}
	so
	\begin{align} \label{eq: shift cohomologous}
		\tilde{\sigma}_g \left( \stackrel{\triangle}{S}_hw \right)
		= \tilde{\sigma}_g(w) + \stackrel{\triangle}{\sigma}_h \left( \tilde{S}_gw \right)
		- \stackrel{\triangle}{\sigma}_h(w),
	\end{align}
	where $\stackrel{\triangle}{\sigma}_h(w) = \left( \sigma_h(w_1), \sigma_h(w_2), \sigma_h(w_3) \right)$.
	
	Now suppose $\tilde{\chi} \in M_z^{\perp}$.
	Then there is a function $F : W_z \to S^1$ such that
	\begin{align*}
		\tilde{\chi} \left( \tilde{\sigma}_g(w) \right) = \frac{F \left( \tilde{S}_gw \right)}{F(w)}
	\end{align*}
	for every $g \in G$ and almost every $w \in W_z$.
	Then by \eqref{eq: shift cohomologous}, we have
	\begin{align*}
		\tilde{\chi} \left( \tilde{\sigma}_g \left( \stackrel{\triangle}{S}_h w \right) \right)
		& = \frac{F \left( \tilde{S}_gw \right)}{F(w)}
		\frac{\chi \left( \stackrel{\triangle}{\sigma}_h \left( \tilde{S}_gw \right) \right)}
		{\chi \left( \stackrel{\triangle}{\sigma}_h(w) \right)}.
	\end{align*}
	Letting $\Phi(w) := F(w) \chi \left( \stackrel{\triangle}{\sigma}_h(w) \right)$, this simplifies to
	\begin{align*}
		\tilde{\chi} \left( \tilde{\sigma}_g \left( \stackrel{\triangle}{S}_h w \right) \right)
		= \frac{\Phi \left( \tilde{S}_gw \right)}{\Phi(w)},
	\end{align*}
	so $\tilde{\chi} \in M_{z + \hat{h}}^{\perp}$.
	We can repeat this argument to show that
	$M_{z + \hat{h}}^{\perp} \subseteq M_{z + \hat{h} - \hat{h}}^{\perp} = M_z^{\perp}$.
	Thus, $M_z = M_{z + \hat{h}}$ for every $h \in G$.
	
	Now, for each $\tilde{\chi} \in \hat{H}^3$, $\{z \in Z : \tilde{\chi} \in M_z^{\perp}\}$ is measurable
	(see \cite[Proposition 2]{lesigne})
	and therefore has measure either $0$ or $1$ by ergodicity of $\mathbf{Z}$.
	Since $\hat{H}^3$ is countable, it follows that $M_z$ is constant a.e., say $M_z = N$.
	It remains to check $M = N$.
	
	Let $\tilde{\chi} \in N^{\perp}$. Then $\tilde{\chi} \circ \tilde{\sigma}$ is a coboundary for $(W_z, m_z, \tilde{S})$
	for almost every $z \in Z$.
	That is, for almost every $z \in Z$, there is a function $F_z : W_z \to S^1$ such that
	\begin{align*}
		\tilde{\chi} \left( \tilde{\sigma}_g(w) \right) = \frac{F_z \left( \tilde{S}_gw \right)}{F_z(w)}
	\end{align*}
	for $m_z$-almost every $w \in W_z$.
	Define $F \left( z + a_1t, z + a_2t, z + a_3t \right) := F_z(z + a_1t, z + a_2t, z + a_3t)$.
	We can ensure this is well-defined by taking $F_z = F_{z'}$ whenever $W_z = W_{z'}$.
	Moreover, we can ensure that the map $z \mapsto F_z$ is measurable (see \cite[Proposition 2]{lesigne})
	so that $F$ is a measurable function.
	Then
	\begin{align*}
		\tilde{\chi} \left( \tilde{\sigma}_g(w) \right) = \frac{F \left( \tilde{S}_gw \right)}{F(w)}
	\end{align*}
	for almost every $w \in W$.
	Hence, $\tilde{\chi} \in M^{\perp}$.
	
	Now suppose $\tilde{\chi} \in M^{\perp}$.
	Then there is a function $F : W \to S^1$ such that
	\begin{align*}
		\tilde{\chi} \left( \tilde{\sigma}_g(w) \right) = \frac{F \left( \tilde{S}_gw \right)}{F(w)}
	\end{align*}
	for almost every $w \in W$.
	By Fubini's theorem, it follows that this equation holds for $m_z$-almost every $w \in W_z$
	for almost every $z \in Z$.
	That is, $\tilde{\chi} \in M_z^{\perp}$ a.e., so $\tilde{\chi} \in N^{\perp}$.
\end{proof}

We use this to prove a version of Proposition \ref{prop: HK P5} for $M(a_1, a_2, a_3)$:
\begin{prop} \label{prop: Mackey ergodicity}
	Assume $\sigma$ is weakly mixing.
	Let $M = M(a_1, a_2, a_3)$ be the Mackey group associated to $\tilde{\sigma}$.
	Let $f_1, f_2, f_3 \in L^{\infty}(Z \times H)$.
	Suppose for every $\tilde{\chi} \in M^{\perp}$,
	\begin{align} \label{eq: orthogonal to annihilators}
		\int_{H^3}{\prod_{i=1}^3{f_i(w_i, x_i)} \overline{\tilde{\chi}(x)}~dx} = 0
	\end{align}
	for almost every $w = (w_1, w_2, w_3) \in W$.
	Then
	\begin{align*}
		\UClim_{g \in G}{\prod_{i=1}^3{f_i \left( z + a_i\hat{g}, x + \sigma_{a_ig}(z) \right)}} = 0
	\end{align*}
	in $L^2(Z \times H)$.
\end{prop}
\begin{proof}
	Let $T^{\sigma}_g(z,h) := \left( z + \hat{g}, h + \sigma_g(z) \right)$,
	and let $\tilde{T}_g(w,x) = \left( \tilde{S}_gw, x + \tilde{\sigma}_g(w) \right)$.
	Let $F(w,x) := \prod_{i=1}^3{f_i(w_i, x_i)}$.
	First, we will use Propositions \ref{prop: HK P5} and \ref{prop: constant Mackey} together to show that
	$F$ is in the ergodic subspace of $L^2(W \times H^3)$ for $\tilde{T}$.
	
	Let $E_1 := \{z \in Z : M_z = M\}$.
	Let $E_2 := \left\{ z \in Z : \text{\eqref{eq: orthogonal to annihilators} holds for}~m_z\text{-a.e.}~w \in W_z \right\}$.
	Finally, let $E_3 := \left\{ z \in Z : F \in L^{\infty}(W_z \times H^3) \right\}$.
	By Proposition \ref{prop: constant Mackey}, $E_1$ has full measure in $Z$.
	By Fubini's theorem and the hypothesis, $E_2$ has full measure.
	Also by Fubini's theorem, $E_3$ has full measure.
	Hence, the set $E := E_1 \cap E_2 \cap E_3 \subseteq Z$ has full measure.
	
	Now let $z \in E$.
	Since $z \in E_3$, one has $F \in L^{\infty}(W_z \times H^3) \subseteq L^2(W_z \times H^3)$.
	Let $\tilde{\chi} \in M_z^{\perp}$.
	Since $z \in E_1$, we have $\tilde{\chi} \in M^{\perp}$.
	Therefore, since $z \in E_2$, \eqref{eq: orthogonal to annihilators} holds
	for $\tilde{\chi}$ and $m_z$-a.e. $w \in W_z$.
	Thus, $F$ satisfies the hypotheses of Proposition \ref{prop: HK P5}, so
	\begin{align} \label{eq: ergodic slices}
		\UClim_{g \in G}{\tilde{T}_gF} = 0
	\end{align}
	in $L^2(W_z \times H^3)$.
	
	Fix a F{\o}lner sequence $(F_N)_{N \in \N}$ in $G$,
	and set $A_N := \frac{1}{|F_N|} \sum_{g \in F_N}{\tilde{T}_gF} \in L^2(W \times H^3)$.
	We want to show $\norm{L^2(W \times H^3)}{A_N} \to 0$.
	By decomposing $m_W = \int_Z{m_z~dz}$, we have
	\begin{align*}
		\norm{L^2(W \times H^3)}{A_N}^2
		= \int_W{\int_{H^3}{\left| A_N(w,x) \right|^2~dx}~dw}
		= \int_Z{\int_{W_z}{\int_{H^3}{\left| A_N(w,x) \right|^2~dx}~dm_z(w)}~dz}.
	\end{align*}
	Now by \eqref{eq: ergodic slices},
	\begin{align*}
		\int_{W_z}{\int_{H^3}{\left| A_N(w,x) \right|^2~dx}~dm_z(w)} \to 0
	\end{align*}
	for every $z \in E$ (so for $m_Z$-a.e. $z \in Z$).
	By the dominated convergence theorem, it follows that $\norm{L^2(W \times H^3)}{A_N} \to 0$ as desired.
	
	\bigskip
	
	We will now use the van der Corput trick (Lemma \ref{lem: vdC}).
	Let $u_g := \prod_{i=1}^3{T^{\sigma}_{a_ig}f_i}$.
	Observe that $\mathbf{Z} \times_{\sigma} H$ has Kronecker factor $\mathbf{Z}$
	because $\sigma$ is weakly mixing.
	Therefore, we can use Lemma \ref{lem: diag meas} (see also Remark \ref{random remark on lem: diag meas}) to compute
	\begin{align*}
		\gamma_h := \UClim_{g \in G}{\innprod{u_{g+h}}{u_g}}
		= \int_{W}{\prod_{i=1}^3{\E{\overline{f_i} \cdot T^{\sigma}_{a_ih}f_i}{Z}(w_i)}~dm_W(w_1,w_2,w_3)}
	\end{align*}
	Now, by Fubini's theorem, we have
	\begin{align*}
		\E{\overline{F} \cdot \tilde{T}_hF}{W}(w) & = \int_{H^3}{\left( \overline{F} \cdot \tilde{T}_hF \right)(w,x)~dx}
		= \int_{H^3}{\prod_{i=1}^3{\left( \overline{f_i} \cdot T^{\sigma}_{a_ih}f_i \right)(w_i,x_i)}~dx} \\
		& = \prod_{i=1}^3{\left( \int_H{\left( \overline{f_i} \cdot T^{\sigma}_{a_ih}f_i \right) (w_i, x_i)~dx_i} \right)}
		= \prod_{i=1}^3{\E{\overline{f_i} \cdot T^{\sigma}_{a_ih}f_i}{Z}}(w_i),
	\end{align*}
	so
	\begin{align*}
		\gamma_h = \int_W{\E{\overline{F} \cdot \tilde{T}_hF}{W}~dm_W}
		= \int_{W \times H^3}{\overline{F} \cdot \tilde{T}_hF~dm_{W \times H^3}}
		= \innprod{F}{\tilde{T}_hF}.
	\end{align*}
	But we showed that $F$ is in the ergodic subspace of $L^2(W \times H^3)$, so
	$\UClim_{h \in G}{\tilde{T}_hF} = 0$ in $L^2(W \times H^3)$.
	Thus, $\UClim_{h \in G}{\gamma_h} = 0$.
	By Lemma \ref{lem: vdC}, $\UClim_{g \in G}{u_g} = 0$ in $L^2(Z \times H)$ as claimed.
\end{proof}


\subsection{Analytic characterizations of cocycles}

In this section, we characterize classes of cocycles in terms of various analytic properties.
These characterizations will be useful in proving the limit formula in Theorem \ref{thm: limit formula}.
As usual, throughout this section we let $G$ be a countable discrete abelian group.
For a Kronecker system $\mathbf{Z} = \left( Z, \B, m_Z, (T_g)_{g \in G} \right)$, we let $\hat{g} \in Z$ denote the group element such that $T_gz = z + \hat{g}$.
The $G = \Z$ case of the results in the section are proved in \cite{hk}.
We follow their approach, making several modifications in order to extend the results to our more general setting.

\begin{lem} \label{lem: HK L1}
	Let $\mathbf{Z}$ be a Kronecker system and $\rho : G \times Z \to S^1$ a cocycle.
	Suppose $(g_n)_{n \in \N}$ is a sequence in $G$ such that $\hat{g}_n \to 0$ in $Z$.
	Then for every $h \in G$,
	\begin{align*}
		\norm{L^2(Z)}{\rho_{g_n} \left( z + \hat{h} \right) - \rho_{g_n}(z)} \to 0.
	\end{align*}
\end{lem}
\begin{proof}
	Expand the cocycle equation two ways:
	\begin{align*}
		\rho_{g_n+h}(z) & = \rho_{g_n}(z) \rho_h \left( z + \hat{g}_n \right),
		\intertext{and}
		\rho_{g_n+h}(z) & = \rho_h(z) \rho_{g_n} \left( z + \hat{h} \right).
	\end{align*}
	Dividing, we get
	\begin{align*}
		\frac{\rho_{g_n} \left( z + \hat{h} \right)}{\rho_{g_n}(z)}
		= \frac{\rho_h \left( z + \hat{g}_n \right)}{\rho_h(z)}.
	\end{align*}
	Since $\hat{g}_n \to 0$ in $Z$ and the operators $U_tf(z) = f(z+t)$ define a continuous action of $Z$ on $L^2(Z)$, we have
	\begin{align*}
		\frac{\rho_h \left( z + \hat{g}_n \right)}{\rho_h(z)} \to 1
	\end{align*}
	in $L^2(Z)$.
	The result follows immediately.
\end{proof}

\begin{prop} \label{prop: HK P2}
	Let $\mathbf{Z}$ be a Kronecker system and $\rho : G \times Z \to S^1$ a cocycle.
	The following are equivalent:
	\begin{enumerate}[(i)]
		\item	$\rho$ is a coboundary;
		\item	for any sequence $(g_n)_{n \in \N}$ in $G$ with $\hat{g}_n \to 0$ in $Z$,
		we have $\rho_{g_n}(z) \to 1$ in $L^2(Z)$.
	\end{enumerate}
\end{prop}
\begin{proof}
	Suppose (i) holds.
	Let $F : Z \to S^1$ so that
	\begin{align*}
		\rho_g(z) = \frac{F \left( z + \hat{g} \right)}{F(z)}.
	\end{align*}
	Suppose $\hat{g}_n \to 0$.
	Then, since the operators $U_tf(z) = f(z+t)$ define a continuous action of $Z$ on $L^2(Z)$, we have
	\begin{align*}
		\rho_{g_n}(z) = \frac{F \left( z + \hat{g}_n \right)}{F(z)} \to 1.
	\end{align*}
	That is, (ii) holds.
	
	Conversely, suppose (ii) holds.
	Let $K := \overline{\left\{ \hat{g} : g \in G \right\}}$.
	Fix a translation-invariant metric $d$ on $K$.
	Let $\eps > 0$.
	By (ii), there exists $\delta > 0$ such that, if $d \left( \hat{g}, 0 \right) < \delta$,
	then $\left\| \rho_g - 1 \right\|_{L^2(Z)} < \eps$.
	Let $g \in G$ with $d \left( \hat{g}, 0 \right) < \delta$.
	Then for $h \in G$, we have
	\begin{align} \label{eq: Cauchy condition}
		\norm{L^2(Z)}{\rho_{g+h}(z) - \rho_h(z)}
		& = \norm{L^2(Z)}{\rho_g \left( z + \hat{h} \right) - 1} & (\text{cocycle equation}) \nonumber \\
		& = \norm{L^2(Z)}{\rho_g(z) - 1} & (\text{translation-invariance of Haar measure}) \\
		& < \eps. \nonumber
	\end{align}
	
	Suppose $(g_n)_{n \in \N}$ is a sequence in $G$ such that $\hat{g}_n$ converges in $Z$.
	By \eqref{eq: Cauchy condition}, the sequence $\left( \rho_{g_n} \right)_{n \in \N}$ is Cauchy in $L^2(Z)$,
	so $\rho_{g_n}$ converges.
	Thus, there is a unique continuous function map $K \ni t \mapsto \varphi_t \in L^2(Z)$
	such that $\varphi_{\hat{g}} = \rho_g$ a.e.
	Applying the cocycle equation and using continuity of this map, for every $g \in G$ and $t \in K$, we have
	\begin{align} \label{eq: extended cocycle eq}
		\varphi_{t+\hat{g}}(z) = \varphi_t(z) \rho_g \left( z + t \right)
	\end{align}
	for almost every $z \in Z$.
	
	By Fubini's theorem, for almost every $z \in Z$,
	equation \eqref{eq: extended cocycle eq} holds for almost every $t \in K$.
	For any such $z \in Z$, set $F_z(t) := \varphi_t(z)$.
	Then by \eqref{eq: extended cocycle eq}, we have
	\begin{align*}
		\frac{F_z \left( t+\hat{g} \right)}{F_z(t)}
		= \frac{\varphi_{t+\hat{g}}(z)}{\varphi_t(z)}
		= \rho_g \left( z + t \right)
	\end{align*}
	for almost every $t \in K$.
	That is, for almost every $z \in Z$, $\rho$ is a coboundary on the ergodic component $z + K$.
	By choosing the functions $F_z$ to depend measurably on $z$ (see \cite[Proposition 2]{lesigne}),
	it follows that $\rho$ is a coboundary for $\mathbf{Z}$.
\end{proof}

\begin{prop} \label{prop: HK P3}
	Let $\mathbf{Z}$ be an ergodic Kronecker system and $\rho : G \times Z \to S^1$ a cocycle.
	The following are equivalent:
	\begin{enumerate}[(i)]
		\item	$\rho$ is cohomologous to a character;
		\item	for any sequence $(g_n)_{n \in \N}$ in $G$ with $\hat{g}_n \to 0$ in $Z$,
		there is a sequence $(c_n)_{n \in \N}$ in $S^1$ such that
		$c_n\rho_{g_n}(z) \to 1$ in $L^2(Z)$;
		\item	for every $t \in Z$,
		\begin{align*}
			\frac{\rho_g(z+t)}{\rho_g(z)}
		\end{align*}
		is a coboundary;
		\item	there is a Borel set $A \subseteq Z$ with $m_Z(A) > 0$ such that
		\begin{align*}
			\frac{\rho_g(z+t)}{\rho_g(z)}
		\end{align*}
		is a coboundary for every $t \in A$.
	\end{enumerate}
\end{prop}
\begin{proof}
	(i)$\implies$(ii).
	Let $F : Z \to S^1$ such that
	\begin{align*}
		\rho_g(z) = \gamma(g) \frac{F \left( z + \hat{g} \right)}{F(z)}.
	\end{align*}
	Suppose $\hat{g}_n \to 0$ in $Z$.
	Let $c_n = \overline{\gamma(g_n)} \in S^1$.
	Then
	\begin{align*}
		c_n \rho_{g_n}(z) = \frac{F \left( z + \hat{g}_n \right)}{F(z)} \to 1
	\end{align*}
	since translation is continuous on $L^2(Z)$.
	
	\bigskip
	
	(ii)$\implies$(iii).
	Fix $t \in Z$, and define $\sigma : G \times Z \to S^1$ by
	\begin{align*}
		\sigma_g(z) := \frac{\rho_g(z+t)}{\rho_g(z)}.
	\end{align*}
	Let $\hat{g}_n \to 0$ in $Z$.
	By Proposition \ref{prop: HK P2}, it suffices to show $\sigma_{g_n}(z) \to 1$ in $L^2(Z)$.
	Using property (ii), let $(c_n)_{n \in \N}$ be a sequence in $S^1$ such that $c_n\rho_{g_n}(z) \to 1$ in $L^2(Z)$.
	Since the Haar measure on $Z$ is translation-invariant, we also have $c_n\rho_{g_n}(z+t) \to 1$ in $L^2(Z)$.
	Hence,
	\begin{align*}
		\sigma_{g_n}(z) = \frac{c_n\rho_{g_n}(z+t)}{c_n\rho_{g_n}(z)} \to 1.
	\end{align*}
	
	\bigskip
	
	(iii)$\implies$(iv).
	This implication is trivial: take $A = Z$.
	
	\bigskip
	
	(iv)$\implies$(i).
	For each $t \in A$, condition (iv) says that there is a function $F_t : Z \to S^1$ such that
	\begin{align} \label{eq: derivative coboundary}
		\frac{\rho_g(z+t)}{\rho_g(z)} = \frac{F_t \left( z + \hat{g} \right)}{F_t(z)}.
	\end{align}
	The functions $F_t$ can be chosen so that $t \mapsto F_t$ is measurable (see \cite[Proposition 2]{lesigne}).
	
	Define $\Phi : (Z \times S^1)^2 \to \C$ by
	\begin{align*}
		\Phi(z, \zeta; w, \eta) := \ind_A(z-w) F_{z-w}(w) \overline{\zeta}\eta.
	\end{align*}
	Then for the action $T^{\rho}_g(z,\zeta) := (z + \hat{g}, \rho_g(z)\zeta)$,
	the identity \eqref{eq: derivative coboundary} ensures that $\Phi$ is $\left( T^{\rho} \times T^{\rho} \right)$-invariant.
	Since $z \mapsto z + \hat{g}$ is ergodic, we can therefore express $\Phi$ as a sum
	\begin{align*}
		\Phi(z, \zeta; w, \eta) = \sum_j{c_j f_j(z, \zeta) \overline{f_j(w, \eta)}},
	\end{align*}
	where $f_j : Z \times S^1 \to S^1$ are eigenfunctions of $T^{\rho}$.
	
	Now take the Fourier expansion of $f_j$ in $\zeta$:
	\begin{align*}
		f_j(z,\zeta) = \sum_{n \in \Z}{a_{j,n}(z) \zeta^{-n}}.
	\end{align*}
	We obtain
	\begin{align*}
		\ind_A(z-w) F_{z-w}(w) \overline{\zeta}\eta
		= \sum_{j,n,m}{c_j a_{j,n}(z) \overline{a_{j,m}(w)} \zeta^{-n} \eta^m}.
	\end{align*}
	Matching coefficients for $n = m = 1$,
	\begin{align*}
		\ind_A(z-w) F_{z-w}(w) = \sum_j{c_j a_{j,1}(z) \overline{a_{j,1}(w)}}.
	\end{align*}
	Since $m_Z(A) > 0$, the left-hand size is nonzero, so $c_ja_{j,1}(z) \ne 0$ for some $j$.
	
	By assumption, $f_j$ is an eigenfunction of $T^{\rho}$, so there is a character $\gamma_j \in \hat{G}$ such that
	\begin{align*}
		T^{\rho}_gf_j = \gamma_j(g)f_j.
	\end{align*}
	Using the Fourier expansion, this gives
	\begin{align*}
		\sum_{n \in \Z}{a_{j,n} \left( z + \hat{g} \right) \rho_g^{-n}(z) \zeta^{-n}}
		= \gamma_j(g) \sum_{n \in \Z}{a_{j,n}(z) \zeta^{-n}}.
	\end{align*}
	Matching coefficients for $n = 1$, we have
	\begin{align*}
		a_{j,1} \left( z + \hat{g} \right) \rho_g^{-1}(z) = \gamma_j(g) a_{j,1}(z).
	\end{align*}
	One consequence of this identity is that $\left| a_{j,1} \left( z + \hat{g} \right) \right| = \left| a_{j,1}(z) \right|$.
	Since the action $z \mapsto z + \hat{g}$ is ergodic, it follows that $|a_{j,1}|$ is constant, say $|a_{j,1}| = C \ne 0$.
	Thus,
	\begin{align*}
		\rho_g(z) = \overline{\gamma_j}(g) \frac{C^{-1} a_{j,1} \left( z + \hat{g} \right)}{C^{-1} a_{j,1}(z)}.
	\end{align*}
	That is, $\rho_g \sim \overline{\gamma_j}$.
\end{proof}


\subsection{Quasi-affine cocycles} \label{sec: QA}

We now seek to give a characterization of CL cocycles with a similar flavor to the previous section.

\begin{defn}
	Let $Z$ be a compact abelian group.
	A function $\omega : Z \to S^1$ is \emph{affine} if there is a constant $c \in S^1$
	and a character $\lambda \in \hat{Z}$ such that
	$\omega(z) = c\lambda(z)$.
\end{defn}

\begin{prop} \label{prop: HK P6}
	Let $\mathbf{Z}$ be an ergodic Kronecker system and $\rho : G \times Z \to S^1$ a cocycle.
	The following are equivalent:
	\begin{enumerate}[(i)]
		\item	for any sequence $(g_n)_{n \in \N}$ in $G$ with $\hat{g}_n \to 0$ in $Z$,
		there is a sequence $(\omega_n)_{n \in \N}$ of affine functions such that
		$\omega_n\rho_{g_n}(z) \to 1$ in $L^2(Z)$;
		\item	for every $t \in Z$,
		\begin{align*}
			\frac{\rho_g(z+t)}{\rho_g(z)}
		\end{align*}
		is cohomologous to a character;
		\item	there is a Borel set $A \subseteq Z$ with $m_Z(A) > 0$ such that
		\begin{align*}
			\frac{\rho_g(z+t)}{\rho_g(z)}
		\end{align*}
		is cohomologous to a character for every $t \in A$.
	\end{enumerate}
\end{prop}

\begin{defn}
	Let $\mathbf{Z}$ be an ergodic Kronecker system.
	A cocycle $\rho : G \times Z \to S^1$ is \emph{quasi-affine}
	if it satisfies any (all) of the conditions of Proposition \ref{prop: HK P6}.
\end{defn}

\begin{rem}
	Condition (ii) is equivalent to satisfying the Conze--Lesigne equation for the parameters $(1,1)$.
	Indeed, writing out the cohomology relation explicitly, there are characters $\gamma_t \in \hat{G}$
	and functions $F_t : Z \to S^1$ such that
	\begin{align*}
		\frac{\rho_g(z+t)}{\rho_g(z)} = \gamma_t(g) \frac{F_t \left( z + \hat{g} \right)}{F_t(z)}.
	\end{align*}
\end{rem}

\begin{proof}[Proof of Proposition \ref{prop: HK P6}]
	(i) $\implies$ (ii).
	Suppose (i) holds.
	Let $t \in Z$, and consider the cocycle
	\begin{align*}
		\sigma_g(z) := \frac{\rho_g(z + t)}{\rho_g(z)}.
	\end{align*}
	We want to show that $\sigma$ is cohomologous to a character.
	By Proposition \ref{prop: HK P3}, it suffices to show that for any $(g_n)_{n \in \N}$ in $G$
	with $\hat{g}_n \to 0$ in $Z$, there is a sequence $(c_n)_{n \in \N}$ in $S^1$ such that
	$c_n\sigma_{g_n}(z) \to 1$ in $L^2(Z)$.
	Let $\hat{g}_n \to 0$.
	By (i), let $\omega_n = c_n\lambda_n$ be affine functions such that
	$c_n\lambda_n(z)\rho_{g_n}(z) \to 1$ in $L^2(Z)$.
	Since Haar measure is shift-invariant, we also have
	$c_n\lambda_n(z+t)\rho_{g_n}(z+t) \to 1$ in $L^2(Z)$.
	Let $d_n = \lambda_n(t) \in S^1$.
	Then
	\begin{align*}
		d_n\sigma_{g_n}(z) = \frac{\lambda_n(t)\rho_{g_n}(z + t)}{\rho_{g_n}(z)}
		= \frac{c_n\lambda_n(z+t)\rho_{g_n}(z + t)}{c_n\lambda_n(z)\rho_{g_n}(z)} \to 1
	\end{align*}
	in $L^2(Z)$.
	Thus, (ii) holds.
	
	\bigskip
	
	(ii) $\implies$ (i).
	Conversely, suppose (ii) holds.
	Write
	\begin{align*}
		\frac{\rho_g(z+t)}{\rho_g(z)} = \gamma_t(g) \frac{F_t \left( z + \hat{g} \right)}{F_t(z)}.
	\end{align*}
	Let $\hat{g}_n \to 0$.
	For each $t \in Z$,
	\begin{align*}
		\frac{\rho_g(z+t)}{\gamma_t(g)\rho_g(z)} = \frac{F_t \left( z + \hat{g} \right)}{F_t(z)}
	\end{align*}
	is a coboundary, so by Proposition \ref{prop: HK P2},
	\begin{align} \label{eq: rho gamma}
		\frac{\rho_{g_n}(z+t)}{\gamma_t(g_n)\rho_{g_n}(z)} \to 1
	\end{align}
	in $L^2(Z)$.
	For ease of notation, let $f_n : Z \to S^1$ be given by $f_n(z) := \rho_{g_n}(z)$
	and $c_n : Z \to S^1$ by $c_n(t) := \gamma_t(g_n)$.
	Finally, set $\varphi_{t,n}(z) := f_n(z + t) - c_n(t) f_n(z)$.
	Then equation \eqref{eq: rho gamma} can be rewritten as:
	\begin{align} \label{eq: norm phi to 0}
		\norm{L^2(Z)}{\varphi_{t,n}} \to 0
	\end{align}
	for every $t \in Z$.
	
	Now we take the Fourier transform of $\varphi_{t,n}$:
	\begin{align*}
		\hat{\varphi}_{t,n}(\lambda)
		= \int_Z{f_n(z+t)\overline{\lambda(z)}~dz} - c_n(t) \int_Z{f_n(z)\overline{\lambda(z)}~dz}
		= \hat{f}_n(\lambda) \left( \lambda(t) - c_n(t) \right).
	\end{align*}
	By Parseval's identity and \eqref{eq: norm phi to 0}, we have
	\begin{align*}
		\sum_{\lambda \in \hat{Z}}{\left| \hat{f}_n(\lambda) \right|^2 \left(2 - 2 \Re{\overline{\lambda(t)}c_n(t)} \right)}
		= \sum_{\lambda \in \hat{Z}}{\left| \hat{\varphi}_{t,n}(\lambda) \right|^2}
		= \norm{L^2(Z)}{\varphi_{t,n}}^2 \to 0.
	\end{align*}
	for every $t \in Z$.
	
	Integrating over $t \in Z$ and applying the dominated convergence theorem, we get
	\begin{align*}
		\eps_n := \sum_{\lambda \in \hat{Z}}{\left| \hat{f}_n(\lambda) \right|^2 \left(1 - \Re{\hat{c_n}(\lambda)} \right)}
		\to 0.
	\end{align*}
	The weights $\left| \hat{f}_n(\lambda) \right|^2$ give a probability density on $\hat{Z}$ for every $n \in \N$.
	Indeed, by Parseval's identity,
	\begin{align} \label{eq: Parseval}
		\sum_{\lambda \in \hat{Z}}{\left| \hat{f}_n(\lambda) \right|^2} = \norm{L^2(Z)}{f_n}^2
		= \int_Z{\left| \rho_{g_n}(z) \right|^2~dz} = \int_Z{1~dz} = 1.
	\end{align}
	Therefore, for some $\lambda_n \in \hat{Z}$, we have
	$1 - \Re{\hat{c}_n(\lambda_n)} \le \eps_n$.
	Then for $\lambda \ne \lambda_n$, we can use orthogonality of characters to get the bound
	\begin{align*}
		\left| \hat{c}_n(\lambda) \right| & = \left| \int_Z{\left( c_n(t) - \lambda_n(t) \right) \overline{\lambda(t)}~dt} \right| \\
		& \le \int_Z{\left| c_n(t) - \lambda_n(t) \right|~dt} \\
		& = \int_Z{\left| c_n(t)\overline{\lambda_n(t)} - 1 \right|~dt} \\
		& \le \left( \int_Z{\left| c_n(t)\overline{\lambda_n(t)} - 1 \right|^2~dt} \right)^{1/2} \\
		& = \left( 2 \Re{\int_Z{\left( 1 - c_n(t)\overline{\lambda_n(t)} \right)~dt}}\right)^{1/2} \\
		& = \sqrt{2} \left( 1 - \Re{\hat{c}_n(\lambda_n)} \right)^{1/2} \\
		& \le \sqrt{2 \eps_n}.
	\end{align*}
	So,
	\begin{align*}
		\sum_{\lambda \ne \lambda_n}{\left| \hat{f}_n(\lambda) \right|^2}
		\le \sum_{\lambda \ne \lambda_n}{\left| \hat{f}_n(\lambda) \right|^2
			\left( \frac{1 - \Re{\hat{c_n}(\lambda)}}{{1 - \sqrt{2\eps_n}}} \right)}
		\le \frac{\eps_n}{1 - \sqrt{2 \eps_n}} \to 0.
	\end{align*}
	Comparing with \eqref{eq: Parseval}, this implies $\left| \hat{f}_n(\lambda_n) \right| \to 1$.
	
	Let $\omega_n : Z \to S^1$ be the affine function
	\begin{align*}
		\omega_n(z) := \frac{\hat{f}_n(\lambda_n)}{\left| \hat{f}_n(\lambda_n) \right|} \lambda_n(z).
	\end{align*}
	Then
	\begin{align*}
		\norm{L^2(Z)}{\overline{w}_n \rho_{g_n}(z) - 1}
		& = \norm{L^2(Z)}{f_n(z) - \frac{\hat{f}_n(\lambda_n)}{\left| \hat{f}_n(\lambda_n) \right|} \lambda_n(z)} \\
		& \le \norm{L^2(Z)}{f_n(z) - \hat{f}_n(\lambda_n) \lambda_n(z)}
		+ \left| \left| \hat{f}_n(\lambda_n) \right| - 1 \right| \\
		& = \left( \sum_{\lambda \ne \lambda_n}{\left| \hat{f}_n(\lambda) \right|^2} \right)^{1/2}
		+ \left| \left| \hat{f}_n(\lambda_n) \right| - 1 \right| \to 0.
	\end{align*}
	Therefore, (i) holds.
	
	\bigskip
	
	(ii) $\implies$ (iii).
	This implication is obvious: take $A = Z$.
	
	\bigskip
	
	(iii) $\implies$ (ii).
	Suppose (iii) holds.
	Define
	\begin{align*}
		K := \left\{ t \in Z : \frac{\rho_g(z+t)}{\rho_g(z)}~\text{is cohomologous to a character} \right\}.
	\end{align*}
	It is easy to check that $K$ is a subgroup of $Z$.
	By assumption, $m_Z(K) > 0$, so $K$ must be a clopen set.
	Using the cocycle equation, we have
	\begin{align*}
		\frac{\rho_g \left( z + \hat{h} \right)}{\rho_g(z)}
		= \frac{\rho_h \left( z + \hat{g} \right)}{\rho_h(z)},
	\end{align*}
	for every $g, h \in G$.
	Therefore, $\hat{h} \in K$ for every $h \in G$.
	But $\left\{ \hat{h} : h \in G \right\}$ is dense in $Z$, and $K$ is closed, so $K = Z$.
	This proves (ii).
\end{proof}

\begin{lem} \label{lem: HK C1}
	Let $\mathbf{Z}$ be an ergodic Kronecker system and $\rho : G \times Z \to S^1$ a cocycle.
	Suppose $(g_n)_{n \in \N}$ is a sequence in $G$ such that $\hat{g}_n \to 0$ in $Z$
	and $\omega_n(z) = c_n\lambda_n(z)$ are affine functions such that
	$\omega_n \rho_{g_n} \to 1$ in $L^2(Z)$.
	Then, for every $h \in G$, $\lambda_n \left( \hat{h} \right) \to 1$.
\end{lem}
\begin{proof}
	Let $h \in G$.
	We have
	\begin{align*}
		\frac{c_n \lambda_n \left( z + \hat{h} \right) \rho_{g_n} \left( z + \hat{h} \right)}
		{c_n \lambda_n(z) \rho_{g_n}(z)}
		= \lambda_n \left( \hat{h} \right) \frac{\rho_{g_n} \left( z + \hat{h} \right)}{\rho_{g_n}(z)}.
	\end{align*}
	By Lemma \ref{lem: HK L1},
	\begin{align*}
		\frac{\rho_{g_n} \left( z + \hat{h} \right)}{\rho_{g_n}(z)} \to 1
	\end{align*}
	in $L^2(Z)$, so $\lambda_n \left( \hat{h} \right) \to 1$.
\end{proof}

\begin{lem} \label{lem: HK L2}
	Let $\mathbf{Z}$ be an ergodic Kronecker system and $\rho : G \times Z \to S^1$ a cocycle.
	Suppose $\left( \hat{g}_n \right)$ converges (to $0$) in $Z$,
	and $\omega_n(z) = c_n\lambda_n(z)$ are affine functions such that
	$\left( \omega_n\rho_{g_n} \right)$ converges (to $1$) in $L^2(Z)$.
	Then for every $a \in \N$,
	\begin{align*}
		c_n^a \lambda_n \left( \binom{a}{2} \hat{g}_n \right) \lambda_n^a(z) \rho_{ag_n}(z)
	\end{align*}
	converges (to $1$) in $L^2(Z)$.
\end{lem}
\begin{proof}
	Let $j \in \Z$.
	Then
	\begin{align*}
		c_n \lambda_n\left( j\hat{g}_n \right) \lambda_n(z) \rho_{g_n} \left( z + j\hat{g}_n \right)
		= \omega_n\left( z + j\hat{g}_n \right) \rho_{g_n}\left( z + j\hat{g_n} \right)
	\end{align*}
	converges (to $1$) in $L^2(Z)$.
	Take the product over $j = 0, \dots, a - 1$:
	\begin{align*}
		\prod_{j=0}^{a-1}{c_n \lambda_n\left( j\hat{g}_n \right) \lambda_n(z) \rho_{g_n} \left( z + j\hat{g}_n \right)}
		& = c_n^a \lambda_n \left( \sum_{j=0}^{a-1}{j \hat{g}_n} \right)
		\lambda_n^a(z) \prod_{j=0}^{a-1}{\rho_{g_n} \left( z + j\hat{g}_n \right)} \\
		& = c_n^a \lambda_n \left( \binom{a}{2} \hat{g}_n \right) \lambda^a_n(z) \rho_{ag_n}(z)
	\end{align*}
	converges (to $1$) in $L^2(Z)$.
\end{proof}

\begin{lem} \label{lem: HK L3}
	Let $\mathbf{Z}$ be an ergodic Kronecker system and $\rho : G \times Z \to S^1$ a quasi-affine cocycle.
	If $\left( \hat{g}_n \right)_{n \in \N}$ converges in $Z$, then there is a sequence of affine functions
	$\omega_n(z) = c_n\lambda_n(z)$ such that $\left( \omega_n \rho_{g_n} \right)$ converges in $L^2(Z)$.
\end{lem}
\begin{proof}
	Let $d$ be a translation-invariant metric on $Z$.
	By Proposition \ref{prop: HK P6}(i), given $\eps > 0$, there exists $\delta = \delta(\eps) > 0$
	such that, if $d \left( \hat{g}, 0 \right) < \delta$, then there is an affine function $\omega : Z \to S^1$ such that
	$\norm{L^2(Z)}{\omega\rho_g - 1} < \eps$.
	
	Assume $\left( \hat{g}_n \right)$ converges in $Z$.
	Let $\eps > 0$.
	Then there is an $N = N(\eps) \in \N$ such that, for all $i, j \ge N$,
	we have $d \left( \hat{g}_i, \hat{g}_j \right) < \delta$.
	Therefore, for $i,j \ge N$, there is an affine function $\omega: Z \to S^1$ such that
	\begin{align*}
		\norm{L^2(Z)}{\omega \rho_{g_i - g_j} - 1} < \eps.
	\end{align*}
	Now, using the cocycle equation
	\begin{align*}
		\frac{\rho_{g_i}(z)}{\rho_{g_j}(z)} = \rho_{g_i-g_j} \left( z + \hat{g}_j \right),
	\end{align*}
	so
	\begin{align} \label{eq: affine Cauchy}
		\norm{L^2(Z)}{\omega \left( z + \hat{g}_j \right) \rho_{g_i}(z) - \rho_{g_j}(z)}
		= \norm{L^2(Z)}{\omega(z) \rho_{g_i - g_j}(z) - 1} < \eps.
	\end{align}
	Note that the function $z \mapsto \omega \left( z + \hat{g}_j \right)$ is affine.
	We will use this to inductively construct a sequence of affine functions $(\omega_n)_{n \in \N}$
	so that $\left( \omega_n \rho_{g_n} \right)_{n \in \N}$ is Cauchy in $L^2(Z)$ and hence convergent.
	
	For each $k \in \N$, let $\eps_k = 2^{-k}$ and $N_k = N \left( \eps_k \right)$.
	For $n \le N_1$, let $\omega_n = 1$.
	Suppose we have chosen $\omega_n$ for $n \le N_k$ for some $k \in \N$.
	For $N_k < n \le N_{k+1}$, use \eqref{eq: affine Cauchy} to choose an affine function $\omega_n : Z \to S^1$
	such that
	\begin{align*}
		\norm{L^2(Z)}{\omega_n \rho_{g_n} - \omega_{N_k} \rho_{g_{N_k}}} < 2^{-k}.
	\end{align*}
	
	Now suppose $i, j > N_k$, $k \in \N$.
	Let $r, s \ge k$ such that $N_r < i \le N_{r+1}$ and $N_s < j \le N_{s+1}$.
	Without loss of generality, $s \ge r$.
	By the triangle inequality,
	\begin{align*}
		& \norm{L^2(Z)}{\omega_j\rho_{g_j} - \omega_i\rho_{g_i}} \\
		& \le \norm{L^2(Z)}{\omega_j\rho_{g_j} - \omega_{N_s}\rho_{g_{N_s}}}
		+ \sum_{t=r}^{s-1}{\norm{L^2(Z)}{w_{N_{t+1}}\rho_{g_{N_{t+1}}} - \omega_{N_t}\rho_{g_{N_t}}}}
		+ \norm{L^2(Z)}{\omega_{N_r}\rho_{g_{N_r}} - \omega_i\rho_{g_i}} \\
		& < 2^{-s} + \sum_{t=r}^{s-1}{2^{-t}} + 2^{-r}
		< 2^{-(r-1)} + 2^{-r} < 2^{-(r-2)} \le 2^{-(k-2)}.
	\end{align*}
	Thus, $\left( \omega_n \rho_{g_n} \right)$ is Cauchy in $L^2(Z)$.
\end{proof}

\begin{defn}
	Let $\mathbf{Z}$ be an ergodic Kronecker system and $H$ a compact abelian group.
	A cocycle $\sigma : G \times Z \to H$ is \emph{quasi-affine} if $\chi \circ \sigma$
	is quasi-affine for all $\chi \in \hat{H}$.
\end{defn}

\begin{defn}
	An ergodic system $\X$ is \emph{quasi-affine} if it is an extension of its Kronecker factor by a quasi-affine cocycle.
	That is, there is a compact abelian group $H$ and a (weakly mixing) quasi-affine cocycle $\sigma : G \times Z \to H$
	such that $\X \cong \mathbf{Z} \times_{\sigma} H$.
\end{defn}

\begin{lem} \label{lem: H divisibility}
	Let $\mathbf{Z} \times_{\sigma} H$ be an ergodic quasi-affine $G$-system.
	Suppose $a \in \Z$ and $aG$ has finite index in $G$.
	Then $aH = H$.
\end{lem}

\begin{rem}
	In the case $G = \Z$, the condition on $a$ in Lemma \ref{lem: H divisibility} is satisfied for every $a \ne 0$.
	Lemma \ref{lem: H divisibility} therefore has a simpler statement in this setting.
	Namely, if $\mathbf{Z} \times_{\sigma} H$ is an ergodic quasi-affine $\Z$-system, then $H$ is divisible.
	This special case is established in \cite[Lemma 5]{hk}.
	
	For our setting of general countable discrete abelian groups,
	the admissibility condition on $a$ cannot be dropped.
	Indeed, in \cite{btz1}, it was shown that for an ergodic quasi-affine $\bigoplus_{n=1}^{\infty}{\F_p}$-system,
	the group $H$ consists of $p$-torsion elements (see \cite[Lemma 4.7]{btz1}).
\end{rem}

In order to prove Lemma \ref{lem: H divisibility}, we need the following fact, which will be used again in the next section:
\begin{lem} \label{lem: sqrt2 bound}
	Let $Z$ be a compact abelian group.
	Let $c_1, c_2 \in S^1$ and $\lambda_1, \lambda_2 \in \hat{Z}$.
	If $\lambda_1 \ne \lambda_2$, then
	\begin{align*}
		\norm{L^2(Z)}{c_1\lambda_1 - c_2\lambda_2} = \sqrt{2}.
	\end{align*}
\end{lem}
\begin{proof}
	This is a simple consequence of orthogonality of characters:
	\begin{align*}
		\norm{L^2(Z)}{c_1\lambda_1 - c_2\lambda_2}^2
		= \innprod{c_1\lambda_1 - c_2\lambda_2}{c_1\lambda_1 - c_2\lambda_2}
		= |c_1|^2 \norm{L^2(Z)}{\lambda_1}^2 + |c_2|^2 \norm{L^2(Z)}{\lambda_2}^2 = 2.
	\end{align*}
\end{proof}

Now we prove Lemma \ref{lem: H divisibility}:
\begin{proof}[Proof of Lemma \ref{lem: H divisibility}]
	Since the set of $a$-torsion elements is the annihilator of $aH$,
	there is an isomorphism between the $a$-torsion elements of $\hat{H}$
	and the group $\hat{H/aH}$
	(see, e.g., \cite[Theorem 2.1.2]{rudin}).
	It therefore suffices to show that $\hat{H}$ contains no nontrivial $a$-torsion elements.
	
	Let $\chi \in \hat{H}$ and suppose $\chi^a = 1$.
	Let $\hat{g}_n \to 0$ in $Z$.
	The cocycle $\chi \circ \sigma$ is quasi-affine, so by Proposition \ref{prop: HK P6}(i),
	there is a sequence of affine functions $\omega_n(z) = c_n \lambda_n(z)$ such that
	\begin{align} \label{eq: QA convergence}
		c_n \lambda_n(z) \chi \left( \sigma_{g_n}(z) \right) \to 1
	\end{align}
	in $L^2(Z)$.
	Taking this expression to the $a$th power and using $\chi^a = 1$, we have
	\begin{align*}
		c_n^a \lambda_n^a(z) \to 1
	\end{align*}
	in $L^2(Z)$.
	In particular, for $n$ sufficiently large,
	\begin{align*}
		\norm{L^2(Z)}{c_n^a\lambda_n^a(z) - 1} < \sqrt{2}.
	\end{align*}
	By Lemma \ref{lem: sqrt2 bound}, this implies $\lambda_n^a = 1$ for all large $n$.
	
	Now we use admissibility of $a$.
	Let $\Lambda_a := \left\{ \lambda \in \hat{Z} : \lambda^a = 1 \right\}$.
	Clearly, $\Lambda_a = \left( aZ \right)^{\perp}$, so
	$\Lambda_a \cong \hat{Z/aZ}$.
	Since $\left\{ \hat{g} : g \in G \right\}$ is dense in $Z$, it is easy to check that\footnote{
		Here is a sketch of the proof.
		Let $x_1, \dots, x_k \in G$ such that $aG + \{x_1, \dots, x_k\} = G$.
		Now let $z \in Z$.
		Then there is a sequence $(g_n)_{n \in \N}$ in $G$ such that $\hat{g}_n \to z$.
		By the pigeonhole principle, we may assume (taking a subsequence if necessary)
		that $g_n = ag'_n + x_i$ for some fixed $1 \le i \le k$.
		Now by compactness of $Z$, we may further assume (taking yet another subsequence if necessary)
		that $\hat{g'_n} \to z' \in Z$.
		Then $z = az' + \hat{x}_i$.
		Hence, $Z = aZ + \{\hat{x}_1, \dots, \hat{x}_k\}$.}
	$[Z : aZ] \le [G : aG] < \infty$.
	Thus, $\Lambda_a$ is a finite group.
	Now, for any pair of distinct characters $\lambda, \lambda' \in \hat{Z}$, there is an element $g \in G$
	such that $\lambda\left( \hat{g} \right) \ne \lambda'\left( \hat{g} \right)$.
	It follows that there is a finite set $h_1, \dots, h_m \in G$ that distinguishes elements of $\Lambda_a$.
	That is, if $\lambda, \lambda' \in \Lambda_a$
	and $\lambda \left( \hat{h}_i \right) = \lambda' \left( \hat{h}_i \right)$ for all $i = 1, \dots, m$,
	then $\lambda = \lambda'$.
	
	By Lemma \ref{lem: HK C1}, $\lambda_n \left( \hat{h}_i \right) \to 1$ for each $i = 1, \dots, m$.
	Since $\lambda_n^a = 1$ for all large $n$, it follows that we may choose $N_i \in \N$ such that
	$\lambda_n \left( \hat{h}_i \right) = 1$ for all $n \ge N_i$.
	Let $N_0 \in \N$ so that $\lambda_n \in \Lambda_a$ for $n \ge N_0$, and set $N := \max_{0 \le i \le m}{N_i}$.
	For $n \ge N$, we have $\lambda_n \in \Lambda_a$ and $\lambda_n \left( \hat{h}_i \right) = 1$
	for every $i = 1, \dots, m$.
	By the choice of $h_i$, this implies $\lambda_n = 1$.
	
	Returning to \eqref{eq: QA convergence} and using the fact that $\lambda_n = 1$ for all large $n$, we have
	\begin{align*}
		c_n \chi \left( \sigma_{g_n}(z) \right) \to 1
	\end{align*}
	in $L^2(Z)$.
	By Proposition \ref{prop: HK P3}, $\chi \circ \sigma$ is cohomologous to a character.
	But $\sigma$ is weakly mixing, so by Proposition \ref{prop: HK P4}(2), $\chi = 1$.
\end{proof}


\subsection{Mackey groups for quasi-affine systems}

We now return to analyzing the Mackey group $M(a_1,a_2,a_3)$, defined in Section \ref{sec: Mackey triples},
in the special case that the cocycle $\sigma$ is quasi-affine.
We begin by fixing notation for the section.
Let $G$ be a countable discrete abelian group.
Let $\{a_1, a_2, a_3\} \subseteq \Z$ be an admissible triple
(meaning that $a_iG$ and $(a_j-a_i)G$ have finite index in $G$ for $1 \le i \ne j \le 3$).
Define
\begin{align*}
	k'_1 & := a_2a_3(a_2-a_3) \\
	k'_2 & := a_3a_1(a_3-a_1) \\
	k'_3 & := a_1a_2(a_1-a_2)
\end{align*}
Let $D := \gcd(k'_1, k'_2, k'_3)$, and define $k_i := \frac{k'_i}{D}$.
Then the numbers $k_i$ satisfy:
\begin{align}
	\sum_{i=1}^3{k_ia_i} = \sum_{i=1}^3{k_ia_i^2} & = 0, \label{eq: k identity 1} \\
	\gcd(k_1,k_2,k_3) & = 1. \label{eq: k identity 2}
\end{align}

Note that admissibility of $\{a_1, a_2, a_3\}$ implies that $k_i'G$ has finite index in $G$ for each $i = 1, 2, 3$.
Hence, $DG$ also has finite index in $G$.

The goal of this section is to prove the following concrete description of the Mackey group:
\begin{thm} \label{thm: HK T10}
	Let $\mathbf{Z} \times_{\sigma} H$ be an ergodic quasi-affine $G$-system.
	Let $M(a_1, a_2, a_3)$ be the Mackey group associated to
	$\tilde{\sigma}_g := \left( \sigma_{a_1g}, \sigma_{a_2g}, \sigma_{a_3g} \right)$.
	Then
	\begin{align*}
		M^{\perp} = \left\{ \left( \chi^{k_1}, \chi^{k_2}, \chi^{k_3} \right) : \chi \in \hat{H} \right\}
		= \left\{ (\chi_1,\chi_2,\chi_3) \in \hat{H}^3 : \prod_{i=1}^3{\chi_i^{a_i}} = \prod_{i=1}^3{\chi_i^{a_i^2}} = 1 \right\}.
	\end{align*}
\end{thm}

Theorem \ref{thm: HK T10} was proved for $\Z$-systems in \cite[Theorem 10]{hk}.
As an immediate consequence, we have:

\begin{cor} \label{cor: HK C4}
	In the setup of Theorem \ref{thm: HK T10},
	\begin{align*}
		M = \left\{ \left( a_1u + a_1^2v, a_2u + a_2^2v, a_3u + a_3^2v \right) : u, v \in H \right\}
		= \left\{ (h_1,h_2,h_3) \in H^3 : \sum_{i=1}^3{k_ih_i} = 0 \right\}.
	\end{align*}
\end{cor}

\begin{proof}[Proof of Theorem \ref{thm: HK T10}]
	First we show $M^{\perp} \subseteq
	\left\{ (\chi_1,\chi_2,\chi_3) \in \hat{H}^3 : \prod_{i=1}^3{\chi_i^{a_i}} = \prod_{i=1}^3{\chi_i^{a_i^2}} = 1 \right\}$.
	Let $\tilde{\chi} = (\chi_1, \chi_2, \chi_3) \in M^{\perp}$.
	Let $\hat{g}_n \to 0$ in $Z$.
	Since $\tilde{\chi} \in M^{\perp}$, $\tilde{\chi} \circ \tilde{\sigma}$ is a coboundary.
	Therefore, by Proposition \ref{prop: HK P2},
	\begin{align} \label{eq: coboundary condition}
		\prod_{i=1}^3{\left( \chi_i \circ \sigma_{a_ig_n} \right)(w_i)} \to 1
	\end{align}
	in $L^2(W)$.
	
	Now, $\chi_i \circ \sigma$ is quasi-affine for each $i = 1, 2, 3$, so there are constants $c_{i,n} \in S^1$
	and $\lambda_{i,n} \in \hat{Z}$ such that
	\begin{align*}
		c_{i,n} \lambda_{i,n}(z) \left( \chi_i \circ \sigma_{g_n} \right)(z) \to 1
	\end{align*}
	in $L^2(Z)$.
	By Lemma \ref{lem: HK L2},
	\begin{align*}
		d_{i,n} \lambda_{i,n}^{a_i}(z) \left( \chi \circ \sigma_{a_ig_n} \right) (z) \to 1
	\end{align*}
	in $L^2(Z)$, where
	\begin{align*}
		d_{i,n} = c_{i,n}^{a_i} \lambda_{i,n} \left( \binom{a_i}{2} \hat{g}_n \right).
	\end{align*}
	
	For notational convenience, let $f_{i,n}(z) := d_{i,n}\lambda_{i,n}^{a_i}(z) \left( \chi \circ \sigma_{a_ig_n} \right)(z)$.
	Note that $f_{i,n} : Z \to S^1$ and $f_{i,n} \to 1$ in $L^2(Z)$ for each $i = 1, 2, 3$.
	Taking the product over $i = 1, 2, 3$, we have
	\begin{align*}
		\norm{L^2(W)}{f_{1,n}(w_1)f_{2,n}(w_2)f_{3,n}(w_3) - 1}
		& \le \norm{L^2(W)}{f_{1,n}(w_1)f_{2,n}(w_2)f_{3,n}(w_3) - f_{2,n}(w_2)f_{3,n}(w_3)} \\
		& + \norm{L^2(W)}{f_{2,n}(w_2)f_{3,n}(w_3) - f_{3,n}(w_3)}
		+ \norm{L^2(W)}{f_{3,n}(w_3) - 1} \\
		& = \sum_{i=1}^3{\norm{L^2(W)}{f_{i,n}(w_i) - 1}}.
	\end{align*}
	Now, for each $i = 1, 2, 3$,
	\begin{align*}
		\norm{L^2(W)}{f_{i,n}(w_i) - 1}^2
		& = \int_Z{\int_Z{\left| f_{i,n}(z + a_it) - 1 \right|^2~dz}~dt}
		& (W = \left\{ (z + a_1t, z + a_2t, z + a_3t) : z, t \in Z \right\}) \\
		& = \int_Z{\int_Z{\left| f_{i,n}(z) - 1 \right|^2~dz}~dt}
		& (\text{Haar measure on}~Z~\text{is translation-invariant}) \\
		& = \norm{L^2(Z)}{f_{i,n} - 1}^2 \to 0.
	\end{align*}
	Thus,
	\begin{align*}
		\prod_{i=1}^3{d_{i,n}\lambda_{i,n}^{a_i}(z) \left( \chi \circ \sigma_{a_ig_n} \right)(w_i)} \to 1
	\end{align*}
	in $L^2(W)$.
	
	Combining with \eqref{eq: coboundary condition}, we have
	\begin{align*}
		\prod_{i=1}^3{d_{i,n} \lambda_{i,n}^{a_i}(w_i)} \to 1
	\end{align*}
	in $L^2(W)$.
	Since $W = \left\{ (z + a_1t, z + a_2t, z + a_3t) : z, t \in Z \right\}$, this is equivalent to
	\begin{align*}
		\int_Z{\int_Z{\left| 1 - \prod_{i=1}^3{d_{i,n} \lambda_{i,n}(a_iz + a_i^2t)} \right|^2~dz}~dt} \to 0.
	\end{align*}
	In particular, for all sufficiently large $n$,
	\begin{align*}
		\norm{L^2(Z \times Z)}{1 - \left( \prod_{i=1}^3{d_{i,n}} \right)
			\left( \prod_{i=1}^3{\lambda_{i,n}^{a_i} \otimes \lambda_{i,n}^{a_i^2}} \right)} < \sqrt{2}.
	\end{align*}
	By Lemma \ref{lem: sqrt2 bound},
	\begin{align*}
		\prod_{i=1}^3{\lambda_{i,n}^{a_i}} = \prod_{i=1}^3{\lambda_{i,n}^{a_i^2}} = 1.
	\end{align*}
	
	Set
	\begin{align*}
		u_n := \prod_{i=1}^3{c_{i,n}^{a_i}}, \quad u'_n = \prod_{i=1}^3{c_{i,n}^{a_i^2}}.
	\end{align*}
	For all large $n$, we have
	\begin{align*}
		u_n \prod_{i=1}^3{\left( \chi_i^{a_i} \circ \sigma_{g_n} \right)(z)}
		= \prod_{i=1}^3{c_{i,n}^{a_i} \lambda_{i,n}^{a_i}(z) \left( \chi \circ \sigma_{g_n} \right)^{a_i}(z)}
		\to 1
	\end{align*}
	in $L^2(Z)$.
	Similarly,
	\begin{align*}
		u'_n \prod_{i=1}^3{\left( \chi_i^{a_i^2} \circ \sigma_{g_n} \right)(z)}
		= \prod_{i=1}^3{c_{i,n}^{a_i^2} \lambda_{i,n}^{a_i^2}(z) \left( \chi \circ \sigma_{g_n} \right)^{a_i^2}(z)}
		\to 1
	\end{align*}
	in $L^2(Z)$.
	Thus, by Proposition \ref{prop: HK P3},
	\begin{align*}
		\left( \prod_{i=1}^3{\chi_i^{a_i}} \right) \circ \sigma \quad \text{and} \quad
		\left( \prod_{i=1}^3{\chi_i^{a_i^2}} \right) \circ \sigma
	\end{align*}
	are cohomologous to characters.
	Since $\sigma$ is weakly mixing, Proposition \ref{prop: HK P4}(2) implies
	\begin{align*}
		\prod_{i=1}^3{\chi_i^{a_i}} = \prod_{i=1}^3{\chi_i^{a_i^2}} = 1.
	\end{align*}
	
	\bigskip
	
	Now we show $\left\{ (\chi_1,\chi_2,\chi_3) \in \hat{H}^3 : \prod_{i=1}^3{\chi_i^{a_i}} = \prod_{i=1}^3{\chi_i^{a_i^2}} = 1 \right\}
	\subseteq \left\{ \left( \chi^{k_1}, \chi^{k_2}, \chi^{k_3} \right) : \chi \in \hat{H} \right\}$.
	Let \, $(\chi_1, \chi_2, \chi_3) \in \hat{H}^3$,
	and suppose $\prod_{i=1}^3{\chi_i^{a_i}} = \prod_{i=1}^3{\chi_i^{a_i^2}} = 1$.
	By direct calculation, $\chi_i^{k'_j} = \chi_j^{k'_i}$ for $1 \le i, j \le 3$.
	By Lemma \ref{lem: H divisibility}, since $DG$ has finite index in $G$,
	the group $\hat{H}$ has no $D$-torsion, so $\chi_i^{k_j} = \chi_j^{k_i}$.
	
	Since $\gcd(k_1, k_2, k_3) = 1$, let $b_1, b_2, b_3 \in \Z$ such that $\sum_{i=1}^3{b_ik_i} = 1$,
	and set $\chi = \prod_{i=1}^3{\chi_i^{b_i}}$.
	Then
	\begin{align*}
		\chi^{k_j} = \prod_{i=1}^3{\chi_i^{b_ik_j}}
		= \prod_{i=1}^3{\chi_j^{b_ik_i}}
		= \chi_j^{\sum_{i=1}^3{b_ik_i}} = \chi_j.
	\end{align*}
	
	\bigskip
	
	Finally, we check that $\left\{ \left( \chi^{k_1}, \chi^{k_2}, \chi^{k_3} \right) : \chi \in \hat{H} \right\}
	\subseteq M^{\perp}$.
	Let $\chi \in \hat{H}$.
	We will apply Proposition \ref{prop: HK P2} to check that
	$\left( \chi^{k_1}, \chi^{k_2}, \chi^{k_3} \right) \circ \tilde{\sigma}$ is a coboundary.
	Let $(g_n)_{n \in \N}$ be a sequence in $G$
	such that $\left( a_1\hat{g}_n, a_2\hat{g}_n, a_3\hat{g}_n \right) \to 0$ in $W$.
	Then $a_i\hat{g}_n \to 0$ in $Z$ for each $i = 1, 2, 3$, so $a\hat{g}_n \to 0$, where $a = \gcd(a_1,a_2,a_3)$.
	The cocycle $\chi \circ \sigma$ is quasi-affine, so there are affine functions $\omega_n(z) = c_n \lambda_n(z)$
	such that
	\begin{align*}
		c_n \lambda_n(z) \chi \left( \sigma_{ag_n}(z) \right) \to 1
	\end{align*}
	in $L^2(Z)$.
	
	Set $a'_i := \frac{a_i}{a}$ and
	\begin{align*}
		d_{i,n} := c_n^{a'_i} \lambda_n \left( \binom{a'_i}{2} a\hat{g}_n \right).
	\end{align*}
	By Lemma \ref{lem: HK L2},
	\begin{align*}
		d_{i,n} \lambda_n^{a'_i}(z) \chi \left( \sigma_{a_ig_n}(z) \right) \to 1
	\end{align*}
	in $L^2(Z)$.
	Therefore,
	\begin{align*}
		\prod_{i=1}^3{d_{i,n}^{k_i} \lambda_n^{k_ia'_i}(w_i) \chi^{k_i} \left( \sigma_{a_ig_n}(w_i) \right)} \to 1
	\end{align*}
	in $L^2(W)$.
	
	From the definition of $d_{i,n}$ and \eqref{eq: k identity 1}, $\prod_{i=1}^3{d_{i,n}^{k_i}} = 1$.
	Moreover, for $w = (z + a_1t, z + a_2t, z + a_3t) \in W$, we have by \eqref{eq: k identity 1},
	\begin{align*}
		\sum_{i=1}^3{k_ia'_iw_i} = \sum_{i=1}^3{k_ia'_iz} + \sum_{i=1}^3{k_ia'_ia_it} = 0,
	\end{align*}
	so $\prod_{i=1}^3{\lambda_n^{k_ia'_i}(w_i)} = 1$.
	Thus,
	\begin{align*}
		\prod_{i=1}^3{\chi^{k_i} \left( \sigma_{a_ig_n}(w_i) \right)} \to 1
	\end{align*}
	in $L^2(W)$.
	That is, $\left( \chi^{k_1}, \chi^{k_2}, \chi^{k_3} \right) \circ \tilde{\sigma}_{g_n}(w) \to 1$ in $L^2(W)$.
	By Proposition \ref{prop: HK P2}, $\left( \chi^{k_1}, \chi^{k_2}, \chi^{k_3} \right) \circ \tilde{\sigma}$
	is a coboundary, so $\left( \chi^{k_1}, \chi^{k_2}, \chi^{k_3} \right) \in M^{\perp}$.
\end{proof}


\subsection{Limit formula}

We now have all of the necessary tools to prove Theorem \ref{thm: limit formula},
which we restate here for the convenience of the reader.
Recall that $\M(Z,H)$ denotes the set of measurable functions $Z \to H$
equipped with the topology of convergence in measure.

{\renewcommand\footnote[1]{}\LimitFormula*}

We now turn to constructing the function $\psi$ appearing in Theorem \ref{thm: limit formula}.
To prove continuity of the map $t \mapsto \psi(t,\cdot)$, we will use the following
characterization of convergence in measure:

\begin{lem} \label{lem: conv in meas}
	Let $(f_n)_{n \in \N}$ be a sequence of functions in $\M(Z,H)$.
	Then $f_n \to f$ in $\M(Z,H)$ if and only if $\chi \circ f_n \to \chi \circ f$ in $L^2(Z)$
	for every character $\chi \in \hat{H}$.
\end{lem}
\begin{proof}
	Denote by $m_Z$ the Haar measure on $Z$.
	Suppose $f_n \to f$ in $\M(Z,H)$.
	Let $\chi \in \hat{H}$.
	Since $\chi$ is (uniformly) continuous, it follows that $\chi \circ f_n \to \chi \circ f$ in measure.
	Moreover, $\chi$ is bounded, so $\chi \circ f_n \to \chi \circ f$ in $L^2(Z)$.
	
	Conversely, suppose $\chi \circ f_n \to \chi \circ f$ in $L^2(Z)$ for every $\chi \in \hat{H}$.
	Let $d_H$ be a translation-invariant metric on $H$.
	Then $d_H : H \times H \to [0, \infty)$ is a continuous function.
	Let $\eps > 0$.
	By the Stone--Weierstrass theorem, there are characters
	$\chi_k, \xi_k \in \hat{H}$ and coefficients $c_k \in \C$, $k = 1, \dots, K$, such that
	\begin{align*}
		\left| d_H(x,y) - \sum_{k=1}^K{c_k \chi_k(x) \xi_k(y)} \right| < \frac{\eps}{3}
	\end{align*}
	for every $x, y \in H$.
	Let $\Phi(x,y) := \sum_{k=1}^K{c_k\chi_k(x)\xi_k(y)}$.
	
	For each $k = 1, \dots, K$, $\chi_k \circ f_n \to \chi_k \circ f$ in $L^2(Z)$ and hence in measure.
	Therefore,
	\begin{align*}
		\Phi(f_n(z), f(z)) \to \Phi(f(z),f(z))
	\end{align*}
	in measure.
	For each $n \in \N$, set
	\begin{align*}
		B_n := \left\{ z \in Z : \left| \Phi(f_n(z), f(z)) - \Phi(f(z),f(z)) \right| \ge \frac{\eps}{3} \right\},
	\end{align*}
	and let $N \in \N$ so that, for $n \ge N$,
	\begin{align*}
		m_Z(B_n) < \eps.
	\end{align*}
	For $z \in Z \setminus B_n$, we have
	\begin{align*}
		d_H(f_n(z), f(z))
		\le & \left| d_H(f_n(z), f(z)) - \Phi(f_n(z), f(z)) \right| \\
		& + \left| \Phi(f_n(z), f(z)) - \Phi(f(z), f(z)) \right|
		+ \left| \Phi(f(z), f(z)) \right| \\
		& < \frac{\eps}{3} + \frac{\eps}{3} + \frac{\eps}{3} = \eps.
	\end{align*}
	Thus, for $n \ge N$,
	\begin{align*}
		m_Z \left( \left\{ z \in Z : d_H(f_n(z), f(z)) \ge \eps \right\} \right) \le m_Z(B_n) < \eps.
	\end{align*}
	That is, $f_n \to f$ in measure.
\end{proof}

\begin{prop} \label{prop: HK P14}
	Let $\mathbf{Z} \times_{\sigma} H$ be an ergodic quasi-affine system.
	There is a function $\psi : Z \times Z \to H$ such that
	\begin{enumerate}[(1)]
		\item	for every $g \in G$,
		\begin{align*}
			\psi \left( \hat{g}, z \right) = \sum_{i=1}^3{k_i \sigma_{a_ig}(z)},
		\end{align*}
		and
		\item	the map $Z \ni t \mapsto \psi(t, \cdot) \in \M(Z,H)$ is continuous.
	\end{enumerate}
\end{prop}
\begin{proof}
	Suppose $\hat{g}_n \to t$.
	We want to show that $\sum_{i=1}^3{k_i \sigma_{a_ig_n}}$ converges in $\M(Z,H)$.
	Equivalently (see Lemma \ref{lem: conv in meas}), for every $\chi \in \hat{H}$,
	\begin{align*}
		\chi \left( \sum_{i=1}^3{k_i \sigma_{a_ig_n}(z)} \right)
	\end{align*}
	converges in $L^2(Z)$.
	
	The cocycle $\chi \circ \sigma$ is quasi-affine, so by Lemma \ref{lem: HK L3},
	there are affine functions $\omega_n(z) = c_n \lambda_n(z)$ such that
	$\left( \omega_n \chi \circ \sigma_{g_n} \right)_{n \in \N}$ converges in $L^2(Z)$.
	Now by Lemma \ref{lem: HK L2},
	\begin{align*}
		d_{i,n} \lambda_n^{a_i}(z) \chi \left( \sigma_{a_ig_n}(z) \right)
	\end{align*}
	converges, where
	\begin{align*}
		d_{i,n} = c_n^{a_i} \lambda_n \left( \binom{a_i}{2} \hat{g}_n \right).
	\end{align*}
	Therefore,
	\begin{align*}
		\prod_{i=1}^3{d_{i,n}^{k_i} \lambda_n^{k_ia_i}(z) \chi^{k_i} \left( \sigma_{a_ig_n}(z) \right)}
	\end{align*}
	converges in $L^2(Z)$.
	But, applying \eqref{eq: k identity 1},
	\begin{align*}
		\prod_{i=1}^3{d_{i,n}^{k_i}} = 1 \quad \text{and} \quad \prod_{i=1}^3{\lambda_n^{k_ia_i}(z)} = 1,
	\end{align*}
	so
	\begin{align*}
		\chi \left( \sum_{i=1}^3{k_i \sigma_{a_ig_n}(z)} \right)
		= \prod_{i=1}^3{\chi^{k_i} \left( \sigma_{a_ig_n}(z) \right)}
		= \prod_{i=1}^3{d_{i,n}^{k_i} \lambda_n^{k_ia_i}(z) \chi^{k_i} \left( \sigma_{a_ig_n}(z) \right)}
	\end{align*}
	converges as desired.
\end{proof}

\begin{rem}
	By the cocyle equation, $\sigma_0(z) = 0$ for $z \in Z$.
	Hence, by property (1), we have $\psi(0,z) = 0$.
\end{rem}

Now we prove Theorem \ref{thm: limit formula}:
\begin{proof}[Proof of Theorem \ref{thm: limit formula}]
	We will write elements of $X$ as $x = (z,h) \in Z \times H$.
	By a standard approximation argument, it suffices to consider $f_i = \omega_i \otimes \chi_i$
	with $\omega_i \in L^{\infty}(Z)$ and $\chi_i \in \hat{H}$.
	Write $\tilde{\chi} = (\chi_1, \chi_2, \chi_3)$.
	By Corollary \ref{cor: HK C4}, the right-hand side of \eqref{eq: limit formula} is equal to
	\begin{align*}
		\left( \int_Z{\prod_{i=1}^3{\omega_i(z + a_it)\chi_i(h)\chi_i(b_i\psi(t,z))}~dt} \right)
		\left( \int_M{\tilde{\chi}~dm_M} \right).
	\end{align*}
	
	\bigskip
	
	First consider the case $\tilde{\chi} \notin M^{\perp}$.
	Then $\int_M{\tilde{\chi}~dm_M} = 0$, so the right-hand side of \eqref{eq: limit formula} is $0$.
	On the other hand, by Proposition \ref{prop: Mackey ergodicity},
	the left-hand side of \eqref{eq: limit formula} is also $0$.
	
	\bigskip
	
	Now suppose $\tilde{\chi} \in M^{\perp}$.
	Then $\tilde{\chi}|_M = 1$, so $\int_M{\tilde{\chi}~dm_M} = 1$.
	By Theorem \ref{thm: HK T10}, there is a character $\chi \in \hat{H}$ such that $\chi_i = \chi^{k_i}$,
	so the right-hand side of \eqref{eq: limit formula} is equal to
	\begin{align*}
		\int_Z{\chi \left( \sum_{i=1}^3{k_ih} \right) \chi \left( \psi(t,z) \right) \prod_{i=1}^3{\omega_i(z + a_it)}~dt}.
	\end{align*}
	Here we have used $\sum_{i=1}^3{b_ik_i} = 1$ to simplify the expressions involving $\psi$.
	
	It remains to compute the left-hand side of \eqref{eq: limit formula}.
	We have
	\begin{align*}
		\prod_{i=1}^3{f_i \left( T_{a_ig}x \right)}
		& = \prod_{i=1}^3{\omega_i \left( z + a_i\hat{g} \right) \chi_i \left( h + \sigma_{a_ig}(z) \right)} \\
		& = \chi \left( \sum_{i=1}^3{k_ih} \right) \chi \left( \sum_{i=1}^3{k_i\sigma_{a_ig}(z)} \right)
		\prod_{i=1}^3{\omega_i \left( z + a_i \hat{g} \right)} \\
		& = \varphi_{\hat{g}}(x),
	\end{align*}
	where $\varphi : Z \times X$ is the function
	\begin{align*}
		\varphi_t(x) := \chi \left( \sum_{i=1}^3{k_ih} \right) \chi \left( \psi(t,z) \right)
		\prod_{i=1}^3{\omega_i \left( z + a_it \right)}.
	\end{align*}
	By Proposition \ref{prop: HK P14}, the map $Z \ni t \mapsto \varphi_t \in L^2(\mu)$ is continuous.
	Since $z \mapsto z + \hat{g}$ is uniquely ergodic (see Lemma \ref{lem: compact group rotations}), it follows that, for any $\xi \in L^2(\mu)$,
	\begin{align*}
		\UClim_{g \in G}{\innprod{\varphi_{\hat{g}}}{\xi}} = \int_Z{\innprod{\varphi_t}{\xi}~dt}.
	\end{align*}
	That is,
	\begin{align*}
		\UClim_{g \in G}{\varphi_{\hat{g}}(x)} = \int_Z{\varphi_t(x)~dt}
	\end{align*}
	weakly in $L^2(\mu)$.
	
	Hence, the formula \eqref{eq: limit formula} holds weakly in $L^2(\mu)$.
	By more general results on norm convergence of multiple ergodic averages (see \cite{austin, zorin}),
	\eqref{eq: limit formula} also holds in norm.
	
\end{proof}


\section{Large intersections for double recurrence} \label{sec: dK}

We want to show the following Khintchine-type theorem for double recurrence:
\DoubleKhintchine*

As discussed in the introduction, the case $G = \Z$ was established in \cite{bhk} for $\varphi(n) = n$ and $\psi(n) = 2n$ and extended in \cite{frathree} to all admissible pairs.
It was shown in \cite{btz2} that Theorem \ref{thm: double Khintchine} holds in $G = \F_p^{\infty}$
when $\varphi$ and $\psi$ are of the form $g \mapsto cg$ with $c \in \Z$.
Though our situation is a significant generalization of these two special cases,
we are able to use the method in \cite{btz2} in order to deduce our result on large intersections
from the limit formula \eqref{eq: ET}.
First we prove a lemma:
\begin{lem} \label{lem: Kronecker with a twist}
	Let $\X$ be an ergodic system.
	Let $f_0, f_1, f_2 \in L^{\infty}(\mu)$, and let $\{\varphi,\psi\}$ be an admissible pair of homomorphisms.
	Then for every continuous function $\eta : Z_{\varphi,\psi} \to \C$,
	\begin{align*}
		\UClim_{g \in G}\eta \left( \hat{\varphi(g)}, \hat{\psi(g)} \right) &
		\int_X{f_0 \cdot T_{\varphi(g)} f_1 \cdot T_{\psi(g)} f_2~d\mu}
		\\  & = \int_{Z_{\varphi,\psi}}
		{\int_Z{ \eta(u,v) \tilde{f}_0(z) \tilde{f}_1(z+u) \tilde{f}_2(z+v)~dz~d\nu_{\varphi,\psi}(u,v)}},
	\end{align*}
	where $\tilde{f}_i$ is the projection of $f_i$ onto the Kronecker factor.
\end{lem}
\begin{proof}
	The trick is to absorb $\eta$ into the functions $f_i$ and then apply Theorem \ref{thm: Kronecker}.
	First observe that if $\eta : Z_{\varphi,\psi} \to \C$ is continuous, then $\eta$ extends to a continuous
	function $\eta_0$ on $Z^2$, since $Z_{\varphi,\psi} \subseteq Z^2$ is closed.
	Then by the Stone--Weierstrass theorem, it suffices to consider the case when
	$\eta_0$ is a character on $Z^2$.
	That is, $\eta_0 \in \hat{Z}^2 = \Lambda^2$.
	Thus, we can assume $\eta(u,v) = \lambda_1(u) \lambda_2(v)$
	with $\lambda_1, \lambda_2 \in \Lambda$.
	
	Now define functions $h_i$ by
	\begin{align*}
		h_0(x) & = \overline{\lambda_1(z)} \overline{\lambda_2(z)} f_0(x), \\
		h_1(x) & = \lambda_1(z) f_1(x), \\
		h_2(x) & = \lambda_2(z) f_2(x),
	\end{align*}
	where $x \mapsto z$ is the projection onto the Kronecker factor.
	The projections $\tilde{h}_i = \E{h_i}{Z}$ satisfy similar identities in terms of $\tilde{f}_i$.
	Now we can compute the limit using Theorem \ref{thm: Kronecker}:
	\begin{align*}
		\UClim_{g \in G}{\eta \left( \hat{\varphi(g)}, \hat{\psi(g)} \right)} &
		\int_X{f_0 \cdot T_{\varphi(g)} f_1 \cdot T_{\psi(g)} f_2~d\mu} \\
		= & ~\UClim_{g \in G}{\int_X{h_0 \cdot T_{\varphi(g)} h_1 \cdot T_{\psi(g)} h_2~d\mu}} \\
		= & \int_X{h_0(x)
			\left( \UClim_{g \in G}{T_{\varphi(g)} h_1(x) \cdot T_{\psi(g)} h_2(x)} \right)~d\mu(x)} \\
		= & \int_X{h_0(x) \left( \int_{Z_{\varphi,\psi}}{
				\tilde{h}_1(z+u) \tilde{h}_2(z+v)~d\nu_{\varphi,\psi}(u,v)} \right)~d\mu(x)} \\
		= & \int_Z{\int_{Z_{\varphi,\psi}}{\tilde{h}_0(z) \tilde{h}_1(z+u) \tilde{h}_2(z+v)
				~d\nu_{\varphi,\psi}(u,v)}~dz} \\
		= & \int_{Z_{\varphi,\psi}}{\int_Z{ \eta(u,v)
				\tilde{f}_0(z) \tilde{f}_1(z+u) \tilde{f}_2(z+v)~dz~d\nu_{\varphi,\psi}(u,v)}}.
	\end{align*}
\end{proof}

Now we prove Theorem \ref{thm: double Khintchine}:
\begin{proof}[Proof of Theorem \ref{thm: double Khintchine}]
	Let $\eps > 0$.
	Let $R_{\eps} := \left\{ g \in G : \mu \left( A \cap T_{\varphi(g)}^{-1}A \cap T_{\psi(g)}^{-1}A \right)
	> \mu(A)^3 - \eps \right\}$.
	Suppose for contradiction that $R_{\eps}$ is not syndetic.
	Then by Lemma \ref{lem: syndetic thick} there is a F{\o}lner sequence $(F_N)_{N \in \N}$ such that
	\begin{align} \label{eq: non-recurrence}
		\mu \left( A \cap T_{\varphi(g)}^{-1}A \cap T_{\psi(g)}^{-1}A \right) \le \mu(A)^3 - \eps
	\end{align}
	for every $N \in \N$ and $g \in F_N$.
	
	Let $f : Z \to \C$ be the projection of $\ind_A$ onto the Kronecker factor.
	Since $\ind_A$ is a nonnegative function, we have $f = \E{\ind_A}{Z} \ge 0$.
	Now by Jensen's inequality,
	\begin{align*}
		\int_Z{f(z) f(z) f(z)~dz} \ge \left( \int_Z{f~dz} \right)^3 = \mu(A)^3.
	\end{align*}
	Therefore,
	\begin{align*}
		\int_Z{f(z)f(z+u)f(z+v)~dz} \ge \mu(A)^3 - \frac{\eps}{2}
	\end{align*}
	for $(u,v)$ in some neighborhood of 0 in $Z_{\varphi,\psi}$.
	Hence, by Urysohn's lemma, there is a continuous function $\eta : Z_{\varphi,\psi} \to [0, \infty)$
	with $\int_{Z_{\varphi,\psi}}{\eta~d\nu_{\varphi,\psi}} = 1$ such that
	\begin{align*}
		\int_{Z_{\varphi,\psi}}{\int_Z{ \eta(u,v) f(z) f(z+u) f(z+v)~dz~d\nu_{\varphi,\psi}(u,v)}}
		\ge \mu(A)^3 - \frac{\eps}{2}.
	\end{align*}
	Applying Lemma \ref{lem: Kronecker with a twist}, we conclude
	\begin{align*}
		& \lim_{N \to \infty}{\frac{1}{|F_N|} \sum_{g \in F_N}{
				\eta \left( \hat{\varphi(g)}, \hat{\psi(g)} \right)
				\mu \left( A \cap T_{\varphi(g)}^{-1}A \cap T_{\psi(g)}^{-1}A \right)}} \\
		= & \lim_{N \to \infty}{\frac{1}{|F_N|} \sum_{g \in F_N}{
				\eta \left( \hat{\varphi(g)}, \hat{\psi(g)} \right)
				\int_X{\ind_A \cdot T_{\varphi(g)}\ind_A \cdot T_{\psi(g)}\ind_A~d\mu}}} \\
		\ge & ~\mu(A)^3 - \frac{\eps}{2}.
	\end{align*}
	
	On the other hand, by \eqref{eq: non-recurrence}, we have
	\begin{align*}
		& \limsup_{N \to \infty}{\frac{1}{|F_N|} \sum_{g \in F_N}{
				\eta \left( \hat{\varphi(g)}, \hat{\psi(g)} \right)
				\mu \left( A \cap T_{\varphi(g)}^{-1}A \cap T_{\psi(g)}^{-1}A \right)}} \\
		\le & \left( \mu(A)^3 - \eps \right)
		\lim_{N \to \infty}{\frac{1}{|F_N|} \sum_{g \in F_N}{\eta \left( \hat{\varphi(g)}, \hat{\psi(g)} \right)}} \\
		= & \left( \mu(A)^3 - \eps \right) \left( \int_{Z_{\varphi,\psi}}{\eta~d\nu_{\varphi,\psi}} \right) \\
		= & ~\mu(A)^3 - \eps.
	\end{align*}
	
	This is a contradiction, so $R_{\eps}$ must be syndetic for every $\eps > 0$.
\end{proof}


\section{Large intersections for triple recurrence} \label{sec: tK}

We will now show the following:
\TripleKhintchine*

In order to apply the limit formula in Theorem \ref{thm: limit formula},
we first need to reduce to studying multiple ergodic averages in a quasi-affine system.
By Theorem \ref{thm: CL characteristic}, the multiple ergodic averages
\begin{align*}
	\UClim_{g \in G}{T_{rg}f_1 \cdot T_{sg}f_2 \cdot T_{(r+s)g}f_3}
\end{align*}
are controlled by the $(r,s)$-CL factor $\B_{CL(r,s)}$.
We will reduce from $(r,s)$-CL cocycles to quasi-affine ($(1,1)$-CL) cocycles by passing to a finite index subgroup
with the help of the following two technical lemmas.
Recall that $CL_{\X}(r,s)$ denotes the group of all CL cocycles for the pair of homomorphisms
$g \mapsto rg$ and $g \mapsto sg$ (see Definition \ref{defn: CL cocycle}).

\begin{lem} \label{lem: CL product}
	Suppose $\rho \in CL_{\X}(r,s)$.
	Then for every $r', s' \in \Z \setminus \{0\}$, we have $\rho \in CL_{\X}(rr',ss')$.
\end{lem}
\begin{proof}
	Since $\rho$ is a $(r,s)$-CL cocycle, it satisfies the Conze--Lesigne equation
	\begin{align*}
		\frac{\rho_{rg}(z+u)}{\rho_{rg}(z)}
		= \Lambda_u(z + rZ)(g) \frac{K_u \left( z + r\hat{g} \right)}{K_u(z)}
	\end{align*}
	for all $g \in G$ and almost every $z \in Z, u \in sZ$.
	Hence, replacing $g$ with $r'g$ and $u$ with $ss'w$, we have
	\begin{align} \label{eq: CL rs}
		\frac{\rho_{rr'g}(z+ss'w)}{\rho_{rr'g}(z)}
		= \Lambda_{ss'w}(z + rZ)(r'g) \frac{K_{ss'w} \left( z + rr'\hat{g} \right)}{K_{ss'w}(z)}
	\end{align}
	for all $g \in G$ and almost every $z, w \in Z$.
	Now, let us define a new function $\tilde{\Lambda} : ss'Z \times Z/rr'Z \to \hat{G}$ by
	\begin{align} \label{eq: new Lambda}
		\tilde{\Lambda}_{ss'w}(z + rr'Z)(g) := \Lambda_{ss'w}(z + rZ)(r'g)
	\end{align}
	for $g \in G$ and $z, w \in Z$.
	Note that $\tilde{\Lambda}$ is well-defined in $z$ because $rZ \supseteq rr'Z$.
	Thus, substituting \eqref{eq: new Lambda} into \eqref{eq: CL rs},
	$\rho$ satisfies the Conze--Lesigne equation with parameters $(rr', ss')$:
	\begin{align*}
		\frac{\rho_{rr'g}(z+ss'w)}{\rho_{rr'g}(z)}
		= \tilde{\Lambda}_{ss'w}(z + rr'Z)(g) \frac{K_{ss'w} \left( z + rr'\hat{g} \right)}{K_{ss'w}(z)}
	\end{align*}
	for $g \in G$ and almost every $z, w \in Z$.
\end{proof}

\begin{lem} \label{lem: CL to QA}
	Let $a \in \Z$ such that $aG$ has finite index in $G$.
	Suppose $\rho \in CL_{\X}(a,a)$.
	Let $S_g := T_{ag}$ for $g \in G$, and let $\tau_g = \rho_{ag}$.
	Then on each of the finitely many ergodic components of $\left( X, \B, \mu, (S_g)_{g \in G} \right)$,
	the cocycle $\tau$ is quasi-affine.
\end{lem}
\begin{proof}
	For each $g \in G$, let $\alpha_g = a\hat{g}$, and let $\tilde{Z} := aZ = \overline{\left\{ \alpha_g : g \in G \right\}}$.
	By Lemma \ref{lem: compact group rotations}, each ergodic component of $S$ has Kronecker factor
	isomorphic to $\tilde{Z}$.
	In fact, the Kronecker factor of each ergodic component appears as a coset
	$\tilde{Z} + t \subseteq Z$ for some $t \in Z$.
	
	Now, the Conze--Lesigne equation with parameters $(a,a)$ reads
	\begin{align*}
		\frac{\rho_{ag}(z + u)}{\rho_{ag}(z)}
		= \Lambda_u \left( z + aZ \right)(g) \frac{K_u \left( z + a\hat{g} \right)}{K_u(z)}
	\end{align*}
	for $g \in G$, almost every $z \in Z$ and almost every $u \in aZ$.
	For a fixed coset $\tilde{Z} + t$, letting $\gamma_u(g) := \Lambda_u(t + \tilde{Z})(g)$, we therefore have
	\begin{align*}
		\frac{\tau_g(z+u)}{\tau_g(z)} = \gamma_u(g) \frac{K_u(z + \alpha_g)}{K_u(z)}
	\end{align*}
	for $g \in G$, almost every $z \in \tilde{Z} + t$, and almost every $u \in \tilde{Z}$.
	That is, $\tau : G \times (\tilde{Z} + t) \to S^1$ is a quasi-affine cocycle.
\end{proof}

\begin{proof}[Proof of Theorem \ref{thm: triple Khintchine}]
	Without loss of generality, we may assume that $\X$ is normal (defined in Section \ref{sec: abelian ext}).
	Indeed, if Theorem \ref{thm: triple Khintchine} holds for a system, then it trivially holds for every factor,
	so we can replace $\X$ by a normal extension, if needed, using Proposition \ref{prop: normal ext}.
	Since $\X$ is normal, it follows by Theorem \ref{thm: CL characteristic} that the averages
	\begin{align*}
		\UClim_{g \in G}{\mu \left( A \cap T_{ag}^{-1}A \cap T_{bg}^{-1}A \cap T_{(a+b)g}^{-1}A \right)}
	\end{align*}
	are controlled by a Conze--Lesigne factor $\B_{CL(a,b)}$ for every $a, b \in \Z$
	such that $\{a, b, a+b\}$ is admissible.
	By Lemma \ref{lem: CL product}, $\B_{CL(a,b)} \subseteq \B_{CL(aa',bb')}$.
	Therefore, if we take an inverse limit $\D := \bigvee_{n \ge 0}{\D_n}$
	with $\D_n = \B_{CL((rs)^n, (rs)^n)}$, then
	\begin{align*}
		& \UClim_{g \in G}{\mu \left( A \cap T_{r(rs)^ng}^{-1}A \cap T_{s(rs)^ng}^{-1}A
			\cap T_{(r+s)(rs)^ng}^{-1}A \right)} \\
		& = \UClim_{g \in G}{\int_X{\E{\ind_A}{\D} \cdot T_{r(rs)^ng}\E{\ind_A}{\D}
				\cdot T_{s(rs)^ng}\E{\ind_A}{\D}\cdot T_{(r+s)(rs)^ng}\E{\ind_A}{\D}~d\mu}}
	\end{align*}
	for every $n \ge 0$.
	Fix $N \in \N$ so that $\norm{L^1(\mu)}{\E{\ind_A}{\D_N} - \E{\ind_A}{\D}} < \frac{\eps}{8}$.
	Then for every $g \in G$, repeated application of H\"{o}lder's inequality and the triangle inequality gives
	\begin{align} \label{eq: inverse lim approx}
		& \left| \int_X{\E{\ind_A}{\D} \cdot T_{r(rs)^ng}\E{\ind_A}{\D}
			\cdot T_{s(rs)^ng}\E{\ind_A}{\D}\cdot T_{(r+s)(rs)^ng}\E{\ind_A}{\D}~d\mu} \right. \nonumber \\
		& - \left. \int_X{\E{\ind_A}{\D_N} \cdot T_{r(rs)^ng}\E{\ind_A}{\D_N}
			\cdot T_{s(rs)^ng}\E{\ind_A}{\D_N}\cdot T_{(r+s)(rs)^ng}\E{\ind_A}{\D_N}~d\mu} \right|
		< \frac{\eps}{2}.
	\end{align}
	Let $\X_{CL((rs)^N, (rs)^N)} = \mathbf{Z} \times_{\sigma} H$ denote the system corresponding to the factor $\D_N$.
	
	\bigskip
	
	Now we restrict to a finite index subgroup and consider the finitely many ergodic components
	for the corresponding subaction.
	Let $d := \gcd(r, s, r + s)$, and set $a_1 := \frac{r}{d}$, $a_2 := \frac{s}{d}$, and $a_3 := \frac{r+s}{d}$.
	For notational convenience, we will also take $a_0 = 0$.
	We will consider the action $S_g = T_{d(rs)^Ng}$ for $g \in G$.
	This action has finitely many ergodic components, since $[G : d(rs)^NG] < \infty$, and we have the identity
	\begin{align*}
		S_{a_0g}^{-1}A \cap S_{a_1g}^{-1}A \cap S_{a_2g}^{-1}A \cap S_{a_3g}^{-1}A
		= A \cap T_{r(rs)^Ng}^{-1}A \cap T_{s(rs)^Ng}^{-1}A \cap T_{(r+s)(rs)^Ng}^{-1}A
	\end{align*}
	for $g \in G$.
	
	We proceed to describe the ergodic decomposition.
	Let $X_1, \dots, X_m$ be the finitely many atoms of the invariant $\sigma$-algebra for $S$
	so that the ergodic decomposition is given by $\mu = \frac{1}{m} \sum_{j=1}^m{\mu_i}$ with
	$\mu_j(A) = \frac{\mu(A \cap X_j)}{\mu(X_j)} = m \cdot \mu(A \cap X_j)$.
	By Lemma \ref{lem: compact group rotations}, the (ergodic) system
	$\X_j = \left( X_j, \B \cap X_j, \mu_j, (S_g)_{g \in G} \right)$
	has Kronecker factor $Z_j = \overline{\left\{ \alpha_g + t_j : g \in G \right\}}$,
	where $\alpha_g = d(rs)^N\hat{g}$ is given by the action of $S$ on $\mathbf{Z}$ and $t_j \in Z$.
	Put $Z_0 := \overline{\left\{ \alpha_g : g \in G \right\}}$ so that $Z_j = Z_0 + t_j$.
	Observe that by Lemma \ref{lem: CL to QA}, the restriction of $\X_j$ to the factor $\D_N$
	is a quasi-affine system $\mathbf{Z}_j \times_{\tau^{(j)}} H_j$.
	
	\bigskip
	
	With all of this setup, we can prove a ``twisted'' version of Theorem \ref{thm: limit formula}.
	Recall the notation from \ref{thm: limit formula}.
	We let $k'_1 = - rs(r+s)$, $k'_2 = rs(r+s)$, and $k'_3 = -rs(s-r)$.
	Set $D := \gcd(k'_1, k'_2, k'_3) = rs \gcd(r+s, s-r)$ and $k_i = \frac{k'_i}{D}$.
	Finally, let $b_1, b_2, b_3 \in \Z$ so that $\sum_{i=1}^3{k_ib_i} = 1$, and take $b_0 = 0$.
	\begin{cla} \label{cla: twisted formula}
		Let $\eta : Z_0 \to \C$ be a continuous function.
		Then for any $f_0, f_1, f_2, f_3 \in L^{\infty}(Z \times H)$
		and each $j = 1, \dots, m$, we have
		\begin{align}
			& \UClim_{g \in G}{\eta(\alpha_g) \int_X{\prod_{i=0}^3{S_{a_ig}f_i}~d\mu_j}} \nonumber \\
			& = \int_{Z_0^2 \times H_j^3}{\eta(t)~
				\prod_{i=0}^3{f_i(z + t_j + a_it, h + a_iu + a_i^2v + b_i \psi_j(t,z+t_j))}~dh~du~dv~dz~dt}.
		\end{align}
	\end{cla}
	\begin{proof}[Proof of Claim.]
		By the Stone--Weierstrass theorem and linearity,
		it suffices to prove the identity in the case $\eta = \lambda \in \hat{Z_0}$.
		
		By construction, $\gcd(a_1, a_2, a_3) = 1$, so let $c_i \in \Z$ such that $\sum_{i=1}^3{c_ia_i} = 1$.
		Set $c_0 := -(c_1+c_2+c_3)$ so that $\sum_{i=0}^3{c_i} = 0$.
		Since $a_0 = 0$, this also gives $\sum_{i=0}^3{c_ia_i} = 1$.
		Define $\omega_i \in L^{\infty}(Z_j \times H_j)$ by
		\begin{align*}
			\omega_i(x) = \lambda(c_i(z-t_j)) f_i(x)
		\end{align*}
		for $x = (z,h) \in Z_j \times H_j$.
		
		Observe:
		\begin{align*}
			\prod_{i=0}^3{\omega_i(S_{a_ig}x)}
			& = \prod_{i=0}^3{\lambda(c_i(z-t_j) + c_i\alpha_{a_ig}) f_i(S_{a_ig}x)} \\
			& = \lambda\left( \sum_{i=0}^3{c_i}(z-t_j) \right) \lambda \left( \sum_{i=0}^3{c_ia_i}\alpha_g \right)
			\prod_{i=0}^3{f_i(S_{a_ig}x)} \\
			& = \lambda(\alpha_g) \prod_{i=0}^3{f_i(S_{a_ig}x)}.
		\end{align*}
		Similarly, for $z \in Z_j$, $h, u, v \in H_j$, and $t \in Z_0$, we have
		\begin{align*}
			\prod_{i=0}^3{\omega_i(z + a_it, h + a_iu + a_i^2v + b_i \psi_j(t,z))}
			& = \prod_{i=0}^3{\lambda(c_i(z-t_j) + c_ia_it) f_i(z + a_it, h + a_iu + a_i^2v + b_i \psi_j(t,z))} \\
			& = \lambda(t) \prod_{i=0}^3{f_i(z + a_it, h + a_iu + a_i^2v + b_i \psi_j(t,z))}. \\
		\end{align*}
		The result immediately follows by applying Theorem \ref{thm: limit formula} to the functions $\omega_i$
		and noting that $\mu_j$ projects to the Haar measure\footnote{By ``Haar measure,'' we mean the unique Borel
			probability measure invariant under shifts from the group $Z_0 \times H_j$.} on $Z_j \times H_j$.
	\end{proof}
	
	\bigskip
	
	Now let $f := \ind_A \in L^{\infty}(\mu)$.
	Write $\tilde{f} = \E{f}{Z \times H}$.
	For each $j = 1, \dots, m$, we can find a small neighborhood $U_j$ of $0$ in $Z_0$ so that
	\begin{align*}
		& \int_{Z_0 \times H_j^3}{\prod_{i=0}^3{\tilde{f}(z + t_j + a_it, h + a_iu + a_i^2v + b_i \psi_j(t,z+t_j))}
			~dh~du~dv~dz}\\
		& \ge \int_{Z_0 \times H_j^3}{\prod_{i=0}^3{\tilde{f}(z + t_j, h + a_iu + a_i^2v)}~dh~du~dv~dz} - \frac{\eps}{4}
	\end{align*}
	for $t \in U_j$.
	By Urysohn's lemma, we can then find a continuous function $\eta : Z_0 \to [0, \infty)$
	with $\int_{Z_0}{\eta} = 1$ concentrated on the neighborhood $U = \bigcap_{j=1}^m{U_j}$ so that,
	by Claim \ref{cla: twisted formula},
	\begin{align*}
		\UClim_{g \in G}{\eta(\alpha_g)~
			\int_X{\prod_{i=0}^3{S_{a_ig} \tilde{f}}~d\mu_j}}
		\ge \int_{Z_0 \times H_j^3}{\prod_{i=0}^3{\tilde{f}(z + t_j, h + a_iu + a_i^2v)}~dh~du~dv~dz} - \frac{\eps}{4}
	\end{align*}
	for every $j = 1, \dots, m$.
	
	We now want to show the inequality
	\begin{align} \label{eq: CL factor bound}
		\int_{Z_0 \times H_j^3}{\prod_{i=0}^3{\tilde{f}(z + t_j, h + a_iu + a_i^2v)}~dh~du~dv~dz}\ge \mu_j(A)^4.
	\end{align}
	Fix $z \in Z_j$ and let $F : H_j \to \C$ be the function $F(h) = \tilde{f}(z,h)$.
	We will show
	\begin{align*}
		\int_{H_j^3}{\prod_{i=0}^3{F(h + a_iu + a_i^2v)}~dh~du~dv} \ge \left( \int_{H_j}{F~dm_{H_j}} \right)^4.
	\end{align*}
	Note that \eqref{eq: CL factor bound} follows from this inequality by an application of Jensen's inequality.
	
	By Lemma \ref{lem: H divisibility}, $a_iH_j = H_j$ for each $i = 1, 2, 3$.
	We will use this fact to make a sequence of substitutions.
	First, take $h = a_3x$:
	\begin{align*}
		& \int_{H_j^3}{\prod_{i=0}^3{F(h + a_iu + a_i^2v)}~dh~du~dv} \\
		& = \int_{H_j^3}{F(a_3x) F\left( a_3\left( x + u + a_3v \right) \right)
			F(a_3x + a_1u + a_1^2v) F(a_3x + a_2u + a_2^2v)~du~dx~dv} \\
		\intertext{Next, $x + u + a_3v = y$:}
		& = \int_{H_j^3}{F(a_3x) F(a_3y) F(a_2x + a_1y - a_1a_2v) F(a_1x + a_2y - a_1a_2v)~dv~dx~dy} \\
		\intertext{Now, $a_1(x+y) - a_1a_2v = z$:}
		& = \int_{H_j^3}{F(a_3x) F(a_3y) F((a_2-a_1)x + z) F((a_2-a_1)y + z)~dx~dy~dz} \\
		& = \int_{H_j}{\left( \int_{H_j}{F(a_3x) F \left( (a_2 - a_1)x + z \right)~dx}\right)^2~dz} \\
		\intertext{Applying Jensen's inequality:}
		& \ge \left( \int_{H_j^2}{F(a_3x) F \left( (a_2 - a_1)x + z \right)~dz~dx} \right)^2 \\
		\intertext{Finally, let $z + (a_2-a_1)x = w$ and $a_3x = t$:}
		& = \left( \int_{H_j}{F(t)~dt} \right)^2 \left( \int_{H_j}{F(w)~dw} \right)^2 \\
		& = \left( \int_{H_j}{F~dm_{H_j}} \right)^4.
	\end{align*}
	Each of these substitutions is measure-preserving because translations
	and continuous surjective homomorphisms preserve the Haar measure (see, e.g., \cite[Section 1.1]{walters}).
	Each change in the order of integration is justified by Fubini's theorem so long as $F \in L^{\infty}(H_j)$,
	which is true for almost every $z \in Z_j$.
	
	\bigskip
	
	Now we finish the proof.
	We have shown
	\begin{align*}
		\UClim_{g \in G}{\eta(d(rs)^N\hat{g})~
			\int_X{\prod_{i=0}^3{T_{da_i(rs)^Ng} \tilde{f}}~d\mu_j}}
		= \UClim_{g \in G}{\eta(\alpha_g)~
			\int_X{\prod_{i=0}^3{S_{a_ig} \tilde{f}}~d\mu_j}}
		\ge \mu_j(A)^4 - \frac{\eps}{4}
	\end{align*}
	for each $j = 1, \dots, m$.
	Applying Jensen's inequality, it follows that
	\begin{align*}
		\UClim_{g \in G}{\eta(d(rs)^N\hat{g})~
			\int_X{\prod_{i=0}^3{T_{da_i(rs)^Ng} \tilde{f}}~d\mu}}
		& = \frac{1}{m} \sum_{j=1}^m{\UClim_{g \in G}{\eta(d(rs)^N\hat{g})~
				\int_X{\prod_{i=0}^3{T_{da_i(rs)^Ng} \tilde{f}}~d\mu_j}}} \\
		& \ge \frac{1}{m} \sum_{j=1}^m{\mu_j(A)^4} - \frac{\eps}{4} \\
		& \ge \left( \frac{1}{m} \sum_{j=1}^m{\mu_j(A)} \right)^4 - \frac{\eps}{4} \\
		& = \mu(A)^4 - \frac{\eps}{4}.
	\end{align*}
	Moreover, by \eqref{eq: inverse lim approx},
	\begin{align*}
		\left| \UClim_{g \in G}{\eta(d(rs)^N\hat{g})~
			\mu \left( A \cap T_{r(rs)^Ng}^{-1}A \cap T_{s(rs)^Ng}^{-1}A \cap T_{(r+s)(rs)^Ng}^{-1}A \right)} \right. \\
		- \left. \UClim_{g \in G}{\eta(d(rs)^N\hat{g})~
			\int_X{\prod_{i=0}^3{T_{da_i(rs)^Ng} \tilde{f}}~d\mu}} \right|
		\le \frac{\eps}{2}.
	\end{align*}
	Thus,
	\begin{align*}
		\UClim_{g \in G}{\eta(d(rs)^N\hat{g})~
			\mu \left( A \cap T_{r(rs)^Ng}^{-1}A \cap T_{s(rs)^Ng}^{-1}A \cap T_{(r+s)(rs)^Ng}^{-1}A \right)}
		\ge \mu(A)^4 - \frac{3\eps}{4}.
	\end{align*}

	Arguing by contradiction and following the strategy in the proof of Theorem \ref{thm: double Khintchine},
	this inequality demonstrates that the set
	\begin{align*}
		\left\{ g \in G : \mu \left( A \cap T_{r(rs)^Ng}^{-1}A \cap T_{s(rs)^Ng}^{-1}A \cap T_{(r+s)(rs)^Ng}^{-1}A \right)
		> \mu(A)^4 - \eps \right\}
	\end{align*}
	is syndetic in $G$.
	Since $(rs)^NG$ has finite index in $G$, this completes the proof.
\end{proof}


\section{Failure of large intersections for inadmissible families of homomorphisms} \label{sec: admissible}

From the start, we have assumed that all families of homomorphisms that we are dealing with are admissible. 
This property seems to be absolutely essential, and we give some support to this claim below,
both at the level of large intersections results and at the level of characteristic factors.

\begin{conj}
	Suppose $G$ is a countable discrete abelian group and $\varphi, \psi : G \to G$ are homomorphisms
	such that $\{ g \in G : \varphi(g) = \psi(g) \}$ has infinite index in $G$.\footnote{We include the assumption to avoid trivialities. Let $H = \{g \in G: \varphi(g) = \psi(g)\}$. By the ergodic theorem, $\UClim_{h\in H}{\mu \left( A \cap T_{\varphi(h)}^{-1}A \cap T_{\psi(h)}^{-1}A \right)} = \UClim_{h\in H}{\mu \left( A \cap T_{\varphi(h)}^{-1}A \right)} \ge \mu(A)^2$. This implies that $\left\{h \in H : \mu\left( A \cap T_{\varphi(h)}^{-1}A \cap T_{\psi(h)}^{-1}A \right) > \mu(A)^2 - \eps \right\}$ is syndetic in $H$. If $H$ has finite index in $G$, then this set is also syndetic in $G$.}
	Then $\{\varphi,\psi\}$ has the large intersections property if and only if $\{\varphi,\psi\}$ is admissible.
\end{conj}

First we show, with the help of an example due to Chu \cite{chu},
that if $\varphi$, $\psi$, and $\psi - \varphi$ all have infinite index images,
then the large intersections property may fail to hold.
\begin{eg} \label{eg: Chu}
	Let $G = \Z^2$, and let $\varphi(n,m) = (n,0)$ and $\psi(n,m) = (0,n)$.
	For ease of notation, given a $\Z^2$-action $(T_{(n,m)})_{(n,m) \in \Z^2}$,
	set $S_1 = T_{(1,0)}$ and $S_2 = T_{(0,1)}$.
	Then
	\begin{align*}
		\mu\left( A \cap T_{\varphi(n,m)}^{-1}A \cap T_{\psi(n,m)}^{-1}A \right)
		= \mu \left(A \cap S_1^{-n}A \cap S_2^{-n} A\right).
	\end{align*}
	A theorem of Chu completes the proof:
	\begin{thm}[\cite{chu}, Theorem 1.2]
		For every $c > 0$, there is a probability space $(X, \B, \mu)$,
		two commuting measure-preserving automorphisms $S_1, S_2 : X \to X$
		such that $(X, \B, \mu, S_1, S_2)$ is ergodic
		(i.e. the $\Z^2$-action generated by $S_1$ and $S_2$ is ergodic),
		and a set $A \in \B$ with $\mu(A) > 0$ such that
		\begin{align*}
			\mu \left( A \cap S_1^{-n}A \cap S_2^{-n}A \right) < c \mu(A)^3
		\end{align*}
		for every $n \ne 0$.
	\end{thm}
	Subsequent work by Donoso and Sun sharpened this bound to $\mu(A)^l$ for any $l < 4$
	(see \cite[Theorem 1.2]{ds}).
\end{eg}

Second, it is reasonable to ask whether there is a pair $\{\varphi,\psi\}$ of homomorphisms with exactly one of $\varphi, \psi, \varphi-\psi$ having infinite index image such that $\{\varphi,\psi\}$ does not have the large intersections property. Although we do not have such an example, we briefly describe how the methods of this article break down in this situation.

Indeed, we use the fact that the Kronecker factor is characteristic for admissible $\{\varphi,\psi\}$ families (Theorem~\ref{thm: Kronecker}) in order to show that $\{\varphi,\psi\}$ has the large intersections property (Theorem~\ref{thm: double Khintchine}). The dependence on this theorem presents a difficulty if we suppose that, say, $\varphi$ has infinite index image. In \cite[Section 4.13]{griesmer}, for any $d > 1$, a mixing $\Z^d$-system is constructed such that any characteristic factor for the family $\{\varphi,\psi\}$ is mixing and nontrivial. Hence, our method to show that $\{\varphi,\psi\}$ has the large intersections property breaks down at essentially the first step if we do not work with an admissible pair $\{\varphi,\psi\}$, because if Theorem~\ref{thm: Kronecker} held even in this situation, then we would conclude that a nontrivial Kronecker factor is mixing. For longer expressions, a similar situation arises to contradict a hypothetical generalization of Theorem~\ref{thm: compact ext} for $Z_k$'s that are characteristic for all families $\{\varphi_1,\ldots,\varphi_k\}$ rather than merely admissible families.


\section{Failure of large intersections for non-ergodic systems} \label{sec: ergodicity}

We have assumed throughout that our measure-preserving systems are ergodic
(or at least that the ergodic decomposition has only finitely many ergodic components).
Based on several examples we will present in this section,
it seems that this assumption cannot be dropped.\footnote{
	Contrast this with Khintchine's theorem (Theorem \ref{history 1}), which does not require ergodicity.
}
Namely, we make the following conjecture:

\begin{conj} \label{conj: ergodicity}
	Let $G$ be a countable discrete abelian group.
	Let $\varphi, \psi : G \to G$ be an admissible pair of homomorphisms.
	There exists a (necessarily non-ergodic) measure-preserving system
	$\left( X, \B, \mu, (T_g)_{g \in G} \right)$, a set $A \in \B$ with $\mu(A) > 0$, and $c < 1$ such that
	\begin{align*}
		\mu \left( A \cap T_{\varphi(g)}^{-1}A \cap T_{\psi(g)}^{-1} A \right) \le c\mu(A)^3
	\end{align*}
	for all $g \ne 0$.
\end{conj}

Some evidence for this conjecture is given in \cite{bhk}.

\begin{thm}[\cite{bhk}, Theorem 2.1] \label{thm: bhk ergodicity}
	There is a non-ergodic system $(X, \B, \mu, T)$ such that, for every integer $l \ge 1$,
	there is a set $A = A(l) \in \B$ with $\mu(A) > 0$ such that
	\begin{align*}
		\mu \left( A \cap T^nA \cap T^{2n}A \right) \le \frac{\mu(A)^l}{2}
	\end{align*}
	for every $n \ne 0$.
\end{thm}

Key in the proof of Theorem \ref{thm: bhk ergodicity} is the following combinatorial fact
due to Behrend:\footnote{
	Behrend proved the case of 3-APs ($a=1, b=2$).
	In Behrend's proof (see \cite{behrend}), $B$ is constructed as the set of points
	$x = x_0 + x_1(2d-1) + \cdots + x_{n-1}(2d-1)^{n-1}$
	whose digits satisfy $0 \le x_i \le d-1$ and $x_0^2+x_1^2+\cdots+x_{n-1}^2 = k$
	(where $d$ and $k$ are chosen so that $B$ has the desired density).
	Modifying this construction to an expansion in base $bd-1$ with the same restriction on the digits
	will produce a pattern-free set of sufficiently large density with $N = (bd-1)^n$.
}
\begin{thm}[\cite{behrend}] \label{thm: Behrend}
	Let $a, b \in \N$ be distinct and nonzero.
	For every $N \in \N$, there is a subset $B \subseteq \{0, 1, \dots, N-1\}$
	such that $|B| > Ne^{-c\sqrt{\log{N}}}$
	and $B$ contains no configuration of the form $\{n, n+am, n+bm\}$.
\end{thm}

We provide new counterexamples for three classes of groups:
free abelian groups $\Z^d$, torsion groups $(\Z/p\Z)^{\infty}$, and the group $(\Q,+)$ of rational numbers (with addition).
In the first two cases, we are able to produce systems with small intersections when
$\varphi$ and $\psi$ are multiplication by integers,
and we reduce more general situations to combinatorial problems
of a similar nature to Behrend's theorem.
As yet, we are unable to solve these combinatorial problems,
so we do not have a full solution to Conjecture \ref{conj: ergodicity} in this setting.

For the group of rational numbers, as we shall see, one can take $\varphi$ and $\psi$ to be multiplication by $r, s \in \Z$
without loss of generality, which allows for a direct application of (a generalized form of) Behrend's theorem.
We are therefore able to verify Conjecture \ref{conj: ergodicity} for $\Q$
(see Theorem \ref{thm: Q ergodicity} below).

The proof technique in all three cases is very much in the spirit of the proof of
Theorem \ref{thm: bhk ergodicity} given in \cite{bhk}.
However, there are some \emph{ad hoc} modifications that must be made in each case,
so we are unable to state a unifying result that applies to all of these examples at once.

\subsection{Free abelian groups of finite rank}

We will consider a special class of endomorphisms of $\Z^d$.
Namely, we will restrict our attention to pairs of homomorphisms arising as
multiplication in a ring of integers.
This additional algebraic structure allows us to employ additional tools that are not available
for general admissible pairs of homomorphisms.

Let $K$ be a number field (finite algebraic extension of $\Q$), and let $\O_K$ be the ring of integers of $K$.
Let $b_1, \dots, b_d \in \O_K$ be an integral basis\footnote{
	This means that every element $x \in \O_K$ has a unique representation of the form
	$x = \sum_{i=1}^d{a_ib_i}$ with $a_i \in \Z$.} for $\O_K$.

To provide a connection with combinatorics in $\Z^d$, we need to represent $\O_K$ by matrices.
We have an isomorphism of the additive groups $\O_K \cong \Z^d$ given by
$\sum_{i=1}^d{a_ib_i} \mapsto (a_1, a_2, \dots, a_d)$.
For each $x \in \O_K$, the map $m_x : y \mapsto xy$ is $\Z$-linear.
It can therefore be represented by an integer matrix $M_x$ in the basis $b_1, \dots, b_d$.

\begin{eg}
	Let $K = \Q(\zeta)$, where $\zeta$ is a $p$th root of unity for some prime $p$.
	Then $\O_K = \Z[\zeta]$ has integral basis $1, \zeta, \dots, \zeta^{p-2}$.
	Now, the minimal polynomial for $\zeta$ is the cyclotomic polynomial
	$\Phi_p(x) = 1+x+x^2+\cdots+x^{p-1}$, so we can easily compute the multiplication matrix $M_{\zeta}$
	in this basis:
	\begin{align*}
		M_{\zeta} := \left( \begin{array}{ccccc}
			0 & 0 & \cdots & 0 & -1 \\
			1 & 0 & \cdots & 0 & -1 \\
			0 & 1 & \ddots & 0 & -1 \\
			\vdots & \vdots & \ddots & \vdots & \vdots \\
			0 & 0 & \cdots & 1 & -1
		\end{array} \right).
	\end{align*}
	This provides an isomorphism $\Z[\zeta] \cong \left\{ q(M_{\zeta}) : q(x) \in \Z[x] \right\}$.
\end{eg}

With this notation, we can show that higher dimensional versions of Theorem \ref{thm: bhk ergodicity}
can be reduced to multi-dimensional combinatorial problems analogous to Behrend's theorem
(Theorem \ref{thm: Behrend}).
\begin{prop} \label{prop: Gaussian ergodicity}
	Let $r, s \in \O_K$ be distinct and nonzero.
	Suppose $r$ and $s$ have matrix representations $M_r = (r_{ij})_{1 \le i,j \le d}$
	and $M_s = (s_{ij})_{1 \le i,j \le d}$.
	For each $1 \le j \le d$, set $c_j := \max_{1 \le i \le d}{\left( |r_{ij}| + |s_{ij}| \right)}$,
	and let $c$ be the geometric mean $c := (c_1c_2 \cdots c_d)^{1/d}$.
	Suppose for some $N \in \N$, there is a set $B \subseteq \{0, 1, \dots, N-1\}^d$
	with $|B| > (2dcN)^{d/2}$ such that $B$ contains no configurations of the form
	\begin{align*}
		P(n,m) := \left\{ n, n + M_rm, n + M_sm \right\}
	\end{align*}
	with $n, m \in \Z^d$, $m \ne 0$.
	Then there is a (non-ergodic) measure-preserving system
	$\left( X, \B, \mu, (T_n)_{n \in \O_K} \right)$ and a set $A \in \B$ with $\mu(A) > 0$ such that
	\begin{align*}
		\mu \left( A \cap T_{rn}A \cap T_{sn}A \right)
		\le \left( \frac{(2dcN)^d}{|B|^2} \right) \mu(A)^3 < \mu(A)^3
	\end{align*}
	for $n \ne 0$.
\end{prop}
\begin{proof}
	Consider the action $T_n : \T^{2d} \to \T^{2d}$ given by
	\begin{align*}
		T_n(x,y) = (x, y + M_nx)
	\end{align*}
	for $n \in \O_K$ and $x, y \in \T^d$.
	This action preserves the Haar measure $\mu$ on $\T^{2d}$.
	
	Now take $N \in \N$ and $B \subseteq \{0, 1, \dots, N-1 \}^d$ with $|B| > (2dcN)^{d/2}$.
	Define
	\begin{align*}
		\tilde{B} = \left\{ \left( \frac{j_1}{N}, \frac{j_2}{N}, \dots, \frac{j_d}{N} \right) : j \in B \right\}
		+ \prod_{i=1}^d{\left[ 0, \frac{1}{dc_iN} \right)},
	\end{align*}
	and let $A = \T^d \times \tilde{B} \subseteq \T^{2d}$.
	Note that $\mu(A) = \frac{|B|}{(dcN)^d}$.
	
	For $n \ne 0$, we have
	\begin{align*}
		\mu \left( A \cap T_{rn}A \cap T_{sn}A \right)
		& = \int_{\T^d \times \T^d}{\ind_{\tilde{B}}(y) \ind_{\tilde{B}}(y+M_{rn}x)
			\ind_{\tilde{B}}(y+M_{sn}x)~dx~dy} \\
		& = \int_{\T^d \times \T^d}{\ind_{\tilde{B}}(y) \ind_{\tilde{B}}(y+M_rx)
			\ind_{\tilde{B}}(y+M_sx)~dx~dy} \\
		& = \mu \left( \left\{ (x,y) : \left\{ y, y+M_rx, y+M_sx \right\} \subseteq \tilde{B} \right\} \right).
	\end{align*}
	It remains to bound the measure of this set.
	Let $(x,y) \in \T^d \times \T^d$, and suppose
	$\left\{ y, y+M_rx, y+M_sx \right\} \subseteq \tilde{B}$.
	From the construction of $\tilde{B}$, we can write
	\begin{align*}
		y & = \left( \frac{j_1}{N} + \alpha_1, \dots, \frac{j_d}{N} + \alpha_d \right), \\
		y+M_rx & = \left( \frac{k_1}{N} + \beta_1, \dots \frac{k_d}{N} + \beta_d \right), \\
		y+M_sx & = \left( \frac{l_1}{N} + \gamma_1,\dots \frac{l_d}{N} + \gamma_d \right),
	\end{align*}
	with $j, k, l \in B$ and
	$\alpha_i, \beta_i, \gamma_i \in  \left[ 0, \frac{1}{dc_iN} \right)$ for $i = 1, \dots, d$.
	Since $M_{s-r}y - M_s\left( y + M_rx \right) + M_r(y+M_sx) = 0$, we have
	\begin{align*}
		M_{s-r} & \left( \frac{j_1}{N}, \dots, \frac{j_d}{N} \right)
		- M_s \left( \frac{k_1}{N}, \dots \frac{k_d}{N} \right)
		+ M_r \left( \frac{l_1}{N}, \dots, \frac{l_d}{N} \right) \\
		& = - M_{s-r} (\alpha_1, \dots, \alpha_d)
		+ M_s (\beta_1, \dots, \beta_d) - M_r(\gamma_1, \dots, \gamma_d) \\
		& = \left( \begin{array}{ccccccccc}
			r_{11}(\alpha_1-\gamma_1) & + & s_{11}(\beta_1-\alpha_1)
			& + & \cdots & + & r_{1d}(\alpha_d-\gamma_d) & + & s_{1d}(\beta_d-\alpha_d) \\
			\vdots & & \vdots & & \vdots & & \vdots & & \vdots \\
			r_{d1}(\alpha_1-\gamma_1) & + & s_{d1}(\beta_1-\alpha_1)
			& + & \cdots & + & r_{dd}(\alpha_d-\gamma_d) & + & s_{dd}(\beta_d-\alpha_d)
		\end{array} \right) \\
		& \in \left( -\frac{1}{N}, \frac{1}{N} \right)^d.
	\end{align*}
	Thus, $M_{s-r} \left( \frac{j_1}{N}, \dots, \frac{j_d}{N} \right)
	- M_s \left( \frac{k_1}{N}, \dots, \frac{k_d}{N} \right)
	+ M_r\left( \frac{l_1}{N}, \dots, \frac{l_d}{N} \right) = 0$.
	
	The configuration $P(n,m)$ satisfies, and is in fact characterized by,
	the equation $M_{s-r}(n) - M_s(n+M_rm) + M_r(n+M_sm) = 0$.
	It follows that $\{j,k,l\} = P(j,m)$ for $m = M_r^{-1}(k-j)$, so $j=k=l$ by the choice of $B$.
	Hence, the points $y$, $y+M_rx$, and $y + M_sx$ all belong to the same region
	$\prod_{i=1}^d{\left[ \frac{j_i}{N}, \frac{j_i}{N} + \frac{1}{dc_iN} \right)}$.
	Therefore,
	$M_rx \in \prod_{i=1}^d{\left( -\frac{1}{dc_iN}, \frac{1}{dc_iN} \right)}$, so
	\begin{align*}
		\mu \left( A \cap T_{rn}A \cap T_{sn}A \right)
		& = \mu \left( \left\{ (x,y) : \left\{ y, y+M_rx, y+M_sx \right\} \subseteq \tilde{B} \right\} \right) \\
		& \le \mu \left( \left\{ (x,y) :
		M_rx \in \prod_{i=1}^d{\left( -\frac{1}{dc_iN}, \frac{1}{dc_iN} \right)},
		y \in \tilde{B} \right\} \right) \\
		& = \left( \frac{2}{dcN} \right)^d \mu(A) \\
		& = \left( \frac{(2dcN)^d}{|B|^2} \right) \mu(A)^3.
	\end{align*}
\end{proof}

\begin{cor} \label{cor: Gaussian Behrend}
	Let $d \in \N$.
	Let $a, b \in \Z$ be distinct and nonzero.
	There is a (non-ergodic) $\Z^d$-system $\left( X, \B, \mu, (T_n)_{n \in \Z^d} \right)$
	such that, for every integer $l \ge 1$, there is a set $A = A(l) \in \B$ with $\mu(A) > 0$ such that
	\begin{align*}
		\mu \left( A \cap T_{an}A \cap T_{bn}A \right) \le \mu(A)^l
	\end{align*}
	for $n \ne 0$.
\end{cor}
\begin{proof}
	Let $B_0 \subseteq \{0, 1, \dots, N-1\}$ such that $B_0$ contains no configurations of the form
	$\{n, n + am, n + bm\}$ with $n,m \in \Z, d \ne 0$.
	Set $B = B_0^d$.
	We claim that $B$ contains no configurations of the form
	\begin{align*}
		P(n,m) = \{n, n + am, n + bm\}
	\end{align*}
	with $n, m \in \Z^d$ and $m \ne 0$.
	Indeed, if $P(n,m) \subseteq B$, then for every $i = 1, \dots, d$,
	we have $\{n_i, n_i + am_i, n_i + bm_i\} \subseteq B_0$.
	Hence $m_i = 0$.
	
	Let $\X = \left( X, \B, \mu, (T_n)_{n \in \Z^d} \right)$ and $A \in \B$ be as
	in the conclusion of Proposition \ref{prop: Gaussian ergodicity}.
	Now, $|B| = |B_0|^d$, and in the notation of Proposition \ref{prop: Gaussian ergodicity},
	$c_j = |a|+|b|$ for every $j = 1, \dots, d$, so
	\begin{align*}
		\mu \left( A \cap T_{an}A \cap T_{bn}A \right)
		\le \left( \frac{(2dcN)^d}{|B|^2} \right) \mu(A)^3
		= \left( \frac{2 \left( |a| + |b| \right) dN}{|B_0|^2} \right)^d \mu(A)^3
	\end{align*}
	for every $n \ne 0$.
	Using $\mu(A) = \left( \frac{|B_0|}{\left( |a| + |b| \right) dN} \right)^d$,
	it remains to check that $B_0$ can be chosen so that
	\begin{align*}
		\left( \frac{2 \left( |a| + |b| \right) dN}{|B_0|^2} \right)^d
		< \left( \frac{|B_0|}{\left( |a| + |b| \right) dN} \right)^{d(l-3)}.
	\end{align*}
	Equivalently, we need
	\begin{align*}
		|B_0| > 2^{1/(l-1)} \left( \left( |a| + |b| \right) d N \right)^{(l-2)/(l-1)} 
		= C_{d,l} N^{1 - \frac{1}{(l-1)}}.
	\end{align*}
	Theorem \ref{thm: Behrend} guarantees such a set for sufficiently large $N$.
\end{proof}

Now we turn to a specific example where the configurations have a simple geometric description.
For the Gaussian integers ($D = -1$), the set of configurations $\{ P(n,m) : n, m \in \Z^2 \}$
is the set of all rotations, translations, and scalings of a fixed triangle $\{0,r,s\} \subseteq \Z^2$.
In particular, for $r=1$ and $s=i$, the forbidden configurations are isosceles right triangles.
Ajtai and Szemer\'{e}di showed a related result:\footnote{The lower bound in Theorem \ref{thm: corner-free} has recently been improved by Linial and Shraibman \cite{ls} and by Green \cite{green-corners}.}

\begin{thm}[\cite{as}, Theorem 1] \label{thm: corner-free}
	There is a constant $c$ such that, for any $N \in \N$,
	there is a set $B \subseteq \{0, 1, \dots, N-1\}^2$ with $|B| > c(r_3(N))^2$
	such that $B$ contains no isosceles right triangles with legs parallel to the axes,
	where $r_3(N)$ denotes the size of the largest $3$-AP-free subset of $\{0, 1, \dots, N-1\}$.
\end{thm}

However, allowing for rotations seems to complicate the picture, and we do not know how to prove the
combinatorial statement in Proposition \ref{prop: Gaussian ergodicity} even for this concrete case.
Since the bounds in Behrend's theorem are much stronger than what is needed for these counterexamples,
we conjecture that an appropriate analogue should hold.
Namely:

\begin{conj}
	Let $K$ be an algebraic number field with ring of integers $\O_K$,
	and let $r, s \in \O_K$ be distinct and nonzero.
	Then for every $C > 0$, there is an $N \in \N$ and a set $B \subseteq \{0, 1, \dots, N-1\}^d$
	such that $|B| > C N^{d/2}$ and $B$ contains no configurations $P(n,m)$ with $m \ne 0$.
\end{conj}

\subsection{Torsion groups}

As was the case with homomorphisms in $\Z^d$,
we will deal with special classes of homomorphisms arising as multiplication in rings.
We will consider the groups $(\Z/p\Z)^{\infty} = \bigoplus_{n=1}^{\infty}{\Z/p\Z}$ with $p \in \Z$ prime,
which we will view as the additive group of the polynomial ring $\F_p[t]$.
With slight modifications to the method employed in the previous subsection,
we get an analogue of Proposition \ref{prop: Gaussian ergodicity} in this setting:

\begin{prop}
	Suppose there is a set $B \subseteq \F_p^N$ with $|B| > p^{N/2+1}$
	such that $B$ contains no nontrivial patterns of the form
	$\{y, y+(x_1, x_2, \dots, x_N), y+(0, x_1, x_2, \dots, x_{N-1})\}$.
	Then there is a measure-preserving system $\left( X, \B, \mu, (T_n)_{n \in \F_p[t]} \right)$
	and a set $A \in \B$ with $\mu(A) > 0$ such that
	\begin{align*}
		\mu \left( A \cap T_nA \cap T_{tn}A \right) \le \left( \frac{p^{N/2+1}}{|B|} \right)^2 \mu(A)^3 < \mu(A)^3
	\end{align*}
	for all $n \ne 0$.
\end{prop}
\begin{proof}
	Let $\F_p((t^{-1}))$ be the field of formal power series $\sum_{n=-\infty}^N{a_nt^n}$
	with $a_n \in \F_p$ and $N \in \Z$.
	Let $\F_p[[t^{-1}]]$ denote the subring consisting of those series with $N \le 0$.
	Consider the quotient space $\mathcal{T} := \F_p((t^{-1}))/\F_p[t]
	\cong t^{-1}\F_p[[t^{-1}]]$,
	and define $T_n : \mathcal{T}^2 \to \mathcal{T}^2$ by $T_n(x,y) = (x,y+nx)$.
	Let $\mu$ be the Haar probability measure on the group $\mathcal{T}^2$.
	
	Let $\tilde{B} = \{y = y_1t^{-1} + y_2t^{-2} + \cdots \in \mathcal{T}
	: (y_N, y_{N-1}, \dots, y_1) \in B\}$,
	let $C := \{x = x_1t^{-1} + x_2t^{-2} + \cdots \in \mathcal{T} : x_N = 0\}$,
	and let $A = C \times \tilde{B}$.
	Note that $\mu(A) = \frac{|B|}{p^{N+1}}$.
	
	Now, for $n \ne 0$,
	\begin{align*}
		\mu \left( A \cap T_nA \cap T_{tn}A \right)
		& = \int_{\mathcal{T}^2}{\ind_{\tilde{B}}(y) \ind_{\tilde{B}}(y+nx) \ind_{\tilde{B}}(y+tnx)
			\ind_C(x)\ind_C(nx)\ind_C(tnx)~dx~dy} \\
		& \le \int_{\mathcal{T}^2}{\ind_{\tilde{B}}(y) \ind_{\tilde{B}}(y+nx) \ind_{\tilde{B}}(y+tnx)
			\ind_C(nx)\ind_C(tnx)~dx~dy} \\
		& = \int_{\mathcal{T}^2}{\ind_{\tilde{B}}(y) \ind_{\tilde{B}}(y+x) \ind_{\tilde{B}}(y+tx)
			\ind_C(x)\ind_C(tx)~dx~dy} \\
		& = \mu\left( \left\{ (x,y) : \{y, y+x, y+tx\} \subseteq \tilde{B}, \{x, tx\} \subseteq C \right\} \right).
	\end{align*}
	
	Suppose $(x,y) \in \mathcal{T}^2$ with $\{y, y+x, y+tx\} \subseteq \tilde{B}$
	and $\{x, tx\} \subseteq C$.
	That is,
	$\{(y_N, \dots, y_1), (y_N+x_N, \dots, y_1+x_1), (y_N+x_{N+1}, \dots, y_1+x_2)\} \subseteq B$
	and $x_N = x_{N+1} = 0$.
	Using that $x_{N+1} = 0$ and the construction of $B$, it follows that $x_1, \dots, x_N = 0$.
	In summary, $x_1, \dots, x_{N+1} = 0$ and $y \in \tilde{B}$.
	Thus,
	\begin{align*}
		\mu \left( A \cap T_nA \cap T_{tn}A \right) \le \frac{1}{p^{N+1}} \cdot \frac{|B|}{p^N}
		= \left( \frac{p^{N/2+1}}{|B|} \right)^2 \mu(A)^3.
	\end{align*}
\end{proof}

Another interesting class of configurations is comprised of those arising from homomorphisms
$\varphi(n) = an$ and $\psi(n) = bn$ with $a,b \in \Z$.
If the characteristic $p$ is large enough, then Behrend's theorem guarantees large sets
$B \subseteq \F_p$ with no 3-term arithmetic progressions.\footnote{We are using the phrase
	``arithmetic progression'' imprecisely here to mean progressions of the form
	$\{n, n + am, n + bm\}$ with fixed $a, b \in \Z$.}
This in turn can be used to produce a system that fails to have large intersections.
\begin{prop} \label{prop: ergodicity large p}
	Let $a, b \in \Z$ be distinct and nonzero.
	For every $L \in \N$, there is a $P = P(L)$ such that for every prime $p \ge P$,
	there is an $\F_p^{\infty}$-system $\left( X, \B, \mu, (T_n)_{n \in \F_p^{\infty}} \right)$
	such that, for every $l \le L$, there is a set $A = A(l) \in \B$ with $\mu(A) > 0$ such that
	\begin{align*}
		\mu \left( A \cap T_{an}A \cap T_{bn}A \right) \le \mu(A)^l
	\end{align*}
	for $n \ne 0$.
	\begin{proof}
		Fix $p$ prime.
		We use the same system as above: $T_n : \mathcal{T}^2 \to \mathcal{T}^2$, $T_n(x,y) = (x,y+nx)$.
		Suppose $B \subseteq \F_p$ avoids patterns of the form $\{n, n+am, n+bm\}$
		and define $\tilde{B} := \left\{ x = \sum_{j=1}^{\infty}{x_jt^{-j}} : x_1 \in B \right\}$.
		Let $A = \mathcal{T} \times \tilde{B}$.
		Note that $\mu(A) = \frac{|B|}{p}$.
		
		Now, for $n \ne 0$,
		\begin{align*}
			\mu \left( A \cap T_{an}A \cap T_{bn}A \right)
			& = \int_{\mathcal{T}^2}{\ind_{\tilde{B}}(y) \ind_{\tilde{B}}(y+anx)
				\ind_{\tilde{B}}(y+bnx)~dx~dy} \\
			& = \int_{\mathcal{T}^2}{\ind_{\tilde{B}}(y) \ind_{\tilde{B}}(y+ax) \ind_{\tilde{B}}(y+bx)~dx~dy} \\
			& = \mu \left( \left\{ (x,y) : \{y, y+ax, y+bx\} \subseteq \tilde{B} \right\} \right).
		\end{align*}
		Suppose $x = \sum_{j=1}^{\infty}{x_jt^{-j}}, y = \sum_{j=1}^{\infty}{y_jt^{-j}} \in \mathcal{T}$.
		Then $\{y, y+ax, y+bx\} \subseteq \tilde{B}$ if and only if
		$\{y_1, y_1+ax_1, y_1+bx_1\} \subseteq B$.
		By the construction of $B$, this happens if and only if $x_1 = 0$ and $y_1 \in B$.
		Thus,
		\begin{align*}
			\mu \left( A \cap T_{an}A \cap T_{bn}A \right) = \frac{|B|}{p^2} = \frac{p^{l-2}}{|B|^{l-1}} \mu(A)^l.
		\end{align*}
		
		So $\mu \left( A \cap T_{an}A \cap T_{bn}A \right) \le \mu(A)^l$ when $|B| \ge p^{1 - \frac{1}{l-1}}$.
		For each $l$, such a $B$ exists so long as $p$ is large enough
		by Theorem \ref{thm: Behrend}.
		The result immediately follows.
	\end{proof}
\end{prop}

\subsection{Rational numbers}

To construct similar counterexamples for $\Q$-systems, we need an analogue of the circle,
which will come from a concrete description of $\hat{\Q}$, the Pontryagin dual group of $(\Q,+)$.
Let $\Prime \subseteq \N$ be the set of prime numbers.
For each $p \in \Prime$, let $|\cdot|_p$ be the $p$-adic absolute value on $\Q$:
\begin{align*}
	\left| p^n \frac{a}{b} \right|_p = p^{-n}
\end{align*}
for $a, b \in \Z$ with $p \nmid a, b$.
The field of \emph{$p$-adic numbers} is the completion $\Q_p$ of $\Q$ in the metric induced by $|\cdot|_p$.
We can write elements of $\Q_p$ as formal series
\begin{align*}
	\Q_p = \left\{ \sum_{i=N}^{\infty}{x_ip^i} : x_i \in \{0, 1, \dots, p-1\}, N \in \Z \right\}.
\end{align*}
The \emph{$p$-adic integers} are the subring $\Z_p \subseteq \Q_p$ defined by
\begin{align*}
	\Z_p = \left\{ x \in \Q_p : |x|_p \le 1 \right\}
	= \left\{ \sum_{i=0}^{\infty}{x_ip^i} : x_i \in \{0, 1, \dots, p-1\} \right\}.
\end{align*}
By expanding integers in base $p$, it is easy to see $\Z \subseteq \Z_p$.
In fact, $\Z_p$ is the closure of $\Z$ in $\Q_p$.

We denote by $\Adele$ the ring of \emph{adeles}
\begin{align*}
	\Adele = \left\{ (a_\infty, a_2, a_3, \dots) \in \R \times \prod_{p \in \Prime}{\Q_p}
	: a_p \in \Z_p~\text{for all but finitely many}~p \in \Prime \right\}.
\end{align*}
Observe that if $x = \frac{a}{b} \in \Q$ and $p \nmid b$, then $x \in \Z_p$.
Since only finitely many primes divide $b$, $\Q$ embeds in $\Adele$ via the map $x \mapsto (x, x, x, \dots)$.
We endow $\Adele$ with a topology generated by basic open sets of the form
\begin{align*}
	U_{\infty} \times \prod_{p \in F}{U_p} \times \prod_{p \in \Prime \setminus F}{\Z_p},
\end{align*}
where $U_{\infty} \subseteq \R$ is open, $F \subseteq \Prime$ is finite, and $U_p \subseteq \Q_p$
is open with respect to $|\cdot|_p$ for every $p \in F$.
It is well-known that $\Q$ is discrete and co-compact in $\Adele$.

Let $\K := \Adele/\Q$.
This compact group will play the role of $\T$ (and $\mathcal{T}$) in the previous examples.
In particular, when $\Q$ is given the discrete topology,
the group $\K$ is isomorphic to the Pontryagin dual of $\Q$.
We can therefore define an action of $\Q$ on $\K^2$ by $T_q(x,y) = (x, y + qx)$,
where $qx$ is the composition of $x$ (viewed as a character on $\Q$) with multiplication by $q$.

In order to handle this $\Q$-action, we give a more concrete description of $\K$.
For each $p \in \Prime$, define the $p$-adic fractional part $f_p : \Q_p \to \Q$ by
$f_p(\sum_{i=N}^{\infty}{x_ip^i}) := \sum_{i=N}^{-1}{x_ip^i}$.
Since the denominator of $f_p(x)$ is a power of $p$, we have $f_p(x) \in \Z_q$ for $q \ne p$.
Define $\tilde{f} : \Adele \to \Q$ by $\tilde{f}(x_{\infty}, x_2, x_3, \dots) := \sum_{p \in \Prime}{f_p(x_p)}$.
This sum has only finitely many nonzero terms, so it is well-defined.
Finally, let $f(x) := \tilde{f}(x) - \floor{x_{\infty} - \tilde{f}(x)}$.
Then $f(x) \in \Q$ and $x - f(x) \in [0,1) \times \prod_{p \in \Prime}{\Z_p}$.
We can view $\K$ as the group $[0,1) \times \prod_{p \in \Prime}{\Z_p}$ with the group operation
\begin{align*}
	x + y := (x + y) - f(x+y).
\end{align*}
Moreover, $\K$ is homeomorphic to the space $\T \times \prod_{p \in \Prime}{\Z_p}$ with the product topology.

We can now prove a version of Theorem \ref{thm: bhk ergodicity} for $\Q$:
\begin{thm} \label{thm: Q ergodicity}
	For every pair of homomorphisms $\varphi, \psi : \Q \to \Q$,
	there is a measure-preserving system $\left( X, \B, \mu, (T_g)_{g \in \Q} \right)$
	such that for all $l \in \N$, there is a set $A = A(l) \in \B$ with $\mu(A) > 0$ such that
	\begin{align*}
		\mu \left( A \cap T_{\varphi(g)}A \cap T_{\psi(g)}A \right) \le \mu(A)^l
	\end{align*}
	for all $g \ne 0$.
\end{thm}
\begin{proof}
	Note that every homomorphism $\varphi : \Q \to \Q$ is determined by the image of $1$.
	That is, if $\varphi(1) = r$, then $\varphi(x) =rx$ for every $x \in \Q$.
	Thus, we may assume that $\varphi$ and $\psi$ are multiplication by $r$ and $s$ respectively.
	
	Now, it suffices to consider $r, s \in \Z$ with $0 < r < s$.
	Indeed, the expression is symmetric in $r$ and $s$, so we may assume that $s > r$.
	By shifting the expression $A \cap T_{rg}A \cap T_{sg}A$ using that $T$ is measure-preserving,
	we may additionally assume $r, s > 0$.
	Finally, given a common denominator $d$ so that $dr, ds \in \Z$,
	we define a $\Q$-action $S_g := T_{dg}$ so that
	$A \cap T_{drg}A \cap T_{dsg}A = A \cap S_{rg}A \cap S_{sg}A$.
	Given an action $T$ that works for $dr$ and $ds$,
	we thus get an action $S$ that works for $r$ and $s$.
	
	Assume $r, s \in \Z$ and $0 < r < s$.
	Let $B \subseteq \{0, 1, \dots, N-1\}$ be such that $B$ contains no pattern of the form
	$\{a, a+rn, a+sn\}$ with $n \in \Z \setminus \{0\}$.
	Consider the action $T_g : \K^2 \to \K^2$ given by the skew-product $T_g(x,y) = (x,y+gx)$.
	Denote by $\mu$ the Haar probability measure on $\K^2$.
	Let $A = \K \times \tilde{B}$, where
	\begin{align*}
		\tilde{B} = \bigcup_{j \in B}{\left[ \frac{j}{(r+s)N}, \frac{j}{(r+s)N} + \frac{1}{(r+s)^2N} \right)}
		\times \prod_{p \in \Prime}{\Z_p}.
	\end{align*}
	For $g \in \Q \setminus \{0\}$,
	\begin{align*}
		\mu \left( A \cap T_{rg}A \cap T_{sg}A \right)
		& = \int_{\K^2}{\ind_{\tilde{B}}(y) \ind_{\tilde{B}}(y+rgx) \ind_{\tilde{B}}(y+sgx)~dx~dy} \\
		& = \int_{\K^2}{\ind_{\tilde{B}}(y) \ind_{\tilde{B}}(y+rx) \ind_{\tilde{B}}(y+sx)~dx~dy}.
	\end{align*}
	
	Suppose $x,y \in \K$ such that $\{y, y+rx, y+sx\} \subseteq \tilde{B}$.
	Then we can write
	\begin{align*}
		y & = \left( \frac{j}{(r+s)N} + \alpha, (y_p)_{p \in \Prime} \right) \\
		y + rx & = \left( \frac{k}{(r+s)N} + \beta, (u_p)_{p \in \Prime} \right) \\
		y + sx & = \left( \frac{l}{(r+s)N} + \gamma, (v_p)_{p \in \Prime} \right)
	\end{align*}
	with $j,k,l \in B$ and $\alpha, \beta, \gamma \in \left[0, \frac{1}{(r+s)^2N}\right)$.
	Since $(s-r)y + r(y+sx) = s(y+rx)$, we have
	\begin{align*}
		\frac{(s-r)j - sk + rl}{(r+s)N} = - (s-r)\alpha + s\beta - r\gamma
		\in \left( -\frac{1}{(r+s)N}, \frac{1}{(r+s)N} \right).
	\end{align*}
	Thus, $(s-r)j - sk + rl = 0$.
	Since $B$ is pattern-free, it follows that $j = k = l$.
	In particular,
	$rx \in \left( -\frac{1}{(r+s)^2N}, \frac{1}{(r+s)^2N} \right) \times \prod_{p \in \Prime}{\Z_p}$, so
	\begin{align*}
		\mu \left( A \cap T_{rg}A \cap T_{sg}A \right) \le \frac{2}{(r+s)^2N} \mu(A).
	\end{align*}
	
	Now by Theorem \ref{thm: Behrend}, we can ensure
	\begin{align*}
		|B| \ge \left( 2(r+s)^{2(l-2)} \right)^{\frac{1}{l-1}} \cdot N^{1 - \frac{1}{l-1}}
	\end{align*}
	so that
	\begin{align*}
		\mu \left( A \cap T_{rg}A \cap T_{sg}A \right) \le \mu(A)^l.
	\end{align*}
\end{proof}


\section{Failure of large intersections for quadruple recurrence}\label{sec: quadruple}

With the exception of Theorem \ref{thm: cubic Khintchine}, we have not considered Khintchine-type results
for patterns of length five and longer.
We suspect that such results are impossible:
\begin{conj} \label{conj: quadruple}
	If $k \ge 4$ and $\varphi_1, \dots, \varphi_k : G \to G$ are distinct, nonzero homomorphisms,\footnote{To avoid trivialities, assume $\{g \in G : \varphi_i(g) = \varphi_j(g)\}$ has infinite index in $G$ for every $i \ne j$.}
	then $\{\varphi_1, \dots, \varphi_k\}$ does not have the large intersections property.
\end{conj}

This was shown in $\Z$ for the specific pattern $(1,2,3,4)$:
\begin{thm}[\cite{bhk}, Theorem 1.3] \label{thm: bhk quadruple}
	There is an ergodic $\Z$-system $(X, \B, \mu, T)$ such that for every $l \ge 1$,
	there is a set $A = A(l) \in \B$ with $\mu(A) > 0$ such that
	\begin{align*}
		\mu \left( A \cap T^nA \cap T^{2n}A \cap T^{3n}A \cap T^{4n}A \right) \le \frac{\mu(A)^l}{2}
	\end{align*}
	for every $n \ne 0$.
\end{thm}

The proof of Theorem \ref{thm: bhk quadruple} in \cite{bhk} relies on a combinatorial fact due to Ruzsa
that plays a similar role to that of Behrend's theorem in constructing non-ergodic counterexamples.
Rather than avoiding $3$-APs, Ruzsa's result concerns certain quadratic configurations.

\begin{defn} \label{defn: QC5}
	Let $P(x)$ be an integer-valued polynomial of degree at most 2.
	The subset \\ $\{P(0), P(1), P(2), P(3), P(4)\} \subseteq \Z$
	is called a \emph{quadratic configuration of 5 terms} (or \emph{QC5} for short).
\end{defn}

\begin{thm}[Ruzsa, see \cite{bhk}, Theorem 2.4] \label{thm: Ruzsa}
	For every $N \in \N$, there is a subset $B \subseteq \{0, 1, \dots, N-1\}$
	such that $|B| > Ne^{-c\sqrt{\log{N}}}$ and $B$ does not contain any QC5.
\end{thm}

By appropriately generalizing Ruzsa's theorem, one can extend Theorem \ref{thm: bhk quadruple}
to general quintuples $(r_0, r_1, r_2, r_3, r_4) \in \Z^5$.
To this end, we introduce a new definition for more general quadratic configurations.
\begin{defn} \label{defn: QC5(c)}
	Let $r = (r_0, r_1, r_2, r_3, r_4) \in \Z^5$.
	If a set $A := \{a_0, a_1, a_2, a_3, a_4\} \subseteq \Z$
	satisfies $a_i = P(r_i)$ for some quadratic polynomial\footnote{
		To be consistent with Definition \ref{defn: QC5}, we interpret ``quadratic'' to mean of degree at most 2
		throughout this section.
	}
	$P(x) \in \Q[x]$, we say that $A$ is \emph{QC5($r$)}.
\end{defn}

\begin{rem}
	In Definition \ref{defn: QC5(c)}, we have made a seemingly weaker assumption on the polynomial $P$.
	Namely, we have only assumed that $P$ has rational coefficients,
	whereas Definition \ref{defn: QC5} assumes that $P$ is integer-valued on $\Z$.
	However, it is easy to check that any rational polynomial of degree $\le d$ taking integer values at
	$d+1$ consecutive integers is automatically an integer-valued polynomial.\footnote{
		This can be done, for example, by the following inductive argument.
		If $P(a), P(a+1), \dots, P(a+d) \in \Z$, then the polynomial $\Delta P(n) := P(n+1)-P(n)$
		has degree $\le d-1$ and takes integer values at $a, a+1, \dots, a+d-1$.
		By the inductive hypothesis, $\Delta P$ is integer-valued,
		and $P(n) = P(0) + \sum_{k=0}^{n-1}{\Delta P(k)}$.
	}
	Hence, QC5 sets are the same as QC5(0,1,2,3,4) sets.
	
	For our dynamical examples, it turns out that rational polynomials are the right objects to consider
	rather than the more restrictive class (when non-consecutive values are allowed)
	of integer-valued polynomials.
\end{rem}

We will prove the following generalization of Ruzsa's theorem in Section \ref{sec: Ruzsa}.
\begin{thm} \label{thm: gen Ruzsa}
	Let $r = (r_0, r_1, r_2, r_3, r_4) \in \Z^5$ be a quintuple of distinct integers.
	There is a constant $c = c(r) > 0$ such that
	for every $N \in \N$, there is a subset $B \subseteq \{0, 1, \dots, N-1\}$
	with $|B| > Ne^{-c\sqrt{\log{N}}}$ such that $B$ does not contain any QC5(r).
\end{thm}

This general version of Ruzsa's theorem verifies Conjecture \ref{conj: quadruple} for $\Z$.
\begin{cor} \label{cor: Z quadruple}
	Let $k \ge 4$.
	Then any $k$-tuple of distinct integers $(r_1, \dots, r_k) \in \Z^k$
	does not have the large intersections property for $\Z$-systems.
\end{cor}

We omit the proof of Corollary \ref{cor: Z quadruple}, since it follows from Theorem \ref{thm: gen Ruzsa}
in the same way that Theorem \ref{thm: bhk quadruple} follows from Theorem \ref{thm: Ruzsa}.
We will, however, present the details of the analogous result for $\Q$ in Section \ref{sec: Q quadruple},
which follows the same general approach as in the integer case.
Corollary \ref{cor: Z quadruple} was obtained independently by similar methods in \cite[Theorem 1.5]{dlms}.

Using Theorem \ref{thm: gen Ruzsa}, analogous counterexamples can be constructed in $\F_p^{\infty}$
when the characteristic $p$ is large.
This approach mirrors the one taken in Proposition \ref{prop: ergodicity large p}.
\begin{thm}
	Suppose $c_1, c_2, c_3, c_4 \in \Z$ are distinct and nonzero.
	For every $L \in \N$, there is a $P = P(c, L)$ such that for every prime $p \ge P$,
	there is an ergodic $\F_p^{\infty}$-system $\left( X, \B, \mu, (T_n)_{n \in \F_p^{\infty}} \right)$ such that,
	for every $l \le L$, there is a set $A = A(l) \in \B$ with $\mu(A) > 0$ such that
	\begin{align}
		\mu \left( A \cap T_{c_1n}A \cap T_{c_2n}A \cap T_{c_3n}A \cap T_{c_4n}A \right) \le \mu(A)^l.
	\end{align}
\end{thm}

As a consequence, if $c_1, \dots, c_k$ are distinct, nonzero integers with $k \ge 4$,
then $(c_1, \dots, c_k) \in \Z^k$ does not have the large intersections property for $\F_p^{\infty}$-systems
so long as $p$ is large enough (depending on $c$).
This complements an earlier result of \cite{btz2} from a different angle:\footnote{Our terminology differs slightly from \cite{btz2}. There, the authors say a tuple $(c_0, c_1, \dots, c_k)$ has the \emph{Khintchine property} if $\{ g \in \F_p^{\infty} : \mu(T_{c_0g}^{-1}A \cap T_{c_1g}^{-1}A \cap \cdots \cap T_{c_kg}^{-1}A) > \mu(A)^{k+1} - \eps \}$ is syndetic for every ergodic $T$, every measurable set $A$, and every $\eps > 0$. This is the same as saying that $(c_1-c_0, \dots, c_k-c_0)$ has the large intersections property.}

\begin{thm}[\cite{btz2}, Theorem 1.15] \label{thm: btz quadruple}
	Let $k \ge 3$.
	There is a constant $C_k$ depending only on $k$ such that for every prime $p$,
	there are at most $C_k p^{k-1}$ tuples $(c_1, \dots, c_k) \in \F_p^k$ with the large intersections property
	for $\F_p^{\infty}$-systems.
\end{thm}

\begin{rem}
	Notice that the assumption in Theorem \ref{thm: btz quadruple} is $k \ge 3$ rather than $k \ge 4$.
	This is assumed because the triple $(c_1, c_2, c_3)$ will fail to have the large intersections property in high
	characteristic unless $(0, c_1, c_2, c_3)$ forms a parallelogram configuration.
	We will revisit this in Section \ref{sec: parallelogram}.
\end{rem}

\subsection{A generalization of Ruzsa's theorem} \label{sec: Ruzsa}

We now set out to prove Theorem \ref{thm: gen Ruzsa}.
The strategy is parallel to Ruzsa's approach to Theorem \ref{thm: Ruzsa}.
First, we show that quadratic configurations in $\R^d$ with at least five points cannot be contained
on the surface of a sphere (Lemma \ref{lem: Ruzsa}).
This fact allows us to mirror Behrend's construction for 3-AP-free sets:
we consider integer expansions in a conveniently chosen base,
and by putting restrictions on the digits of these expansions, we can avoid all QC5(r) patterns
while having relatively large density in the interval $\{0, 1, \dots, N-1\}$.

Let us turn to the details.
\begin{lemma} \label{lem: Ruzsa}
	Let $P : \R \to \R^d$ be a polynomial\footnote{
		By this, we mean that $P = (P_1, \dots, P_d)$ with $P_i(x) \in \R[x]$.
		The degree of $P$ is the maximum of the degrees of the $P_i$'s.
	} of degree at most 2.
	Suppose $r_0, r_1, r_2, r_3, r_4 \in \Z$ are distinct, and let $a_i = P(r_i)$.
	If $\|a_i\| = c$ for every $i=0, \dots, 4$, where $\|\cdot\|$ is the Euclidean norm on $\R^d$,
	then $P$ is constant.
	Hence, the values $a_i$ are all the same.
\end{lemma}
\begin{proof}
	Consider the function $f : \R \to \R$ defined by $f(x) := \|P(x)\|^2 - c^2$.
	This can be written as $f(x) = P_1(x)^2 + \cdots + P_d(x)^2 - c^2$,
	so $f$ is a polynomial of degree at most 4.
	By assumption, $\|a_i\| = c$ for each $i=0, \dots, 4$, so $f(r_i) = 0$.
	Thus, $f$ has at least 5 distinct roots, which implies $f \equiv 0$.
	In particular, $P$ is a bounded function.
	Since $P$ is polynomial, it follows that $P$ is constant.
\end{proof}

We now prove Theorem \ref{thm: gen Ruzsa}.
\begin{proof}[Proof of Theorem \ref{thm: gen Ruzsa}]
	We assume without loss of generality that $0 = r_0 < r_1 < r_2 < r_3 < r_4$.
	Let $b, d, m \in \N$, and let $k \in \N$ with $k \le (b-1)^2 d$.
	Consider the set
	\begin{align*}
		B := \left\{ x = x_0 + x_1(mb-1) + \cdots + x_{d-1}(mb-1)^{d-1}
		: 0 \le x_j \le b-1, \sum_{j=0}^{d-1}{x_j^2} = k \right\}.
	\end{align*}
	We will specify $b$, $d$, $m$, and $k$ in what follows so that $B$ has the desired properties.
	
	First, we claim that $B$ is QC5($r$)-free for appropriately chosen $m$.
	Indeed, suppose $P$ is a quadratic polynomial such that $a_i := P(r_i) \in B$ for each $i = 0, \dots, 4$.
	Since a quadratic polynomial is determined by its values at any 3 points, we have the relations
	\begin{align} \label{eq: polynomial}
		(r_1r_2^2-r_1^2r_2) a_i = &
		\left( (r_2-r_1)r_i^2 - (r_2^2-r_1^2)r_i + (r_1r_2^2 - r_1^2r_2) \right) a_0 \nonumber \\
		& - (r_2r_i^2 - r_2^2r_i) a_1 + (r_1r_i^2 - r_1^2r_i) a_2
	\end{align}
	for $i = 3, 4$.
	
	We will now specify $m$ in terms of the coefficients.
	Define
	\begin{align*}
		s & := r_1r_2^2-r_1^2r_2, \\
		U(r) & := (r_2-r_1)r^2 - (r_2^2-r_1^2)r + (r_1r_2^2 - r_1^2r_2), \\
		V(r) & := r_2 r^2 - r_2^2 r, \\
		W(r) & := r_1 r^2 - r_1^2 r.
	\end{align*}
	Set
	\begin{align*}
		m := \max\{U(r_4) + W(r_4), V(r_4) + s\}.
	\end{align*}
	Note that $s \ge 1$, and $U(r)$, $V(r)$, and $W(r)$ are all positive and increasing on $(r_2, \infty)$,
	so replacing $r_4$ by $r_3$ in this expression produces a smaller constant.
	
	Expand $a_i$ in base $mb-1$ as $a_i = a_{i,0} + a_{i,1}(mb-1) + \cdots + a_{i,d-1}(mb-1)^{d-1}$.
	Then $0 \le a_{i,j} \le b-1$ and $\sum_{j=0}^{d-1}{a_{i,j}^2} = k$.
	Using the relations \eqref{eq: polynomial}, we see that the polynomial
	\begin{align*}
		f_i(x) := \sum_{j=0}^{d-1}{\left( U(r_i)a_{0,j} - V(r_i)a_{1,j} + W(r_i)a_{2,j} - s a_{i,j} \right) x^j}
	\end{align*}
	has a root at $x = mb-1$ for $i=3, 4$.
	By the choice of $m$, the coefficients of $f_i$ are integers in the interval $\left( -m(b-1), m(b-1) \right)$.
	Hence, by the rational root theorem, $f_i$ must be identically zero.
	That is, $sa_{i,j} = U(r_i)a_{0,j} - V(r_i)a_{1,j} + W(r_i)a_{2,j}$ for $i = 3,4$ and $j = 0, \dots, d-1$.
	
	There are therefore quadratic polynomials $Q_j$ such that $Q_j(r_i) = a_{i,j}$.
	Namely,
	\begin{align*}
		Q_j(r) = \frac{1}{s} \left( U(r)a_{0,j} - V(r)a_{1,j} + W(r)a_{2,j} \right).
	\end{align*}
	Let $Q = (Q_0, Q_1, \dots, Q_{d-1})$.
	Then
	\begin{align*}
		\|Q(r_i)\| = \left( \sum_{j=0}^{d-1}{a_{i,j}^2} \right)^{1/2} = \sqrt{k}
	\end{align*}
	for every $i=0, \dots, 4$,
	so by Lemma \ref{lem: Ruzsa}, all of the $a_i$'s are the same.
	Thus, $B$ contains no QC5($r$).
	
	It remains to choose $b$, $d$, and $k$ to ensure that $B$ is a sufficiently dense set.
	This argument will make the same kinds of estimates that appear in \cite{behrend}.
	There are $b^d$ values of $x$ with $0 \le x_j \le b-1$,
	and $(b-1)^2d+1$ possible values of $k$, so for some $k$, $B$ contains at least
	\begin{align*}
		\frac{b^d}{(b-1)^2d+1} > \frac{b^{d-2}}{d}
	\end{align*}
	elements.
	
	Given $N \in \N$, let $d = \floor{\sqrt{\frac{\log{N}}{\log{m}}}}$, and let $b \in \N$ so that
	\begin{align*}
		(mb-1)^d \le N < (mb-1)^{d+1}.
	\end{align*}
	Then $B \subseteq \{0, 1, \dots, N-1\}$ and, noting $b > \frac{N^{1/(d+1)}}{m}$,
	\begin{align*}
		|B| > \frac{b^{d-2}}{d} > \frac{N^{\frac{d-2}{d+1}}}{dm^{d-2}}
		= N \left( \frac{1}{N^{\frac{3}{d+1}} d m^{d-2}} \right).
	\end{align*}
	Now, since $d \le \sqrt{\frac{\log{N}}{\log{m}}} < d + 1$,
	\begin{align*}
		\log{\left( N^{\frac{3}{d+1}} d m^{d-2} \right)}
		& = \frac{3}{d+1} \log{N} + \log{d} + (d-2)\log{m} \\
		& < 3\sqrt{\log{m}} \sqrt{\log{N}}
		+ \frac{1}{2} \left( \log{\log{N}} - \log{\log{m}} \right)
		+ \sqrt{\log{m}} \sqrt{\log{N}}.
	\end{align*}
	Thus, $|B| > N e^{-c\sqrt{\log{N}}}$ for
	$c = 4 \sqrt{\log{m}} + \max_{N \in \N}{\left( \frac{\log{\log{N}}}{2\sqrt{\log{N}}} \right)}$.
\end{proof}

Though we will not need it for dynamical applications, it is not hard to extend this argument
to construct large sets that avoid degree $d$ polynomial configurations on $2d+1$ points.

\subsection{Rings of integers}

Just as higher-dimensional versions of Behrend's theorem would allow us to construct
non-ergodic systems that fail to have large intersections for double recurrence in rings of integers,
we can reduce analogues of Theorem \ref{thm: bhk quadruple}
to combinatorial results in the vein of Ruzsa's theorem.

Let $K$ be an algebraic number field of degree $d = [K : \Q]$.
Let $\O_K$ be the ring of integers of $K$, and let $b_1, \dots, b_d$ be an integral basis for $\O_K$.
Fix $r = (r_1, r_2, r_3, r_4) \in \O_K^4$.
For notational convenience, set $r_0 = 0$.
Call a set $\{a_0, a_1, a_2, a_3, a_4\} \subseteq \Z^d$ a \emph{QC5($r$) set}
if there is a quadratic polynomial $P(x) \in K[x]$ such that
$P(r_i) = \sum_{j=1}^d{a_{i,j}b_j}$ for $i = 0, \dots, 4$,
where $a_i = (a_{i,1}, \dots, a_{i,d}) \in \Z^d$.

\begin{prop}
	Let $r_1, r_2, r_3, r_4 \in \O_K$ be distinct and nonzero.
	Let $r = (0,r_1,r_2,r_3,r_4)$.
	If for every $C > 0$, there is an $N \in \N$ and a set $B \subseteq \{0, 1, \dots, N-1\}^d$
	such that $|B| > CN^{d/2}$ and $B$ is QC5($r$)-free,
	then $(r_1, r_2, r_3, r_4)$ does not have the large intersections property.
	\begin{proof}
		Fix $C > 0$, and let $B \subseteq \{0, 1, \dots, N-1\}^d$ be QC5($r$)-free with $|B| > CN^{d/2}$.
		Fix $\alpha \in \T^d$ such that $(M_n\alpha)_{n \in \O_K}$ is dense in $\T^d$.
		We will consider the skew-product system $T_n : \T^{2d} \to \T^{2d}$ given by
		\begin{align*}
			T_n(x,y) = \left( x + M_n\alpha, y + M_{2n}x + M_{n^2}\alpha \right),
		\end{align*}
		which preserves the Haar measure $\mu$ on $\T^{2d}$.
		This is an action of $(\O_K, +)$ because $n \mapsto M_n$ is a homomorphism from $(\O_K, +)$
		to the group of $d \times d$ integer matrices under addition.
		Moreover, $(T_n)_{n \in \O_K}$ is ergodic because $(M_n\alpha)_{n \in \O_K}$ is dense.
		
		Fix $m \in \N$ to be determined later.
		For $j \in \{0, 1, \dots, N-1\}^d$, let
		\begin{align*}
			C_j := \prod_{i=1}^d{\left[ \frac{j_i}{mN}, \frac{j_i}{mN} + \frac{1}{m^2N} \right)},
		\end{align*}
		and define $\tilde{B} := \bigcup_{j \in B}{C_j}$.
		Let $A = \T^d \times \tilde{B} \subseteq \T^{2d}$.
		Then $\mu(A) = \frac{|B|}{(m^2N)^d}$.
		
		For $n \in \O_K$,
		\begin{align*}
			\mu \left( A \cap T_{r_1n}A \cap T_{r_2n}A \cap T_{r_3n}A \cap T_{r_4n}A \right)
			& = \iint_{\T^d \times \T^d}{\ind_{\tilde{B}}(y)
				\prod_{i=1}^4{\ind_{\tilde{B}}(y+M_{2r_in}x+M_{r_i^2n^2}\alpha)}~dx~dy}.
		\end{align*}
		Suppose $n \ne 0$, and this integrand is nonzero.
		That is, $y \in \tilde{B}$ and $y + M_{2r_in}x + M_{r_i^2n^2}\alpha \in \tilde{B}$ for $i = 1, \dots, 4$.
		
		Interpreting $x$, $y$, and $\alpha$ as elements in $\R^d$ with coordinates in $[0,1)$,
		the map $f(r) := y + M_{2rn}x + M_{r^2n^2}\alpha$ is a polynomial of degree $2$
		from $\O_K$ to $\R^d$.
		It is therefore fully determined by its values at $0$, $r_1$, and $r_2$.
		To be explicit, for any $r \in \O_K$,
		\begin{align*}
			M_{r_1r_2^2 - r_1^2r_2} f(r) & = M_{(r_2-r_1)r^2 - (r_2^2-r_1^2)r + (r_1r_2^2 - r_1^2r_2)}f(0) \\
			& - M_{r_2r^2 - r_2^2r} f(r_1) + M_{r_1r^2 - r_1^2r} f(r_2).
		\end{align*}
		Letting $a_0 = f(0) = y$ and $a_i = f(r_i)$ for $i = 1, \dots, 4$, we have therefore found
		integer matrices $S$, $U_i$, $V_i$, and $W_i$ for $i = 3, 4$ so that
		\begin{align} \label{eq: polynomial relation}
			S a_i = U_i a_0 - V_i a_1 + W_i a_2.
		\end{align}
		
		By assumption, $a_i \in \tilde{B}$, say $a_i \in B_{j_i}$, for every $i = 0, \dots, 4$.
		Therefore, we can write $a_i = \frac{j_i}{mN} + \beta_i$ with $j_i \in B$
		and $\beta_i \in \left[ 0, \frac{1}{m^2N} \right)^d$.
		By \eqref{eq: polynomial relation}, we have
		\begin{align*}
			\frac{Sj_i - U_i j_0 + V_i j_1 - W_i j_2}{mN}
			= -S \beta_i + U_i \beta_0 - V_i \beta_1 + W_i \beta_2.
		\end{align*}
		for $i = 3, 4$.
		We can choose $m \ge \max\left\{ \left| s_{kl} \right|+\left| u^{(i)}_{kl} \right|
		+\left| v^{(i)}_{kl} \right|+\left| w^{(i)}_{kl} \right| : 1 \le k,l \le d, i = 3,4 \right\}$
		so that the quantity on the right-hand side is in the set $\left( -\frac{1}{mN}, \frac{1}{mN} \right)^d$.
		It then follows that the left-hand side is equal to 0.
		That is,
		\begin{align*}
			S j_i = U_i j_0 - V_i j_1 + W_i j_2.
		\end{align*}
		
		Now, $S$ does not depend on $i$, and the matrices $U_i$, $V_i$, and $W_i$ depend quadratically on $r_i$,
		so there is a polynomial $P$ of degree at most two for which $j_i = P(r_i)$.
		But $j_i \in B$ for every $i = 0, \dots, 4$ and $B$ is QC5($r$)-free by assumption.
		Thus, there is some $j \in B$ so that $j_i = j$ for all $i = 0, \dots, 4$.
		That is, $\{a_0, a_1, a_2, a_3, a_4\} \subseteq C_j$.
		
		Now,
		\begin{align*}
			M_{2(r_1r_2^2 - r_1^2r_2)n}x & = - M_{r_2^2-r_1^2} a_0 + M_{r_2^2} a_1 - M_{r_1^2} a_2 \\
			& = M_{r_1^2}(a_0 - a_2) + M_{r_2^2}(a_1-a_0).
		\end{align*}
		We know by the above that $a_i - a_j \in \left( -\frac{1}{m^2N}, \frac{1}{m^2N} \right)^d$.
		Thus, setting
		\begin{align*}
			L := 2 \max_{1 \le k,l \le d}{\left( \left| \left( M_{r_1^2} \right)_{kl} \right|
				+ \left| \left( M_{r_2^2} \right)_{kl} \right|\right)},
		\end{align*}
		we have $M_{2(r_1r_2^2 - r_1^2r_2)n}x \in \left( -\frac{L}{2m^2N}, \frac{L}{2m^2N} \right)^d$.
		
		By assumption, $y \in \tilde{B}$, so
		\begin{align*}
			\mu \left( A \cap T_{r_1n}A \cap T_{r_2n}A \cap T_{r_3n}A \cap T_{r_4n}A \right)
			\le \left( \frac{L}{m^2N} \right)^d \mu(A)
			= \left( \frac{m^{2d} L^d N^d}{|B|^2} \right) \mu(A)^3.
		\end{align*}
		Since $L$ and $m$ depend only on $r_1$ and $r_2$ (and not on $B$ or $N$),
		this shows that $(r_1, r_2, r_3, r_4)$ will fail to have the large intersections property
		so long as $C > m^d L^{d/2}$.
	\end{proof}
\end{prop}

\begin{cor} \label{cor: Gaussian Ruzsa}
	Let $d \in \N$.
	Let $c_1, \dots, c_4 \in \Z$ be distinct and nonzero.
	There is an ergodic $\Z^d$-system $\left( X, \B, \mu, (T_n)_{n \in \Z^d} \right)$
	such that, for every integer $l \ge 1$, there is a set $A = A(l) \in \B$ with $\mu(A) > 0$ such that
	\begin{align*}
		\mu \left( A \cap T_{c_1n}A \cap T_{c_2n}A \cap T_{c_3n}A \cap T_{c_4n}A \right) \le \mu(A)^l
	\end{align*}
	for $n \ne 0$.
\end{cor}

The proof of Corollary \ref{cor: Gaussian Ruzsa} is completely parallel
to the proof of Corollary \ref{cor: Gaussian Behrend},
using Theorem \ref{thm: gen Ruzsa} in place of Behrend's theorem (Theorem \ref{thm: Behrend}) to satisfy the necessary bounds.
We therefore omit the details.

\subsection{Rational numbers} \label{sec: Q quadruple}

Using our generalization of Ruzsa's theorem (Theorem \ref{thm: gen Ruzsa}), we can prove an analogue of Theorem \ref{thm: bhk quadruple} in $\Q$.
By restricting our attention to the real coordinate of adeles, we can adapt the proof in \cite[Section 2.2]{bhk}.
\begin{thm} \label{thm: Q quadruple}
	There is an ergodic $\Q$-system $\left( X, \B, \mu, (T_n)_{n \in \Q} \right)$ such that
	for every quintuple of distinct rationals $r = (r_0, r_1, r_2, r_3, r_4) \in \Q^5$ and every $l \ge 1$,
	there is a set $A = A(r,l) \in \B$ with $\mu(A) > 0$ such that
	\begin{align*}
		\mu \left( T_{r_0n}A \cap T_{r_1n}A \cap T_{r_2n}A \cap T_{r_3n}A \cap T_{r_4n}A \right) \le \mu(A)^l
	\end{align*}
	for every $n \ne 0$.
\end{thm}

Before proving the theorem, we note that it immediately implies that
Conjecture \ref{conj: quadruple} holds in $\Q$:
\begin{cor}
	If $k \ge 4$, then every $k$-tuple of distinct nonzero rationals
	$(r_1, \dots, r_k) \in \Q^k$ fails to have the large intersections property for $\Q$-systems.
\end{cor}

\begin{proof}[Proof of Theorem \ref{thm: Q quadruple}]
	Consider the $\Q$-action on $\K^2$ given by $T_n(x,y) = (x+n\alpha, y+2nx+n^2\alpha)$,
	which preserves the Haar probability measure $\mu$ on $\K^2$.
	We can choose $\alpha \in \K \setminus \Q$ so that $(n\alpha)_{n \in \Q}$ is dense
	and the $\Q$-action is consequently ergodic.
	
	Without loss of generality, we may assume $r \in \Z^5$ and $0 = r_0 < r_1 < r_2 < r_3 < r_4$.
	Let $B \subseteq \{0, 1, \dots, N-1\}$ be QC5($r$)-free.
	Define polynomials
	\begin{align*}
		s & := r_1r_2^2-r_1^2r_2, \\
		U(r) & := (r_2-r_1)r^2 - (r_2^2-r_1^2)r + (r_1r_2^2 - r_1^2r_2), \\
		V(r) & := r_2 r^2 - r_2^2 r, \\
		W(r) & := r_1 r^2 - r_1^2 r,
	\end{align*}
	and set
	\begin{align*}
		m := \max\{ U(r_4) + W(r_4), V(r_4) + s\}
	\end{align*}
	as in the proof of Theorem \ref{thm: gen Ruzsa}.
	For $j \in \{0, 1, \dots, N-1\}$, let
	\begin{align*}
		I_j := \left[ \frac{j}{mN}, \frac{j}{mN} + \frac{1}{m^2N} \right) \times \prod_{p \in \Prime}{\Z_p},
	\end{align*}
	and set $\tilde{B} := \bigcup_{j \in B}{I_j}$.
	Finally, let $A := \K \times \tilde{B}$.
	Note $\mu(A) = \frac{|B|}{m^2N}$.
	
	Suppose $(x,y) \in A \cap T_{r_1}nA \cap T_{r_2n}A \cap T_{r_3n}A \cap T_{r_4n}A$ for some $n \ne 0$.
	Then $\{y, y+2r_1nx+r_1^2n^2\alpha, y+2r_2nx+r_2^2n^2\alpha, y+2r_3nx+r_3^2n^2\alpha,
	y+2r_4nx+r_4^2n^2\alpha \} \subseteq \tilde{B}$.
	For $i = 0, \dots, 4$, let $a_i = y+r_i(2nx)+r_i^2(n^2\alpha)$.
	We have $\{a_0, a_1, a_2, a_3, a_4\} \subseteq \tilde{B}$.
	
	Consider $a_i$ as an element of $[0,1) \times \prod_{p\in\Prime}{\Z_p} \subseteq \Adele$.
	By the definition of $\tilde{B}$, the real coordinate of $a_i$ is in fact in the interval
	$\left[ 0, \frac{1}{m} \right)$.
	
	Observe
	\begin{align*}
		s a_i =  U(r_i) a_0 - V(r_i) a_1 + W(r_i) a_2 \pmod{\Q}
	\end{align*}
	for $i = 3, 4$.
	Now, by the choice of $m$, the adele on the right-hand side has real coordinate in the interval
	$\left( -\frac{m-1}{m}, 1 \right)$, since the real coordinate of each $a_i$ is in $\left[ 0, \frac{1}{m} \right)$.
	It follows that the real coordinate of the left-hand side is equal to the real coordinate of the right-hand side
	considered as elements of $\Adele$.
	
	Let $j_i \in B$ such that $a_i \in I_{j_i}$.
	Then $U(r_i) a_0 - V(r_i) a_1 + W(r_i) a_2$ belongs to the set
	\begin{align*}
		J := \left( \frac{U(r_i) j_0 - V(r_i) j_1 + W(r_i) j_2}{mN} - \frac{m-1}{m^2N},
		\frac{U(r_i) j_0 - V(r_i) j_1 + W(r_i) j_2}{mN} + \frac{1}{mN} \right)
		\times \prod_{p \in \Prime}{Z_p}.
	\end{align*}
	If $j \ne U(r_i) j_0 - V(r_i) j_1 + W(r_i) j_2$, then $J \cap I_j = \es$.
	Thus, $s j_i = U(r_i) j_0 - V(r_i) j_1 + W(r_i) j_2$ for $i=3,4$.
	
	It follows that there is a quadratic polynomial $P$ such that $P(r_i) = j_i$ for $i = 0, \dots, 4$.
	But $B$ is QC5($r$)-free, so $P$ must be constant.
	That is, $a_i \in I_j$ for some fixed $j \in B$.
	Now,
	\begin{align*}
		2sn x & = - (r_2^2-r_1^2) a_0 + r_2^2 a_1 - r_1^2 a_2.
	\end{align*}
	so $2snx \in \left( - \frac{L}{2m^2N}, \frac{L}{2m^2N} \right)
	\times \prod_{p \in \Prime}{\Z_p}$, where $L = 2r_2^2$.
	
	The transformation $x \mapsto 2snx$ is measure-preserving, so we have
	\begin{align*}
		\mu \left( \bigcap_{i=0}^4{T_{r_in}A} \right)
		\le \frac{L |B|}{m^4N^2}.
	\end{align*}
	To get the bound of $\mu(A)^l$, it suffices to have
	\begin{align*}
		|B| \ge L^{\frac{1}{l-1}} m^{2 \left( 1 - \frac{1}{l-1} \right)} N^{1 - \frac{1}{l-1}}.
	\end{align*}
	Theorem \ref{thm: gen Ruzsa} guarantees that such a set $B$ exists
	for large enough $N$ depending on $r$ and $l$.
\end{proof}


\section{Failure of large intersections for triple recurrence without the parallelogram condition}
\label{sec: parallelogram}

We say that a quadruple $(\varphi_0,\varphi_1,\varphi_2,\varphi_3)$ \emph{forms a parallelogram} if
$\varphi_i+\varphi_j=\varphi_k+\varphi_l$
for some permutation $(i,j,k,l)$ of $(0,1,2,3)$.
This property characterizes quadruples $(0, \varphi, \psi, \varphi + \psi)$ up to reordering and shifts.
Theorem \ref{thm: triple Khintchine} shows that, for admissible families of homomorphisms coming as multiplication by integers, the parallelogram condition is sufficient for large intersections.
We believe it is also necessary.

\begin{conj} \label{conj: parallelogram}
	Let $G$ be a countable discrete abelian group.
	An admissible triple $(r, s, t)$
	has the large intersections property if and only if $(0, r, s, t)$ forms a parallelogram.
\end{conj}

It is stated in a footnote in \cite{btz2} that if $(0,c_1,c_2,c_3) \in \Z^4$
does not form a parallelogram, then $(c_1,c_2,c_3)$ fails to have the large intersections property for $\Z$-systems,
establishing Conjecture \ref{conj: parallelogram} for $\Z$.
Moreover, if $p$ is sufficiently large (depending on $(c_1,c_2,c_3)$),
then the triple fails to have the large intersections property for $\F_p^{\infty}$-systems as well.

\bigskip

Some progress has also been made on a related combinatorial problem.
As noted in the introduction, for $r, s \in \Z$ distinct and nonzero and $\delta, \eps > 0$,
there exists $N_0 = N_0(r, s, \delta, \eps) \in \N$ such that if $N \ge N_0(\delta, \eps)$ and $A \subseteq \{1, \dots, N\}$ has size $|A| \ge \delta N$, then there exists $n \ne 0$ such that
\begin{align*}
	\left| A \cap (A-rn) \cap (A-sn) \cap (A-(r+s)n) \right| & > \left( \delta^4 - \eps \right) N.
\end{align*}

Recent work of Sah, Sawhney, and Zhao shows that the parallelogram property is necessary for this combinatorial result.
\begin{thm}[\cite{ssz}, Theorem 1.6] \label{thm: ssz}
	There is an absolute constant $\delta > 0$ such that, for every $(c_0, c_1, c_2, c_3) \in \Z^4$
	that does not form a parallelogram, the following holds:
	for all $\alpha \in \left( 0, \frac{1}{2} \right)$ and all $N_0 \in \N$,
	there is an $N \ge N_0$ and a set $B \subseteq \{0, 1, \dots, N-1\}$
	with $|B| > \alpha N$ such that for every $n \in \Z$,
	\begin{align*}
		\left| (B-c_0n) \cap (B-c_1n) \cap (B-c_2n) \cap (B-c_3n) \right| \le (1-\delta)\alpha^4 N.
	\end{align*}
\end{thm}

For general admissible families of homomorphisms $\{\varphi, \psi, \varphi + \psi\}$, it is unclear from our methods whether or not the large intersections property holds.
Indeed, the limit formula in Theorem \ref{thm: limit formula} only applies to homomorphisms given by multiplication by integers.

\bigskip

We now discuss a generalization of parallelogram configurations where we can prove a (much weaker) result about multiple recurrence.
One advantage of parallelogram configurations is that we can view them iteratively.
That is,
\begin{align*}
	\mu \left( A \cap T_{\varphi(g)}^{-1}A \cap T_{\psi(g)}^{-1}A \cap T_{(\varphi+\psi)(g)}^{-1}A \right)
	= \mu \left( \left( A \cap T_{\varphi(g)}^{-1}A \right)
	\cap T_{\psi(g)}^{-1} \left( A \cap T_{\varphi(g)}^{-1}A \right) \right).
\end{align*}
In this form, we see, through two successive applications of Khintchine's theorem, that there are $g, h \in G$
such that
\begin{align*}
	\mu \left( \left( A \cap T_{\varphi(g)}^{-1}A \right)
	\cap T_{\psi(h)}^{-1} \left( A \cap T_{\varphi(g)}^{-1}A \right) \right) > \mu(A)^4 - \eps.
\end{align*}
The difficulty is in getting pairs $(g,h)$ to lie in the diagonal of $G^2$, i.e. $g = h$.
Beyond this, there are no obvious obstacles to repeating an iterative process.
Since large intersections begin to fail for longer expressions, this idea must be well short of verifying that parallelogram families have the large intersections property.
On the other hand, we can produce meaningful results about large intersections, albeit of a different variety, for general ``cubic'' configurations:

\begin{thm} \label{thm: cubic Khintchine}
	Let $\X = \left( X, \B, \mu, (T_g)_{g \in G} \right)$ be a measure-preserving system.\footnote{
		Note that we do not assume ergodicity here, in contrast to the results of Section \ref{sec: ergodicity}.
		Moreover, as noted in the discussion above, we obtain here ``cubic'' expressions of arbitrary length.
		This is an indication that these methods are unlikely to prove that
		$R_{\eps}(\varphi, \psi)$ is syndetic for any fixed pair of homomorphisms $\{\varphi,\psi\}$.}
	Let $A \in \B$, $k \in \N$, and $\eps > 0$.
	There is a syndetic set $S \subseteq \mathrm{Hom}(G,G)^k$ such that
	for every $(\varphi_1, \varphi_2, \dots, \varphi_k) \in S$, the set
	\begin{align*}
		R_{\eps}(\varphi_1, \dots, \varphi_k) := \left\{ g \in G :
		\mu \left( \bigcap_{\eps \in \{0,1\}^k}{
			T_{(\eps_1\varphi_1 + \eps_2\varphi_2 + \cdots + \eps_k\varphi_k)(g)}^{-1}A} \right)
		> \mu(A)^{2^k} - \eps \right\}
	\end{align*}
	is syndetic in $G$.
\end{thm}

Before proving Theorem \ref{thm: cubic Khintchine}, we need to unpack some definitions.
The collection of homomorphisms from $G$ to itself, $\Hom(G,G)$,
is an abelian group under pointwise addition $(\varphi + \psi)(g) := \varphi(g) + \psi(g)$.
Even for countable $G$, the group $\Hom(G,G)$ may be uncountable.
For example, $\Hom(\F_p^{\infty}, \F_p^{\infty})$ contains the group $\Hom(\F_p^{\infty}, \F_p) \cong \hat{\F_p^{\infty}} \cong \prod_{n=1}^{\infty}{\F_p}$.
To discuss uniform Ces\`{a}ro limits and syndetic sets in uncountable groups, we need to introduce a toplogy.

The natural topology to put on $\Hom(G,G)$ is the topology of pointwise convergence.\footnote{When dealing with groups $G$ that are not discrete, the appropriate topology is the compact-open topology.}
That is, $\varphi_n \to \varphi$ if and only if for every $g \in G$, $\varphi_n(g) = \varphi(g)$ for all large enough $n$.
Note that this topology has a basis of clopen sets of the form $U_{g_1, \dots, g_n}(\varphi_0) := \left\{ \varphi \in \Hom(G,G) : \varphi(g_i) = \varphi_0(g_i)~\text{for}~i = 1, \dots, n \right\}$.
It is an easy exercise to check that these sets are also compact, so $\Hom(G,G)$ is a locally compact abelian group.

In this setting, we say $S \subseteq \Hom(G,G)$ is \emph{syndetic} if there is a compact set $K \subseteq \Hom(G,G)$ such that $S + K = \Hom(G,G)$.
A sequence of compact subsets $(F_N)_{N \in \N}$ is a \emph{F{\o}lner sequence} if for all $\varphi \in \Hom(G,G)$,
$\frac{m \left( (F_N + \varphi) \triangle F_N \right)}{m(F_N)} \to 0$,
where $m$ is the Haar measure on $\Hom(G,G)$.\footnote{Note that these definitions are consistent with the definitions of syndetic sets and F{\o}lner sequences given for countable discrete groups in the introduction.}
It is still true that a set is syndetic if and only if it intersects every F{\o}lner sequence
(see Lemma \ref{lem: syndetic thick}),
so uniform Ces\`{a}ro limits remain useful for proving syndeticity in this more general context.

The key tool for proving Theorem \ref{thm: cubic Khintchine} is a ``Fubini'' theorem for (locally compact) amenable groups proven in \cite{bl}:

\begin{lem}[\cite{bl}, Lemma 1.1] \label{lem: Fubini}
	Let $G, H$ be amenable groups, and let $(h,g) \mapsto v_{h,g}$ be a bounded continuous map
	from $H \times G$ to a Banach space $V$.
	Assume that $\UClim_{(h,g) \in H \times G}{v_{h,g}}$ exists and for every $g \in G$,
	$\UClim_{h \in H}{v_{h,g}}$ exists.
	Then
	\begin{align*}
		\UClim_{(h,g) \in H \times G}{v_{h,g}} = \UClim_{g \in G}{\UClim_{h \in H}{v_{h,g}}}.
	\end{align*}
\end{lem}

\begin{proof}[Proof of Theorem \ref{thm: cubic Khintchine}]
	First we will apply Lemma \ref{lem: Fubini} to show that
	\begin{align} \label{eq: cubic avg}
		\UClim_{(\varphi_1, \dots, \varphi_k; g)}{\mu \left( \bigcap_{\eps \in \{0,1\}^k}{
				T_{(\eps_1\varphi_1 + \eps_2\varphi_2 + \cdots + \eps_k\varphi_k)(g)}^{-1}A} \right)}
		\ge \mu(A)^{2^k}.
	\end{align}
	This implies that the set of $(\varphi_1, \dots, \varphi_k; g)$ with large intersections is syndetic.
	We then apply Lemma \ref{lem: Fubini} again to obtain a syndetic slice $S \subseteq \Hom(G,G)^k$
	for which syndetically many $g \in G$ produce large intersections.
	
	\bigskip
	
	We know that for every $g \in G$
	\begin{align*}
		\UClim_{\varphi_1 \in \Hom(G,G)}{\mu(A \cap T_{\varphi_1(g)}^{-1}A)} \ge \mu(A)^2.
	\end{align*}
	This is the content of (a general version of) Khintchine's recurrence theorem
	for the action of $\Hom(G,G)$ given by $S_{\varphi} := T_{\varphi(g)}$,
	and it follows immediately from the ergodic theorem and Cauchy--Schwarz.
	Applying Lemma \ref{lem: Fubini}, we therefore have
	\begin{align*}
		\UClim_{(\varphi_1; g)}{\mu(A \cap T_{\varphi_1(g)}^{-1}A)}
		& = \UClim_{g}{\UClim_{\varphi_1}{\mu(A \cap T_{\varphi_1(g)}^{-1}A)}} \\
		& \ge \UClim_{g}{\mu(A)^2} \\
		& = \mu(A)^2,
	\end{align*}
	which proves the $k = 1$ case of \eqref{eq: cubic avg}.
	
	Now suppose \eqref{eq: cubic avg} holds for some $k \ge 1$.
	For notational convenience, let
	\begin{align*}
		A_{\varphi_1, \dots, \varphi_k; g} := \bigcap_{\eps \in \{0,1\}^k}{
			T_{(\eps_1\varphi_1 + \eps_2\varphi_2 + \cdots + \eps_k\varphi_k)(g)}^{-1}A}
	\end{align*}
	for $(\varphi_1, \dots, \varphi_k; g)$.
	In this notation, the induction hypothesis says
	\begin{align*}
		\UClim_{(\varphi_1, \dots, \varphi_k; g)}{\mu \left( A_{\varphi_1, \dots, \varphi_k; g} \right)}
		\ge \mu(A)^{2^k}.
	\end{align*}
	Applying Lemma \ref{lem: Fubini}, we have
	\begin{align*}
		& \UClim_{(\varphi_1, \dots, \varphi_k, \varphi_{k+1}; g)}{
			\mu \left( \bigcap_{\eps \in \{0,1\}^{k+1}}{
				T_{(\eps_1\varphi_1 + \eps_2\varphi_2 + \cdots
					+ \eps_k\varphi_k + \eps_{k+1}\varphi_{k+1})(g)}^{-1}A} \right)} \\
		& = \UClim_{(\varphi_1, \dots, \varphi_k; g)}{\UClim_{\varphi_{k+1}}{
				\mu \left( A_{\varphi_1, \dots, \varphi_k; g}
				\cap T_{\varphi_{k+1}(g)}^{-1} A_{\varphi_1, \dots, \varphi_k; g} \right)}} \\
		& \ge \UClim_{(\varphi_1, \dots, \varphi_k; g)}{\mu\left( A_{\varphi_1, \dots, \varphi_k; g} \right)^2} \\
		& \ge \left( \UClim_{(\varphi_1, \dots, \varphi_k; g)}
		{\mu \left( A_{\varphi_1, \dots, \varphi_k; g} \right)} \right)^2 \\
		& \ge \left( \mu(A)^{2^k} \right)^2 \\
		& = \mu(A)^{2^{k+1}}.
	\end{align*}
	Thus, \eqref{eq: cubic avg} holds by induction.
	
	Now, we can apply Lemma \ref{lem: Fubini} one final time to obtain the inequality
	\begin{align*}
		\UClim_{(\varphi_1, \dots, \varphi_k)}{\UClim_{g}{
				\mu \left( \bigcap_{\eps \in \{0,1\}^k}{
					T_{(\eps_1\varphi_1 + \eps_2\varphi_2 + \cdots + \eps_k\varphi_k)(g)}^{-1}A} \right)}}
		\ge \mu(A)^{2^k}.
	\end{align*}
	It follows that the set
	\begin{align*}
		S := \left\{ (\varphi_1, \dots, \varphi_k) \in \Hom(G,G)^k :
		R_{\eps}(\varphi_1, \dots, \varphi_k)~\text{is syndetic in}~G \right\}
	\end{align*}
	is syndetic in $\Hom(G,G)^k$.
\end{proof}


\section*{Acknowledgments} 
We thank Or Shalom for detecting an erroneous formula in an earlier version of this paper. This necessitated significant revisions, which are reflected in Sections \ref{sec: limit formula 3} and \ref{sec: tK}.
We also thank an anonymous referee for helpful comments and drawing our attention to the paper \cite{dlms}.

\bibliographystyle{amsplain}

\begin{thebibliography}{0} 
	
	\bibitem[AjSz74]{as}
	M.~Ajtai and E.~Szemer\'{e}di.
	\newblock Sets of lattice points that form no squares.
	\newblock {\em Studia Sci. Math. Hungar.}, 9:9--11 (1975), 1974.
	
	\bibitem[Au16]{austin}
	T.~Austin.
	\newblock Non-conventional ergodic averages for several commuting actions of an amenable group.
	\newblock {\em J. Anal. Math.}, 130:243--274, 2016.
	
	\bibitem[Beh46]{behrend}
	F.~A. Behrend.
	\newblock On sets of integers which contain no three terms in arithmetical
	progression.
	\newblock {\em Proc. Nat. Acad. Sci. U.S.A.}, 32:331--332, 1946.
	
	\bibitem[B00]{et/dp}
	V.~Bergelson.
	\newblock Ergodic theory and {D}iophantine problems.
	\newblock In {\em Topics in Symbolic Dynamics and Applications ({T}emuco,
		1997)}, volume 279 of {\em London Math. Soc. Lecture Note Ser.}, pages
	167--205. Cambridge Univ. Press, Cambridge, 2000.
	
	\bibitem[BFe]{bfm}
	V.~Bergelson and A.~Ferr\'{e} Moragues.
	\newblock An ergodic correspondence principle, invariant means and applications.
	\newblock {\em Israel J. Math.}, to appear.
	\newblock arXiv:2003:03029.
	
	\bibitem[BHKr05]{bhk}
	V.~Bergelson, B.~Host, and B.~Kra.
	\newblock Multiple recurrence and nilsequences.
	\newblock {\em Invent. Math.}, 160(2):261--303, 2005.
	\newblock With an appendix by Imre Ruzsa.
	
	\bibitem[BLei15]{bl}
	V.~Bergelson and A.~Leibman.
	\newblock Cubic averages and large intersections.
	\newblock In {\em Recent Trends in Ergodic Theory and Dynamical Systems},
	volume 631 of {\em Contemp. Math.}, pages 5--19. Amer. Math. Soc.,
	Providence, RI, 2015.
	
	\bibitem[BMc07]{bm-roth}
	V.~Bergelson and R.~McCutcheon.
	\newblock Central sets and a non-commutative {R}oth theorem.
	\newblock {\em Amer. J. Math.}, 129(5):1251--1275, 2007.
	
	\bibitem[BTZ10]{btz1}
	V.~Bergelson, T.~Tao, and T.~Ziegler.
	\newblock An inverse theorem for the uniformity seminorms associated with the
	action of {$\mathbb{F}_p^\infty$}.
	\newblock {\em Geom. Funct. Anal.}, 19(6):1539--1596, 2010.
	
	\bibitem[BTZ15]{btz2}
	V.~Bergelson, T.~Tao, and T.~Ziegler.
	\newblock Multiple recurrence and convergence results associated to
	{$\mathbb{F}_p^\omega$}-actions.
	\newblock {\em J. Anal. Math.}, 127:329--378, 2015.
	
	\bibitem[BerSahSawTi]{bsst}
	A.~Berger, A.~Sah, M.~Sawhney, and J.~Tidor.
	\newblock Popular differences for matrix patterns.
	\newblock arXiv:2102.01684.
	
	\bibitem[Chu11]{chu}
	Q.~Chu.
	\newblock Multiple recurrence for two commuting transformations.
	\newblock {\em Ergodic Theory Dynam. Systems}, 31(3):771--792, 2011.
	
	\bibitem[CoLes84]{cl}
	J.-P.~Conze and E.~Lesigne.
	\newblock Th\'{e}or\`{e}mes ergodiques pour des mesures diagonales.
	\newblock {\em Bull. Soc. Math. France}, 112(2):143--175, 1984.
	
	\bibitem[DLeMSu21]{dlms}
	S.~Donoso, A.~Le, J.~Moreira, and W.~Sun.
	\newblock   Optimal lower bounds for multiple recurrence.
	\newblock   {\em Ergodic Theory Dynam. Systems}, 41:379--407, 2021.
	
	\bibitem[DSu18]{ds}
	S.~Donoso and W.~Sun.
	\newblock Quantitative multiple recurrence for two and three transformations.
	\newblock {\em Israel J. Math.}, 226(1):71--85, 2018.
	
	\bibitem[Fr08]{frathree}
	N.~Frantzikinakis.
	\newblock Multiple ergodic averages for three polynomials and applications.
	\newblock {\em Trans. Amer. Math. Soc.}, 360(10):5435--5475, 2008.
	
	\bibitem[Fu77]{diag}
	H.~Furstenberg.
	\newblock Ergodic behavior of diagonal measures and a theorem of
	{S}zemer\'{e}di on arithmetic progressions.
	\newblock {\em J. Analyse Math.}, 31:204--256, 1977.
	
	\bibitem[Fu81]{fursbook}
	H.~Furstenberg.
	\newblock {\em Recurrence in ergodic theory and combinatorial number theory}.
	\newblock Princeton University Press, Princeton, N.J., 1981.
	\newblock M. B. Porter Lectures.
	
	\bibitem[FuKa85]{fk-IP}
	H.~Furstenberg and Y.~Katznelson.
	\newblock An ergodic {S}zemer\'{e}di theorem for {IP}-systems and combinatorial
	theory.
	\newblock {\em J. Analyse Math.}, 45:117--168, 1985.
	
	\bibitem[FuWe96]{fw}
	H.~Furstenberg and B.~Weiss.
	\newblock A mean ergodic theorem for $\frac{1}{N}
	\sum_{n=1}^n{f(T^nx)g(T^{n^2}x)}$.
	\newblock In {\em Convergence in Ergodic Theory and Probability ({C}olumbus,
		{OH}, 1993)}, volume~5 of {\em Ohio State Univ. Math. Res. Inst. Publ.},
	pages 193--227. De Gruyter, Berlin, 1996.
	
	\bibitem[Gl03]{glasner}
	E.~Glasner.
	\newblock {\em Ergodic Theory via Joinings}, volume 101 of {\em Mathematical
		Surveys and Monographs}.
	\newblock American Mathematical Society, Providence, RI, 2003.
	
	\bibitem[Gre05]{green}
	B.~Green.
	\newblock A {S}zemer\'{e}di-type regularity lemma in abelian groups, with
	applications.
	\newblock {\em Geom. Funct. Anal.}, 15(2):340--376, 2005.
	
	\bibitem[Gre]{green-corners}
	B.~Green.
	\newblock Lower bounds for corner-free sets.
	\newblock {\em New Zealand J. Math.}, to appear. arXiv:2102.11702.
	
	\bibitem[GreT10]{gt}
	B.~Green and T.~Tao.
	\newblock An arithmetic regularity lemma, an associated counting lemma, and
	applications.
	\newblock In {\em An irregular mind}, volume~21 of {\em Bolyai Soc. Math.
		Stud.}, pages 261--334. J\'{a}nos Bolyai Math. Soc., Budapest, 2010.
	
	\bibitem[Gri09]{griesmer}
	J.~T. Griesmer.
	\newblock {\em Ergodic averages, correlation sequences, and sumsets}.
	\newblock ProQuest LLC, Ann Arbor, MI, 2009.
	\newblock Thesis (Ph.D.)--The Ohio State University.
	
	\bibitem[HKr02]{hk}
	B.~Host and B.~Kra.
	\newblock An odd {F}urstenberg-{S}zemer\'{e}di theorem and quasi-affine systems.
	\newblock {\em J. Anal. Math.}, 86:183--220, 2002.
	
	\bibitem[Kh35]{khintchine}
	A.~Khintchine.
	\newblock Eine {V}ersch\"{a}rfung des {P}oincar\'{e}schen
	``{W}iederkehrsatzes''.
	\newblock {\em Compositio Math.}, 1:177--179, 1935.
	
	\bibitem[Ko]{kovac}
	V.~Kova\v{c}.
	\newblock Popular difference for right isosceles triangles.
	\newblock arXiv:2101.12714.
	
	\bibitem[Les93]{lesigne}
	E.~Lesigne.
	\newblock \'{E}quations fonctionnelles, couplages de produits gauches et th\'{e}or\`emes ergodiques pour mesures diagonales.
	\newblock {\em Bull. Soc. Math. France}, 121(3):315--351, 1993.
	
	\bibitem[LiShr]{ls}
	N.~Linial and A.~Shraibman.
	\newblock Larger corner-free sets from better NOF exactly-$N$ protocols.
	\newblock Discrete Analysis 2021:19, 9 pp.
	
	\bibitem[R90]{rudin}
	W.~Rudin.
	\newblock {\em Fourier Analysis on Groups}.
	\newblock Wiley Classics Library.
	\newblock John Wiley \& Sons, Inc., New York, 1990.
	
	\bibitem[SahSawZh]{ssz}
	A.~Sah, M.~Sawhney, and Y.~Zhao.
	\newblock Patterns without a popular difference.
	\newblock Discrete Analysis 2021:8, 30 pp.
	
	\bibitem[Sha]{shalom}
	O.~Shalom.
	\newblock Multiple ergodic averages in abelian groups and Khintchine type recurrence.
	\newblock arXiv:2102.07273.
	
	\bibitem[V63]{vara}
	V.~S. Varadarajan.
	\newblock Groups of automorphisms of {B}orel spaces.
	\newblock {\em Trans. Amer. Math. Soc.}, 109:191--220, 1963.
	
	\bibitem[Wa82]{walters}
	P.~Walters.
	\newblock {\em An Introduction to Ergodic Theory}, volume~79 of {\em Graduate
		Texts in Mathematics}.
	\newblock Springer-Verlag, New York-Berlin, 1982.
	
	\bibitem[Z07]{ziegler}
	T.~Ziegler.
	\newblock Universal characteristic factors and {F}urstenberg averages.
	\newblock {\em J. Amer. Math. Soc.}, 20(1):53--97, 2007.
	
	\bibitem[Z-K16]{zorin}
	P.~Zorin-Kranich.
	\newblock Norm convergence of multiple ergodic averages on amenable groups.
	\newblock {\em J. Anal. Math.}, 130:219--241, 2016.
	
\end{thebibliography}




\begin{dajauthors}
\begin{authorinfo}[ack]
  Ethan Ackelsberg\\
  Ohio State University\\
  Columbus, Ohio, USA\\
  ackelsberg\imagedot{}1\imageat{}buckeyemail\imagedot{}osu\imagedot{}edu
\end{authorinfo}
\begin{authorinfo}[ber]
	Vitaly Bergelson\\
    Ohio State University\\
  Columbus, Ohio, USA\\
  vitaly\imageat{}math\imagedot{}ohio-state\imagedot{}edu
\end{authorinfo}
\begin{authorinfo}[bes]
	Andrew Best\\
    Ohio State University\\
  Columbus, Ohio, USA\\
  best\imagedot{}221\imageat{}buckeyemail\imagedot{}osu\imagedot{}edu
\end{authorinfo}
\end{dajauthors}

\end{document}